\documentclass[reqno,twoside,a4paper,11pt]{amsart}
\usepackage{amsmath,amssymb,dsfont,amsthm}
\usepackage{graphics}
\usepackage{xypic}
\usepackage{latexsym}
\usepackage{amsfonts}
\usepackage{geometry}
\usepackage{tensor}
\usepackage{bm,color}
\usepackage{relsize}
\usepackage{hyperref}
\usepackage{upgreek}

\hypersetup{
    colorlinks=true,
    linkcolor=blue,
    filecolor=magenta,      
    urlcolor=cyan,
    citecolor=blue,
    pdftitle={Sharelatex Example},
    bookmarks=true,
    pdfpagemode=FullScreen,
}

\date{\today}

\title[Hodge de Rham theory and analytic torsion for spaces with horn type singularities]{Hodge de Rham theory and analytic torsion for spaces with horn type singularities}

\author{M. Spreafico}
\address[Mauro Spreafico]{\tt Dipartimento di matematica ed applicazioni, Universit\'a Milano Bicocca, Milano, and INFN Lecce, Italy}
\email{mauro.spreafico@unimib.it}

\numberwithin{equation}{section}

\newtheorem{theo}[subsubsection]{Theorem}
\newtheorem{lem}[subsubsection]{Lemma}
\newtheorem{corol}[subsubsection]{Corollary}
\newtheorem{defi}[subsubsection]{Definition}

\newtheorem{prop}[subsubsection]{Proposition}
\newtheorem{rem}[subsubsection]{Remark}

\newcommand{\beq}{\begin{equation}}
\newcommand{\eeq}{\end{equation}}

\newcommand{\te}{\theta}
\newcommand{\de}{{\delta}}
\newcommand{\vv}{{\varphi}}
\newcommand{\ep}{\epsilon}
\newcommand{\al}{\alpha}
\newcommand{\be}{\beta}

\newcommand{\ka}{\kappa}
\newcommand{\la}{\lambda}
\newcommand{\ga}{\gamma}

\newcommand{\om}{\omega}

\renewcommand{\Re}{{\rm Re}}
\renewcommand{\Im}{{\rm Im\hspace{2pt}}}

\renewcommand{\b}{{\partial}}

\newcommand{\rk}{{\rm rk}}
\newcommand{\da}{{\dagger}}
\renewcommand{\det}{{\rm det\hspace{1pt}}}

\newcommand{\bu}{{\bullet}}
\newcommand{\Sp}{{\rm Sp}}

\renewcommand{\d}{{\rm d}}

\newcommand{\Det}{{\rm Det\hspace{1pt}}}
\newcommand{\Tr}{{\rm Tr\hspace{.5pt} }}
\renewcommand{\t}{\tilde}
\newcommand{\lp}{\langle}
\newcommand{\rp}{\rangle}
\newcommand{\lk}{\rm lk}

\newcommand{\Z}{{\mathds{Z}}}
\newcommand{\R}{{\mathds{R}}}
\newcommand{\C}{{\mathds{C}}}

\newcommand\sh{{\rm sh}}
\newcommand\ch{{\rm ch}}

\newcommand\e{{\rm e}}

\newcommand{\es}{\mathsf{e}}
\newcommand{\gs}{\mathsf{g}}

\newcommand{\ns}{\mathsf{n}}
\newcommand{\alphas}{\mathsf{\alpha}}

\newcommand{\g}{{\mathsf{g}}}

\newcommand{\uf}{\mathfrak{u}}
\newcommand{\ff}{\mathfrak{f}}

\newcommand{\pf}{{\mathfrak{p}}}
\newcommand{\mf}{{\mathfrak{m}}}

\newcommand{\af}{{\mathfrak a}}

\newcommand{\dr}{\updelta}
\newcommand{\is}{\mathsf{i}}

\newcommand{\LL}{\mathcal{L}}
\newcommand{\CF}{\mathfrak{C}}
\newcommand{\D}{\mathsf{D}}
\newcommand{\DD}{\mathcal{D}}

\renewcommand{\H}{\mathcal{H}}

\newcommand{\FF}{\mathcal{F}}

\newcommand{\DF}{\mathfrak{D}}

\newcommand{\CS}{\mathsf{C}}
\newcommand{\DS}{\mathsf{D}}

\newcommand{\E}{\mathcal{E}}

\newcommand{\QQ}{\mathcal{Q}}

\newcommand{\Ha}{{\mathcal{H}}}
\newcommand{\A}{{\mathcal{A}}}

\newcommand{\UU}{\mathcal{U}}
\newcommand{\TT}{\mathcal{T}}
\newcommand{\Ei}{\mathcal Eig}

\DeclareMathOperator*{\Rz}{Res_0}
\DeclareMathOperator*{\Ru}{Res_1}

\newcommand{\ZZ}{{Z^\ka_{l}(W)}}

\begin{document}

\maketitle
\begin{abstract} We study the global analytic properties of a space $X$ with a horn type singularity. In particular, we introduce some de Rham complex of square integrable forms and we describe its homology and the spectral properties of the associated Hodge Laplace operator. All this is applied to produce a suitable description of the analytic torsion of $X$ and to prove an extension of the Cheeger M\"{u}ller theorem. 
\end{abstract}

\tableofcontents

\section{Introduction}
\label{intro}

If $X$ is a compact connected metric space such that, for some point $x_0$ in $X$, $X-\{x_0\}$ is  a 
dense incomplete smooth Riemannian manifold, we call $X$ a space with an isolated singularity. From the topological point of view, such a space is just the attachment of a cone to a manifold with boundary, but from the geometric point of view the singularity is characterised by the metric near the singular point. The first natural type of singularity is the conical one, where the metric reads locally $dx\otimes dx +x^2 H^2(x) \t g$, where $x$ is the distance from the singular point,  $\t g$ the metric on the section of the cone, and $H$ a smooth positive function. The global geometric and analytic properties of these spaces have been originally investigated by J. Cheeger \cite{Che0} \cite{Che1} \cite{Che2} (see also \cite{BS1} \cite{BS2}). Since then, a considerable number of works by several authors appeared. In particular, a good effort has been spent to find  a suitable extension of the Cheeger M\"{u}ller CM Theorem on these spaces \cite{HS1,HS4,HS5,MV, Lud, HS6}. The next natural type of singularity is that of an horn, where the metric reads locally $dx\otimes dx +x^{2\ka} H^2(x) \t g$, with $\ka>1$. Again, these spaces were originally considered by J. Cheeger \cite[Section 2]{Che3}. However, they have been much less intensively studied.  In \cite{LP}, the closed extensions of the exterior derivative operator are classified and the index problem is discussed (see also \cite{Lap}).

The aim of these notes is to study the global analytic properties of a space with an horn singularity, in particular pointing  to a suitable extension of the CM Theorem. Following the classical approach to global analysis on open manifolds   \cite{CY, Gaf1, Gaf2}, we decompose the  space $X$ 
into a model finite horn $Z$ and a smooth manifold with boundary $Y$. 
Then, we may deal with the analysis on the finite horn by decomposing the problem into a Sturm Liouville problem on the finite line $(0,1)$ times a classical problem on the section, that is a smooth manifold. A key point here is that in the case of an horn, the relevant SL problems are irregular singular problems. While  regular and regular singular SL problems have been intensively studied, only few works are available on  irregular singular problems. We dedicate Appendix \ref{SL} to develop all the result necessary for our purpose. There is work in progress with G. Metafune and D. Pallara on irregular singular SL problems.  

Back to the present work, the decomposition described above allow us to introduce the suitable closed extensions of the exterior derivative operator and to give an explicit characterisation of it, Section \ref{max}. Then, we   define  some de Rham complexes of Hilbert space and closed operators, Section \ref{drcomp}, and we compute their homology, Section \ref{DRcohom}. We find that the homology of some of these de Rham complexes coincides  with the intersection homology of Goresky and MacPhearson \cite{GM1, GM2}, that we recall in Section \ref{DRX}.  Then, we turn to the Hodge Laplace operator, Section \ref{hodge}. On one side we study the Hodge Laplace operator associated to the de Rham complexes above, on the other we define some suitable self adjoint extensions of the formal Hodge Laplace operator. We prove the equivalence of the two approaches, and we describe the main spectral properties of this operator. In particular, we have an explicit spectral resolution on the finite horn, and we  identify the harmonic fields of $X$ in terms of harmonic fields of $Y$, Proposition \ref{harmonicsX}. This permits to  construct explicit de Rham maps inducing the isomorphism in homology, Section \ref{DRX}. 

A remark on this result is in order. From the topological point of view, the obstruction for a space $X$ with an isolated singularity to be a manifold lies in the homology of the link of the singular point: if this is a sphere, then $X$ is an homology sphere, and therefore a manifold, otherwise it is not. A consequence is that duality does not hold for the classical homology of $X$. The original idea of Goresky and MacPherson was to introduce a suitable homology theory where duality held. The key technical point is to select the cells (in some combinatorial decomposition of $X$) that intersect the singular point. It is clear enough that eliminating all such cells one would end up in removing the cone $C(\lk(x))$ over the link of the singular point $x$, while on the other side, allowing all the cells intersecting $x$, one would end up with the homology of the pair $(X-C(\lk (x)), \lk(x))$, see Section \ref{DRX} and \cite{Spr12}. 
However, since the $q$ homology does not follow by the chain module in degree $q$, if we fix some degree $q_0$, and we define a cell complex removing some cells for $q\leq q_0$ and allowing  all cells for $q>q_0$, we end up with  the homology of the manifold $X-C(\lk (x))$, for $q<q_0$, and  with the homology of the pair $(X-C(\lk (x)), \lk(x))$, for $q>q_0$, but the homology in degree $q_0$ will depend on the interaction between the relevant homology exact sequences, see Propositions \ref{pop1} and \ref{pop2}.  This means that, using the duality for the manifold with boundary $X-C(\lk (x))$, we recover duality for $X$, when the dimension of $X$ is even $m=2k$, but we do not when dimension is odd. In such a case, to guarantee duality, we need to produce two different cell complexes, $C$ and $C'$, essentially with $q_0'=q_0+1$ (these correspond to the two middle perversites: $\mf$ and $\mf^c$), and duality will read $H_q(C)\cong H_{2k+1-q}(C')$, see the end of Section \ref{DRX}. 
What is the analytic counterpart of this? This is indeed a major problem, as witnessed by the fact that in the literature it has been introduced ad ad hoc condition, the Witt condition, that means in this case the vanishing of the homology of the link of the singularity in middle degree. Our solution to this problem is to introduce some boundary conditions at the singular point for the relevant Hodge Laplace operator. For note that the relevant SL problem reads
\[
-h^{2\al-1}\left(h^{1-2\al}f'\right)' +\frac{\t\la}{h^2} f=\la f,
\]
where $h(x)=x^{2\ka}H(x)$, and $\al$ is a dimensional parameter. This is in general an irregular singular problem, but in the case of the harmonics, $\la=\t\la=0$, and therefore  it becomes a regular problem when $\al=\frac{1}{2}$. This may happen only in odd dimension and in degrees $q_0$ and $q_0+1$. Thus, in these cases, we have different possible self adjoint extensions, and in particular the two classical ones (relative and absolute, or Dirichlet and Newman). Choosing these two extensions we have two different de Rham complexes whose associated Hodge Laplace operators have different harmonic fields in particular in the critical degrees. In such a way we guarantee the isomorphism with the intersection homology described above in all degrees and without any further assumption. 

Our last results concern analytic torsion and the Cheeger M\"{u}ller Theorem \cite{RS, Che0, Mul1}. Using the spectrum of the Hodge Laplace operator on the finite horn we express the torsion zeta function as the sum of some double and some simple Dirichlet series, Section \ref{torhorn}. We deal with the double series using the spectral decomposition lemma in \cite{Spr9}. Here the explicit knowledge of the analytic properties of some spectral functions   and of the asymptotic expansions of the solutions of the associated irregular singular SL problems is determinant. All this information is in Appendix \ref{SL}. Using this result, we prove a formula for the analytic torsion of the finite horn, Theorem \ref{t1}. Then, extending the result of Ray and Singer \cite{RS} on the variation of the analytic torsion under the variation of the metric, and using the glueing formula of Vishick Lesh \cite{Vis, Les2}, we have the analytic torsion of  $X$ \ref{tX}. The last result is to compare the analytic torsion with the combinatorial intersection RS torsion (see Section \ref{DRX} for the definition). The result is in Theorems \ref{tcm} and \ref{tcmX}, where we see that if the dimension of $X$ is even, then the two torsion coincide, while two anomaly terms appears in odd dimension. 

We add some remarks on the CM formula in Theorem \ref{tcmX}. The first anomaly term $A_{\rm comb, \pf}$ is purely combinatoric and may be explained as follows.  The Euler isomorphism maps the determinant line of the homology to the determinant line of the chain complex (compare \cite{Spr20}). The combinatorial part of the torsion, i.e. the Reidemeister torsion, is the coefficient of the image of a basis for the determinant homology line in the cell  basis of  the determinant chain line. In the case of the intersection chain complex, the cells do not determine a preferred basis for the chain line. A preferred basis may be fixed by using the natural standard integral basis. The anomaly $A_{\rm comb, \pf}$ measures precisely this change of basis. 

The second anomaly term $A_{\rm analy}$ is more involved. This term clearly comes from the analytic side, i.e. is a term in the explicit calculation of the analytic torsion that has not a counter part in the combinatorial torsion. Even if its natural meaning appears still quite obscure, we have now the possibility of comparing it in at least three different cases, and this gives us some clarifications on its character. The three cases are: regular boundary (direct straightforward calculation), conical singularity and horn singularity. The difference traduces in the different shape of the function $h$ appearing in the relevant SL problem. 
In the regular case, $h$ is a constant near $x=0$. This traduces in regular SL problems, and the anomaly reduces to 
$
A_{\rm analy}(W)=\frac{1}{4}\chi(W)\log 2.
$ 
This is the classical anomaly term originally detected by L\"{u}ck \cite{Luc}. Passing to the next case, the conical one, we have \cite[7.12.2]{HS6}
\beq\label{PL}
\begin{aligned}
A_{\rm analy}(W)=&\frac{1}{4}\chi(W)\log 2+\sum_{q=0}^{p-1}(-1)^{q+1} r_q \log (2p-2q-1)!!\\
&+\frac{1}{2}\sum_{q=0}^{p-1}(-1)^{q} \sum_{n=1}^{\infty} m_{{\rm cex},q,n} \log \frac{\mu_{q,n}+\alpha_{q}}{\mu_{q,n}-\alpha_{q}}.
\end{aligned}
\eeq
where $r_k=\rk H_q(W)$, $\al_q=q-\frac{1}{2}(1-2p)$, $2p=\dim W$, $\mu_{q,n}=\sqrt{\t\la_{q,n}+\al^2_q}$, and $\t\la_{q,n}$ are the eigenvalues of the Hodge Laplace operator on the section $W$. The first term is the same as above, but we observe that the second term too is a genuine combinatorial/homological one: in fact, as the Euler characteristic, it is the graded sum weighted by the rank of the homology, of the logarithm of some integer numbers (taking the constant 2 instead we have the Euler characteristic). The last term  has not a clear  meaning, but the  analysis below will give us at least some ideas for its appearance.  

The last case 
is that of the horn. The result is in Theorem \ref{tcm}. 
This tells us two things. First, the analytic anomaly term becomes now purely combinatorial, and exactly a generalisation of the second term of the conical case, for we have: 
\begin{align*}
A_{\rm analy}(W)(\ka)=&\frac{1}{4}\chi(W)\log 2+\frac{1}{2}\sum_{q=0}^{p-1}(-1)^{q+1} r_q \log \frac{2^{\ka(1-2\al_{q})}}{\pi} \Gamma^2\left(\frac{1}{2}-\ka\left(\al_q-\frac{1}{2}\right)\right), 
\end{align*}
and $A_{\rm analy}(W)(1)=\frac{1}{4}\chi(W)\log 2+\sum_{q=0}^{p-1}(-1)^{q+1} r_q \log (2p-2q-1)!!$. About this point also observe that this terms  reduces to the logarithm of an integer number whenever the parameter $\ka$ is itself an integer. 
These considerations permit to write the CM theorem in a more fascinating (but equivalent) way. For both the anomaly terms are combinatoric (since the representation is trivial) and graded, so we can use them to normalising the homology basis. Considering instead of $I^\pf \tau_{\rm RS}(\ZZ)$ the torsion in this renormalised graded homology basis, say $I^\pf \tau_{\rm \widehat{RS}}(\ZZ)$, we would have 
\[
\log T_{\pf, \rm abs}(\ZZ)=\log I^\pf \tau_{\rm \widehat {RS}}(\ZZ)+A_{\rm BM, abs}(W),
\] 
in all dimensions. The second thing we learn concerns the third anomaly term appearing in equation (\ref{PL}). This term is not homological, since it involves the positive eigenvalues of the section, and it vanishes in the case of the horns. We suggest as a possible motivation for the appearance of this term that it is due to the fact that in the case of the cone the degree of the singularity coincides with the degree of the operator. In such a case in fact, the behaviour of the solutions of the relevant SL problem near the singular point is determined by the solutions of the indicial equation, that depends on both the spectral parameter on the section $\t\la_{q,n}$,  and the dimensional parameter $\al_q$. Otherwise, namely when the degree of the singularity is different from the degree of the operator, then the behaviour of the solutions depends only on spectral parameter. Further investigation on this point is in progress.

\vspace{10pt}

\section{Spaces with horns type singularities}
\label{horn}

\subsection{Underlying geometry}
\label{geo}

Let $W$ be an orientable  compact connected Riemannian manifold of finite dimension $m$ without boundary  and with Riemannian structure $\t g$. Let $Z_l(W)=(0,l]\times W$ be the product manifold, where $(0,l]$ is a positive interval of the real line, $l>0$. This is an orientable open connected smooth manifold of dimension $n=m+1$ with boundary $\b Z_l(W)=\{l\}\times W$.  
Let $x$ be the natural global coordinate on $(0,l]$, and $h(x)=x^\ka H(x)$, with  $H$  a smooth non vanishing  function    on $[0,l]$, with $H(0)=1$, and $\ka$ a real number, with $\ka>1$. Then 
\beq\label{g1}
g=dx\otimes dx+h^2(x) \t g=d x\otimes dx +x^{2\ka} H^2(x) \t g,
\eeq
is a Riemannian metric on $\ZZ$. We denote by $\ZZ$ the manifold $Z_l(W)$ with the metric in equation (\ref{g1}), and we call it the open finite horn over $W$. We denote by $\overline{\ZZ}$ the space $\ZZ\cup\{0\}$, and we call it the completed finite metric horn over $W$. Note that $\overline{\ZZ}$ is homeomorphic to the cone $C(W)$. The boundary  of $\ZZ$ is of course diffeomorphic to $W$, and isometric to $(W,h(l)^2g)$.  
If $y$ is a local coordinate system on $W$, then $(x,y)$ is a local coordinate system on $\ZZ$. Following common notation, we call $W$ the section of $\ZZ$. Also following usual notation, a tilde will denote operations on the section, 
and not on the boundary. 

Next, let $(Y,W)$ be a compact connected orientable smooth Riemannian   manifold of dimension $n$ with boundary $\b Y$. Assume that $\b Y$ has two disjoint component $(\b Y)_1$ and $(\b Y)_2$, with $(\b Y)_1$  isometric to $(W,h^2(l) \t g)$,  for some $l>0$. Consider the space
$
X=\ZZ\cup_W Y,
$ 
where the glueing is smooth: $X$ is an open oriented connected smooth manifold with boundary $\b X=(\b Y)_2$. Then, also the metric $g_Y$ glues smoothly to the metric $g$ of $\ZZ$, for note that we have not requirements on the function $H$ near the boundary of $\ZZ$ (beside obviously that it does not vanish). We will denote the global metric on $X$ by $g_X$, and its restrictions on $Y$ and $\ZZ$ by $g_Y$ and $g_Z$, respectively. $X$ is an open connected orientable Riemannian manifold, that we call a space with an isolated metric horn singularity. We call the compact connected space $X\cup \{0\}$ the completion of $X$ and we denote it by $\overline{X}$. Observe that from the topological point of view the interesting object is $\overline X$, since $X$ is homeomorphic to $Y-\b Y$, and has the same simple homotopy type of $Y$.

Let $\iota=\iota(Y)$ be the injectivity radius of $Y$ \cite[2.7]{Gil}. Then, we have the collar $\CF=(0,l+\iota)\times W$ of $W$ inside $X$, and on it we may use  the local system of coordinates $(x,y)$ described above. Thus on $\CF$, and obviously on $\ZZ\subseteq \CF$ the decomposition and all the results given on $\ZZ$ hold.


Let $\pi_1(\overline{X})$ denote the fundamental group of $\overline{X}$. Let $\rho:\pi_1(\overline{X})\to O(E)$ be an orthogonal representation in some $k$ dimensional real vector space $E$. Let $\widetilde X$ denote the universal covering space of $\overline{X}$, and $E_\rho=E\otimes_\rho \widetilde X$ the associated vector bundle. Observe that this construction makes sense since $\overline{X}$ and $Y/\b Y$ have the same homotopy type. Therefore, any representation $\rho:\pi_1(\overline{X})\to O(V)$ restricts to the trivial representation $\rho_0:\pi_1(\overline{C_{0,l}(W)})\to O(E)$. Thus, the universal covering space of $\overline{X}$ is the universal covering manifold of $Y$ with a cone over the boundary of each sheet (compare \cite[6.9, 6.10, 7.16]{HS6}).  We will work with  forms on $X$ with coefficients in $E_\rho$: $\Omega^q(X,E_\rho)=\Gamma(X,\Lambda^q T^* E_\rho)=\Gamma(X,E_\rho\otimes  \Lambda^q T^* X)$.  In the following, if possible, we will omit explicit reference to the representation in the notation.

\subsection{Square integrable forms}

We assume that the vector space $E$ has an inner product, that will be omitted in the notation. 
Let $\gamma (\mathring X,E_\rho \otimes\Lambda^q T^*\mathring X)$ be the vector space of the  sections of the exterior algebra bundle with coefficients in $E_\rho$.    
On $\gamma (\mathring X,E_\rho \otimes\Lambda^q T^* \mathring X)$ we may consider the inner product
\beq\label{innerprod}
\lp\om,\vv\rp=\int_X \om\wedge \star \vv,
\eeq
for $\om,\vv\in \gamma(\mathring X,E_\rho \otimes\Lambda^q T^*\mathring X)$, and the induced norm $\|\om\|^2=\lp \om, \om\rp$. We denote by $L^2  (X,E_\rho)=L^2\gamma(\mathring X,E_\rho \otimes\Lambda^q T^*\mathring X)$  the Hilbert   space of the   sections in $\gamma(\mathring X,E_\rho \otimes\Lambda^q T^*\mathring X)$ with finite norm. We denote by $\Omega_0^q(X,E_\rho)$ the  vector sub space of the sections in $ \Omega^q(X,E_\rho)$ with compact support in $X$. It is clear that $ \Omega_0^q(X,E_\rho)\leq  \Omega^q(X,E_\rho)$ is a vector sub space of $L^2  (X,E_\rho)$. Note that

\[
\langle \om,\vv\rangle=\int_{X} \om\wedge\star \vv=\int_{\ZZ} \om\wedge\star \vv+\int_{Y} \om\wedge\star \vv.
\]

We may extend the usual definition of distributions  defining distributions on $q$ forms: given $\om\in L^1_{loc}\gamma(\mathring{X},E_\rho \otimes\Lambda^q T^*\mathring{X})$ (in particular $\om \in L^2(\mathring{X},E_\rho)$), we define the distribution $T_\om$ acting on $\Omega^{n-q}_0(\mathring{X},E_\rho)$, by
\[
T_\om(\vv)=\int_{\mathring{X}} \om\wedge \star\vv=\lp \om, \vv\rp,
\]
for each $\vv\in \Omega^{n-q}_0(\mathring{X},E_\rho)$. 
Observe the following explicit form for the distribution $T_\om$ on the collar $\CF$:

\[
T_{\om|_\CF}(\vv|_\CF)=\lp \om,\vv\rp_{\CF}=\int_0^{l+\ep} (\rho h^{1-2\al_{q}} f_1 g_1)(x)  dx \lp \t\om_1,\t\vv_1\rp_W
+\int_0^{l+\ep} (\rho h^{1-2\al_{q-1}} f_2 g_2)(x) dx  \lp \t\om_2,\t\vv_2\rp_W.
\]

This formula and formal duality of exterior derivative operators $d$ and $\de$, permit to extend these operators to the space of square integrable forms \cite[3.4]{Spr12} \cite[pg. 405]{Schm}. We use the notation 
\[
D: L^1_{loc}\gamma(\mathring{X},E_\rho \otimes\Lambda^q T^*\mathring{X})\to L^1_{loc} (\mathring{X},E_\rho \otimes\Lambda^{q+1} T^*\mathring{X}),
\]
for the the (weak) exterior derivative of  locally square integrable forms 
and
$
\DD^q=(-1)^{(m(q+1)+q} \star D^{n-q}\star,
$ 
for its adjoint. We have the Green formulas
$\lp d\om,\vv\rp=\lp \om,\DD \vv\rp$, 
for all $\om \in \Omega^{q}_0(M, E_\rho)$, and $\vv\in L^1_{loc}\gamma(\mathring{X},E_\rho \otimes\Lambda^{n-q} T^*\mathring{X})$.
and
$\lp D\om,\vv\rp=\lp \om,\de \vv\rp$, 
for all $\om  \in L^1_{loc}\gamma(\mathring{X},E_\rho \otimes\Lambda^q T^*\mathring{X})$, and $\vv\in \Omega^{n-q}_0(\mathring{X}, E_\rho)$.

\subsection{Explicit expression for the relevant geometric operators}
\label{explicit}

We describe in this section the explicit expression for the more relevant formal operators on the collar. Recall we have  an isometry $\sigma_\CF:(0,l+\ep)\times W\to X$, 
of the product space with a collar neighbourhood of $X_W$ inside $X$, we denote the image of $\sigma_\CF$ by $\CF$. The metric  restricted to $\CF$ reads
\[
g|_\CF=dx\otimes dx+x^\ka H^2(x) \t g,
\]
for some smooth positive $H$, with $H(0)=1$. Taking  $\omega\in \ga(\mathring{X}, \Lambda^q T^* \mathring{X})$ in the domain of either of $D$ or of $\DD$, and using separation of variables, 
\[
\omega|_\CF=f_1\tilde\omega_1+f_2 dx\wedge \tilde\omega_2,
\]
with  functions $f_1,f_2 :(0,l+\ep)\to \R$, and  square integrable forms $\tilde\omega_1\in \ga(W,\Lambda^q T^* W)$ and $\tilde\omega_2 \in \ga(W,\Lambda^{q-1} T^* W)$, and ($f^I$ denotes the weak derivative)
\begin{align*}
\star \omega|_\CF=&
h^{1-2\al_{q-1}} f_2\tilde\star \tilde\omega_2+(-1)^q h^{1-2\al_{q}} f_1 
dx\wedge\tilde\star \tilde\omega_1,\\
D\omega|_\CF  =&f_1 \tilde d\tilde\omega_1 + d x\wedge\left(  
f^I_1 \tilde\omega_1 - f_2  \tilde D\tilde\omega_2\right),\\
\DD \omega|_\CF =& (-1)^{(m+1)q+(m+1)+1}\star D \star\omega\\
=&-h^{2\al_{q-1}-1} (h^{1-2\al_{q-1}}f_2)^I\tilde\omega_2 + 
\frac{1}{h^2}f_1 \tilde \DD\tilde\omega_1  - \frac{1}{h^2}f_2dx \wedge \tilde \DD \tilde\omega_2.
\end{align*}

\begin{align*}
\Delta^{(q)}\omega|_\CF=&(\DD D+D\DD)\om|_\CF=\FF_1^q (f_1\t\omega_1)
-2\frac{h'}{h}f_2 \t D\t\omega_2
+dx \wedge\FF_2^{q} (f_2  \t\omega_2)
-2\frac{h'}{h^3}f_1 dx\wedge\t \DD\t\omega_1.
\end{align*}
where (compare with equation (\ref{F})
\beq\label{F0}\begin{aligned}
\FF_1^q (f\t\om)&=-h^{2\al_q-1}\left(h^{1-2\al_q}f^I\right)^I\t\om +\frac{f}{h^2} \t \Delta^{(q)}\t\om,\\
\FF_2^{q}&=\FF_1^{q-1}-(1-2\al_{q-1})\left(\frac{h'}{h}\right)' ,
\end{aligned}
\eeq
and $\t \Delta$ denotes the Hodge Laplace operator on square integrable forms on the section $W$.  
Note that for square integrable forms, there is the identification
\begin{align}\label{ide1}
L^2(\CF)&=
L^2((0,l+\ep),h^{1-2\al_{q}}\oplus  h^{1-2\al_{q-1}})\otimes   \left(L^2\ga^q(W)\times L^2\ga^{q-1}(W)\right).
\end{align}

If $\om$ and $\vv$ are locally integrable, then, a direct calculation exploiting the formula for the Hodge star given above,  gives
\beq\label{product}
\lp \om,\vv\rp|_\CF=\int_0^{l+\ep} h^{1-2\al_{q}} f_1 g_1  \lp \t\om_1,\t\vv_1\rp_W
+\int_0^{l+\ep} h^{1-2\al_{q-1}} f_2 g_2  \lp \t\om_2,\t\vv_2\rp_W.
\eeq

Moreover, we have the following Green formula. 

\begin{lem}\label{GreenDistributions} For all the pairs: $\om  \in L^1_{loc}\gamma(\mathring{X},E_\rho \otimes\Lambda^q T^*\mathring{X})$, in the domain of $D$, and $\vv\in \Omega^{n-q-1}_0(\mathring{X}, E_\rho)$, and all the pairs: $\vv  \in L^1_{loc}\gamma(\mathring{X},E_\rho \otimes\Lambda^q T^*\mathring{X})$, in the domain of $\DD$, and $\om\in \Omega^{n-q}_0(\mathring{X}, E_\rho)$, we have $\lp D\om ,\vv\rp=\lp \om, \DD\vv\rp$. 
\end{lem}


\vspace{10pt}

\section{De Rham theory}
\label{DR}

In this section we study some possible extensions of the classical de Rham theory on the space $X$ with an horn singularity. More precisely, we give some explicit description of the domain of the maximal closed extension of the exterior derivative that permits to produce some boundary conditions that characterise the minimal closed extension. Using this extension we introduce some de Rham complexes of closed operator, and we compute their homology. The results of this section should be compared with the analogous ones in \cite{Spr12} for the conical singularity.

\subsection{The exterior derivative operator and its adjoint}
\label{ss3.1}

We define the concrete operators in the Hilbert space  $L^2 (X,E_\rho)$ associated to formal operators $D^q$ and $\DD^q$.  
We define the pre minimal operators  by 
\begin{align*}
\D(\d^q_0)&=\Omega_0^q(\mathring{X},E_\rho),& \d_0^q\om=d^q\om,\\
\D(\dr^q_0)&=\Omega_0^q(\mathring{X},E_\rho),& \dr_0^q\om=\de^q\om,
\end{align*}
and the maximal operators by 
\begin{align*}
\D(\d^q_{\rm max})&=\{\om\in L^2 (X,E_\rho)~|~D^q \om\in L^2 (X,E_\rho)\}& \d_{\rm max}^q\om=D^q\om,\\
\D(\dr^q_{\rm max})&=\{\om\in L^2(X,E_\rho)~|~\DD^q \om\in L^2 (X,E_\rho)\},& \de_{\rm max}^q\om=\DD^q\om.
\end{align*}

The operators $\d^q_0$ and $\dr^q_0$ are densely defined and closable. 
Their closures $\d^q_{\rm min}=\overline{\d^q_0}$ and $\dr^q_{\rm min}=\overline{\dr^q_0}$ are the operators \cite{Wei}: 
\begin{align*}
\D(\d^q_{\rm min})&=\{ \om \in L^2(X,E_\rho)~|~ \exists \{\om_n\}\in \D( \d^q_0), \om_n\to_n \om, \d^q_{0}\om_n\to \vv\in L^2(X,E_\rho)\},\\
\d^q_{\rm min}\om&=\vv,\\
\D(\dr^q_{\rm min})&=\{  \om \in L^2(X,E_\rho)~|~\exists \{\om_n\}\in \D( \dr^q_0), \om_n\to_n \om, \de^q_{0}\om_n\to \vv\in L^2(X,E_\rho)\},\\
\dr^q_{\rm min}\om&=\vv.
\end{align*}

These operators are each other adjoints in the following sense (and therefore  closed):
\begin{align*}
(\d^q_{\rm min})^\da&=(\d^q_0)^\da=\dr^{q+1}_{\rm max},&(\dr^q_{\rm min})^\da&=(\dr^q_0)^\da=\d^{q-1}_{\rm max},\\
(\dr^{q}_{\rm max})^\da&=\d^{q-1}_{\rm min},&(\d^{q}_{\rm max})^\da&=\dr^{q+1}_{\rm min}.
\end{align*}

\subsection{Explicit characterisation of the maximal operators}
\label{max}

We need some technical results, with quite elementary proof. 

\begin{rem}\label{rem1} Note that $L^2((0,1), h^c)\subseteq L^2((0,1), h^{c'})$ if $c\leq c'$. Note also that $h^a f\in L^2((0,1), h^{2c})$ if and only if $f\in L^2((0,1), h^{2(a+c)})$.
\end{rem}

\begin{lem}\label{LemA.0} Let  $f\in L^2((0,l],h^{2a})$, and $f^I\in L^2((0,l],h^{2b})$, 
then $f\in AC_{\rm loc}((0,l])$, and $f$ behaves as follows near $x=0$:
\begin{enumerate}
\item if $b<\frac{1}{2\ka}$,  $a\leq-\frac{1}{2\ka}$, then
$
f(x)=O(x^{\frac{1}{2}-\ka b});
$

\item if $b<\frac{1}{2\ka}$,  $a>-\frac{1}{2\ka}$, then
$
f(x)=C+O(x^{\frac{1}{2}- \ka b}),
$
where $C\not=0$ is  a constant that depends on $f$;
\item if $b=\frac{1}{2\ka}$, then
$
f(x)=O(\sqrt{|\log x|});
$
\item if $b>\frac{1}{2\ka}$, then
$
f(x)=O(x^{\frac{1}{2}-\ka b}).
$
\end{enumerate}
\end{lem}
\begin{proof} 
The result follows observing that
\begin{align*}
\left|\int_x^l f^I\right|\leq \left|\int_x^l h^{-b}h^{b}f^I\right|
\leq\sqrt{\int_x^l h^{-2b}}
\sqrt{\int_0^l h^{2b} {f^I}^2}
\leq\sqrt{\int_x^l h^{-2b}}\|f^I\|.
\end{align*}
\end{proof}

\begin{lem}\label{lem1}  If $\om\in \D(\d^{(q)}_{\rm max})$, $\om|_\CF=f_1\t\om_1+dx\wedge f_2\t\om_2$, and  $f_1$ is not a constant, then  $f_1$ and $f_2$ are characterised as follows.  If $\t D\t\om_2=0$, then $f_2\in L^2((0,l+\ep),h^{1-2\al_{q-1}})$. Otherwise, $ f_1\in AC_{\rm loc}((0,l+\ep))$, and we have the following cases (where it is assumed that $\t D\t\om_2\not=0$, note however that this is irrelevant for the behaviour of $f_1$).

\begin{enumerate}
\item  If $\t D\t\om_1\not=0$, then   $f_2\in L^2((0,l+\ep),h^{1-2\al_{q}})$, $f_1\in L^2((0,l+\ep),h^{1-2\al_{q+1}})$,  $f_1^I\in L^2((0,l+\ep),h^{1-2\al_{q}})$, and (for small $x$)
$
f_1(x)=O(x^{\frac{1}{2}-\left(\frac{1}{2}-\al_q\right)\ka}).
$

\item If $\t D\t\om_1=0$, and either $\t\om_1\not= \t D\t\om_2$, or $f_2\not=f_1^I$, then $f_2\in L^2((0,l+\ep),h^{1-2\al_{q}})$, $f_1\in L^2((0,l+\ep),h^{1-2\al_{q}})$,  $f_1'\in L^2((0,l+\ep),h^{1-2\al_{q}})$, and (for small $x$)
\begin{enumerate}
\item if $\al_q\not=\frac{1}{2}$, then
$
f_1(x)=O(x^{\frac{1}{2}-\left(\frac{1}{2}-\al_q\right)\ka})
$, 
\item if $\al_q=\frac{1}{2}$, then
$
f_1(x)=C+O(x^{\frac{1}{2}-\left(\frac{1}{2}-\al_q\right)\ka})=C+O(\sqrt{x}),
$ 
where $C$ is a constant that depends on $f$;

\end{enumerate}

\item If $\t D\t\om_1=0$,  $\om_1= \t D\t\om_2$, and $f_2=f_1^I$, then $f_2\in L^2((0,l],h^{1-2\al_{q-1}})$, $f_1\in L^2((0,l+\ep),h^{1-2\al_{q}})$,  $f_1^I\in L^2((0,l+\ep),h^{1-2\al_{q-1}})$, and (for small $x$)
$
f_1(x)=O(x^{\frac{1}{2}-\left(\frac{1}{2}-\al_{q-1}\right)\ka}).
$

\end{enumerate}

\end{lem}
\begin{proof} The proof follows by Lemma \ref{LemA.0}. 
In details, and considering the restriction on the finite horn, if $\om=f_1\t\om_1+dx\wedge f_2\t\om_2\in \D(\d^{(q)}_{\rm max})$, then obviously $f_1\in L^2((0,l],h^{1-2\al_{q}})$, and $f_2\in L^2((0,l],h^{1-2\al_{q-1}})$, since $\om$ is square integrable. It is clear that this is the unique condition on $f_2$ when $\t D \t\om_2=0$. For $D\om$ to be square integrable, since
\[
D^{(q)}\om=f_1\t D\t\om_1+dx \wedge(f_1^I\t\om_1-f_2\t D \t\om_2),
\]
there are the following  possibilities. 
\begin{enumerate}
\item  If $\t D\t\om_1\not=0$, then  it is necessary that $f_1\in L^2((0,l],h^{1-2\al_{q+1}})$, $f_2\in L^2((0,l],h^{1-2\al_{q}})$, and $f_1^I\in L^2((0,l],h^{1-2\al_{q}})$, since $\t\om_1$ can not be a multiple of $\t D\t \om_2$. Whence, applying Lemma \ref{LemA.0} to $f=f_1$, with  $a=\frac{1}{2}-\al_{q+1}$ and $b=\frac{1}{2}-\al_q$. We compute that
$
b=\frac{1}{2}-\al_q<\frac{1}{2\ka}
$, gives 
\beq\label{cond1}
\al_q> \frac{1}{2}-\frac{1}{2\ka}.
\eeq

Since $\ka>1$, we have that
$
0<\frac{1}{2}-\frac{1}{2\ka}<\frac{1}{2}. 
$
Since $\al_q$ is an half integer, condition (\ref{cond1}) is equivalent to $\al_q\geq \frac{1}{2}$. Proceeding analogously, the condition
$
b=\frac{1}{2}-\al_q\geq \frac{1}{2\ka},
$ 
gives $\al_q\leq 0$. The conditions on $a$ may be worked out in the same way, and this completes the proof in this first case.

\item If $\t D\t\om_1=0$, and either $\om_1\not= \t D\t\om_2$, or $f_2\not=f_1^I$, then $f_1\in L^2((0,l],h^{1-2\al_{q}})$, $f_2\in L^2((0,l],h^{1-2\al_{q}})$, and $f_1^I\in L^2((0,l],h^{1-2\al_{q}})$ (this happens in particular when $\t\om_1$ is an harmonic). 
Whence, applying Lemma \ref{LemA.0} to $f=f_1$, with $a=b=\frac{1}{2}-\al_q$, and proceeding as in the previous point we have the result.

\item If $\t D\t\om_1=0$,  $\om_1= \t D\t\om_2$, and $f_2=f_1^I$, then $f_1\in L^2((0,l],h^{1-2\al_{q}})$, and $f_2=f_1^I\in L^2((0,l],h^{1-2\al_{q-1}})$. Whence, applying Lemma \ref{LemA.0} to $f=f_1$, with $a=\frac{1}{2}-\al_{q}$, and $b=\frac{1}{2}-\al_{q-1}$, and proceeding as in the previous point we have the result.

\end{enumerate}

\end{proof}

\begin{lem}\label{lem2}  If $\om\in \D(\dr^{(q)}_{\rm max})$, $\om|_\CF=f_1\t\om_1+dx\wedge f_2\t\om_2$, and $f_2\not=h^{2\al_{q-1}-1}$,  then $f_1$ and  $f_2$ are characterised as follows. If $\t \DD \t\om_1=0$, then $f_1\in L^2((0,l+\ep),h^{1-2\al_{q}})$. Otherwise,  $ f_2\in AC_{\rm loc}((0,l+\ep))$, and we have the following cases (where it is assumed that $\t \DD \t\om_1\not =0$, note however that this is irrelevant for the behaviour of $f_2$).
\begin{enumerate} 

\item If $\t \DD\t\om_2\not=0$,  then $f_1\in L^2((0,l+\ep),h^{1-2\al_{q+1}})$, $f_2\in L^2((0,l+\ep),h^{1-2\al_{q}})$,  $(h^{1-2\al_{q-1}} f_2)^I\in L^2((0,l+\ep), h^{2\al_{q-1}-1})$, and (for small $x$)
$f_2(x)=O(x^{\frac{1}{2}-\left(\frac{1}{2}-\al_{q-1}\right)\ka})$. 

\item If $\t \DD\t\om_2=0$,  and either $\t\om_2\not = \t \DD\t\om_1$, or $h^{2\al_{q-1}-1}(h^{1-2\al_{q-1}} f_2)^I\not=\frac{f_1}{h^2}$, then  $f_1\in L^2((0,l+\ep),h^{1-2\al_{q+1}})$, $f_2\in L^2((0,l+\ep),h^{1-2\al_{q-1}})$,     
$h^{2\al_{q-1}-1}(h^{1-2\al_{q-1}} f_2)^I\in L^2((0,l+\ep), h^{1-2\al_{q-1}})$, and (for small $x$)
\begin{enumerate}
\item if $\al_{q-1}\not=\frac{1}{2}$, then
$
f_2(x)=O(x^{\frac{1}{2}-\left(\frac{1}{2}-\al_{q-1}\right)\ka});
$
\item if $\al_{q-1}=\frac{1}{2}$, then
$
f_2(x)=Ch^{2\al_{q-1}-1}+O(x^{\frac{1}{2}-\left(\frac{1}{2}-\al_{q-1}\right)\ka}), 
$
where $C$ is a constant that depends on $f$;

\end{enumerate}

\item If $\t \DD\t\om_2=0$,   $\t\om_2= \t \DD\t\om_1$, and $h^{2\al_{q-1}-1}(h^{1-2\al_{q-1}} f_2)^I=\frac{f_1}{h^2}$, i.e.
$
\om=f_1\t\om_1+f_2 dx\wedge\t\om_2=-\de^{(q+1)} (h^2 f_2 dx \wedge\t\om_1), 
$
then   $f_1\in L^2((0,l+\ep),h^{1-2\al_{q}})$, $f_2\in L^2((0,l+\ep),h^{1-2\al_{q-1}})$,  and   $h^{2\al_{q-1}-1}(h^{1-2\al_{q-1}} f_2)^I=\frac{f_1}{h^2}\in L^2((0,l+\ep), h^{1-2\al_{q-2}})$, and (for small $x$):
$f_2(x)=O(x^{\frac{1}{2}-\left(\frac{1}{2}-\al_{q-2}\right)\ka}). 
$
\end{enumerate}

\end{lem}

\begin{lem}\label{lll1} The form $\om=\t\om_1$ is in the domain of $\d_{\rm max}^{(q)}$, if and only if  either $\t\om_1$ is co exact and $\al_q<0$ or $\t\om_1\in \ker \t D$ and $\al_q<1$. The form $\vv=dx\wedge h^{2\al_{q-1}-1} \t\vv_2$ is in the domain of $\dr_{\rm max}^{(q)}$, if and only if  either $\t\vv_2$ is exact and $\al_{q-1}>1$ or $\t\vv_2\in \ker \t \DD$ and $\al_{q-1}>0$. 
\end{lem}

\begin{lem}\label{lemmanew1}  If $\om\in \D(\d^{(q)}_{\rm max})$,  $\vv=\in \D(\dr^{(q+1)}_{\rm max})$, $\om|_\CF=f_1\t\om_1+f_2 dx\wedge\t\om_2$,   $\vv|_\CF=g_1\t\vv_1+dx\wedge g_2\t\vv_2$, then: 
\[
\lim_{x\to 0^+} (h^{1-2\al_q}f_1 g_2)(x)\lp\t\om_1,\t\vv_2\rp_W=0,
\]
except than when $\t\om_1$ and $\t\vv_2$ are either both exact, both co exact or both harmonic.
\end{lem}

\subsection{Boundary conditions}

We introduce some boundary conditions. The boundary conditions at the boundary $\b X$ are the classical ones, see for example \cite{Gil} or \cite{RS}. 
We further introduce some boundary conditions at the tip of the horn. 
Let $\al\in \DS(\dr_{\rm max}^q)$, $\al|_\CF(x,y)=a_1(x)\tilde\al_1(y)+a_2(x)dx\wedge \tilde\al_2(y)$.  We know that $a_2$ is absolutely continuous on $(0,l+\ep)$. We call absolute boundary value at the tip of the horn the value
\[
bv^{(q)}_{\rm abs}(x_0)(\al)=\lim_{x\to 0^+} (h^{1-2\al_{q-1}}f_1 a_2)(x)(\t\om_1,\t\al_2)_W,
\]
for any $\om\in \DS(\d^{(q-1)}_{\rm max})$, with $\om |_\CF=f_1\t\om_1+dx\wedge f_2\t\om_2$. Let $\al\in \DS(\d_{\rm max}^q)$, $\al|_\CF(x,y)=a_1(x)\tilde\al_1(y)+a_2(x)dx\wedge \tilde\al_2(y)$. We know that $a_1$ is absolutely continuous on $(0,l+\ep)$. We call relative boundary value at the tip of the horn the value
\[
bv^{(q)}_{\rm rel}(x_0)(\al)=\lim_{x\to 0^+} (h^{1-2\al_{q}}a_1 g_2)(x)(\t\al_1,\t\vv_2)_W, 
\]
for any $\vv\in \DS(\dr^{(q+1)}_{\rm max})$, $\vv|_\CF=g_1\t\vv_1+dx\wedge g_2\t\vv_2$.

\subsection{Green formula and  minimal operator}
\label{ss3.10}

Direct calculation gives the following results. 

\begin{prop}\label{greenformulaXd} Let $\om\in \D(\d_{\rm max}^q)$ and $\vv\in \D(\dr_{\max}^{q+1})$, such that $\om|_\CF=f_1\t\om_1+dx\wedge f_2\t\om_2$, and $\vv|_\CF=g_1\t\vv_1+dx\wedge g_2\t\vv_2$, then we have the following Green formula
\[
\lp \d^q_{\rm max} \om,\vv\rp=\lp\om, \dr_{\rm max}^{q+1}\vv\rp+\lim_{x\to 0^+} (h^{1-2\al_q}f_1 g_2)(x)  \lp\t\om_1,\t\vv_2\rp_W+\int_{\b X} t \om\wedge\star_W n\vv.
\]
\end{prop}

\begin{prop}\label{domd1X} The closure of the  operator $\d^{q}_{0}$  is the operator $\d^{q}_{\rm min}$ with  the following domain:
\begin{align*}
\D(\d^{q}_{\rm min})
=&\left\{\om\in \D(\d^{q}_{\rm max})~|~\om|_\CF=f_1\t\om_1+dx\wedge f_2\t\om_2, \forall \vv\in \D(\dr^{q+1}_{\rm max}),\, \vv|_\CF=g_1\t\vv_1+dx\wedge  g_2\t\vv_2, \right.\\
&\hspace{20pt}\left.\,\lim_{x\to 0^+} (h^{1-2\al_q}f_1 g_2)(x)\lp\t\om_1,\t\vv_2\rp_W= (\om_{\rm tan})|_{\b X}=0 \right\}.
\end{align*}
\end{prop}

\begin{prop}\label{domdeX} The closure of the  operator $\dr^{q}_{0}$  is the operator $\dr^{q}_{\rm min}$ with  the following domain (here $p\geq 1$ is an integer):
\begin{align*}
\D(\dr^{(q)}_{\rm min})
=&\left\{\vv \in \D(\dr^{q}_{\rm max})~|~ \vv|_\CF=g_1\t\vv_1+dx\wedge  g_2\t\vv_2,\,\forall\om\in \D(\d^{q-1}_{\rm max}),\,\om|_\CF=f_1\t\om_1+dx\wedge f_2\t\om_2, \right.\\
&\hspace{20pt}\left.\,\lim_{x\to 0^+} (h^{1-2\al_{q-1}}f_1 g_2)(x)\lp\t\om_1,\t\vv_2\rp_W= (\vv_{\rm norm}) |_{\b X}=0 \right\}.
\end{align*}
\end{prop}

\begin{prop} Let $A$ be any of the operators $\d_{\rm min}$, $\d_{\rm max}$, $\dr_{\rm min}$, or $\dr_{\rm max}$. Then, $A^2=0$ and $\Im (A^q)\subseteq \DS(A^{q\pm 1})$.
\end{prop}

In order to prove the next result we may proceed as in the proof of Proposition 4.6.9 of \cite{Spr12}. It is only necessary to adapt the behaviour of the integrals near $x=0$, as in the proof of Lemma \ref{LemA.0}. Moreover, observe that this result would follows by the explicit calculation of co homology below.

\begin{prop} Let $B$ denote the restriction of any of the operators $\d_{\rm min}$, $\d_{\rm max}$, $\dr_{\rm min}$, or $\dr_{\rm max}$ to the finite horn $\ZZ$. Then, $B$ has closed range.
\end{prop}

\subsection{The intersection exterior derivative operator and its dual}
\label{ss3.16}

\begin{defi}\label{defdrelX} The operators $\d^{(\bu)}_{\mf,\rm rel}$ and $\d^{(\bu)}_{\mf^c,\rm rel}$ are defined  as follows: 

\begin{enumerate}

\item if $(\pf, m)=(\mf,2p-1), (\mf^c,2p-1), (\mf^c, 2p)$,  then:
\begin{align*}
\DS(\d^{(q)}_{\mf, \rm rel})=\DS(\d^{(q)}_{\mf^c, \rm rel})&=\DS(\d^{(q)}_{\rm min}),&0&\leq q\leq p,\\
\DS(\d^{(q)}_{\mf, \rm rel})=\DS(\d^{(q)}_{\mf^c, \rm rel})&=\left\{\om\in \DS(\d^{(q)}_{\rm max})~|~bv_{\rm rel}(\b X)(\om)=0\right\},
&p+1&\leq q\leq m+1;
\end{align*}

\item if $(\pf,m)=(\mf, 2p)$,  then:
\begin{align*}
\DS(\d^{(q)}_{\mf, \rm rel})&=\DS(\d^{(q)}_{\rm min}),&0&\leq q\leq p-1,\\
\DS(\d^{(q)}_{\mf, \rm rel})&=\left\{\om\in \DS(\d^{(q)}_{\rm max})~|~bv_{\rm rel}(\b X)(\om)=0\right\},
&p&\leq q\leq 2p+1.
\end{align*}


\end{enumerate}

We call these operators the lower and upper middle perversities intersection exterior differentiation operators on $X$ with relative  boundary conditions. 
\end{defi}

\begin{defi}\label{defdabsX} The operators $\d^{(\bu)}_{\mf,\rm abs}$ and $\d^{(\bu)}_{\mf^c,\rm abs}$ are defined  as follows: 

\begin{enumerate}
\item if $(\pf, m)=(\mf,2p-1), (\mf^c,2p-1), (\mf^c, 2p)$,  then:
\begin{align*}
\DS(\d^{(q)}_{\mf, \rm abs})=\DS(\d^{(q)}_{\mf^c, \rm abs})&=\left\{\om\in \DS(\d^{(q)}_{\rm max})~|~bv_{\rm rel}(x_0)(\om)=0\right\},&0&\leq q\leq p,\\
\DS(\d^{(q)}_{\mf, \rm abs})=\DS(\d^{(q)}_{\mf^c, \rm abs})&= \DS(\d^{(q)}_{\rm max}),
&p+1&\leq q\leq m+1;
\end{align*}

\item if $m=2p$,  then:
\begin{align*}
\DS(\d^{(q)}_{\mf, \rm abs})&=\left\{\om\in \DS(\d^{(q)}_{\rm max})~|~bv_{\rm rel}(x_0)(\om)=0\right\},&0&\leq q\leq p-1,\\
\DS(\d^{(q)}_{\mf, \rm abs})&= \DS(\d^{(q)}_{\rm max}),
&p&\leq q\leq 2p+1.
\end{align*}


\end{enumerate}

We call these operators the lower and upper middle perversities intersection exterior differentiation operators on $X$ with absolute boundary conditions. 
\end{defi}


\begin{defi}\label{defdeabsX} The operators $\dr^{(\bu)}_{\mf,\rm abs}$ and $\dr^{(\bu)}_{\mf^c,\rm abs}$ are defined  as follows: 

\begin{enumerate}

\item if $(\pf, m)=(\mf,2p-1), (\mf^c,2p-1), (\mf^c, 2p)$,  then:
\begin{align*}
\DS(\dr^{(q)}_{\mf, \rm abs})=\DS(\dr^{(q)}_{\mf^c, \rm abs})&=\left\{\vv\in \DS(\dr^{(q)}_{\rm max})~|~bv_{\rm abs}(\b X)(\vv)=0\right\},&0&\leq q\leq p+1,\\
\DS(\dr^{(q)}_{\mf, \rm abs})=\DS(\dr^{(q)}_{\mf^c, \rm abs})&=\DS(\dr^{(p)}_{\rm min}),
&p+2&\leq q\leq m+1;
\end{align*}

\item if $m=2p$,  then:
\begin{align*}
\DS(\dr^{(q)}_{\mf^c, \rm abs})&=\left\{\vv\in \DS(\dr^{(q)}_{\rm max})~|~bv_{\rm abs}(\b X)(\vv)=0\right\},&0&\leq q\leq p+1,\\
\DS(\dr^{(q)}_{\mf^c, \rm abs})&=\DS(\dr^{(p)}_{\rm min}),
&p+2&\leq q\leq 2p+1.
\end{align*}


\end{enumerate}

\end{defi}

\begin{defi}\label{defderelX} The operators $\dr^{(\bu)}_{\mf,\rm rel}$ and $\dr^{(\bu)}_{\mf^c,\rm rel}$ are defined  as follows: 

\begin{enumerate}

\item if $(\pf, m)=(\mf,2p-1), (\mf^c,2p-1), (\mf^c, 2p)$,  then:
\begin{align*}
\DS(\dr^{(q)}_{\mf, \rm rel})=\DS(\dr^{(q)}_{\mf^c, \rm rel})&=\DS(\dr^{(q)}_{\rm max}),&0&\leq q\leq p+1,\\
\DS(\dr^{(q)}_{\mf, \rm rel})=\DS(\dr^{(q)}_{\mf^c, \rm rel})&=\left\{\vv\in \DS(\dr^{(q)}_{\rm max})~|~bv_{\rm abs}(x_0)(\vv)=0\right\},
&p+2&\leq q\leq m+1;
\end{align*}

\item if $m=2p$,  then:
\begin{align*}
\DS(\dr^{(q)}_{\mf, \rm rel})&=\DS(\dr^{(q)}_{\rm max}),&0&\leq q\leq p,\\
\DS(\dr^{(q)}_{\mf, \rm rel})&=\left\{\vv\in \DS(\dr^{(q)}_{\rm max})~|~bv_{\rm abs}(x_0)(\vv)=0\right\},
&p+1&\leq q\leq 2p+1.
\end{align*}


\end{enumerate}

\end{defi}


We will write $\pf$ for either $\mf$ or $\mf^c$, and ${\rm bc}$ for either $\rm rel$ or $\rm abs$. 
These operators are closed extensions of the minimal operators. They are each other adjoints,  more precisely: $(\d^{(q)}_{\pf,\rm bc})^\da=\dr^{(q+1)}_{\pf,\rm bc}$. 
Moreover, $\Im (\d^{(q)}_{\pf, \rm bc})\subseteq \DS(\d^{(q+1)}_{\pf,\rm bc})$, and $\d^{(q+1)}_{\pf, \rm bc}\d^{(q)}_{\pf, \rm bc}=0$, and $\Im (\dr^{(q)}_{\pf, \rm bc})\subseteq \DS(\dr^{(q-1)}_{\pf, \rm bc})$, and $\dr^{(q-1)}_{\pf, \rm bc}\dr^{(q)}_{\pf, \rm bc}=0$. 

\subsection{Some de Rham complexes}
\label{drcomp}

By the results of the previous section, we see that each of the pairs $(\DS(\d^{(\bu)}_{\pf, \rm bc}), \d^{(\bu)}_{\pf, \rm bc})$, $(\DS(\d^{(\bu)}_{ \rm min}), \d^{(\bu)}_{ \rm min})$, and $(\DS(\d^{(\bu)}_{\rm max}), \dr^{(\bu)}_{ \rm max})$ defines a complex of Hilbert spaces and closed operators. We call these complexes the intersection de Rham complexes, and the relative and the absolute de Rham complex of the space $X$ (with coefficients in $E_\rho$). We  use the following notation 
\begin{align*}
I^\pf H_{\rm DR}^q(X,\b X, E_\rho)&=H_q(\DS( \d_{\pf, \rm rel}^{\bu}), \d_{\pf, \rm rel}^{\bu}),&
I^\pf H^q_{ \rm DR}(X, E_\rho)&=H_q(\DS( \d_{\pf, \rm abs}^{(\bu)}), \d_{\pf, \rm abs}^{(\bu)}),\\
 H_{\rm DR}^q(X,\b X, E_\rho)&=H_q(\DS( \d_{ \rm min}^{\bu}), \d_{ \rm min}^{\bu}),&
 H_{\rm DR}^q(X, E_\rho)&=H_q(\DS( \d_{ \rm max}^{\bu}), \d_{ \rm max}^{\bu}).
\end{align*}

We now compute the homology of these complexes. For we introduce the suitable co homology long exact sequence of the pair on one side, and on the other we compute explicitly the co homology of the finite metric horn $\ZZ$. We proceed by assuming that $X$ has empty boundary.

The inclusions of sub manifolds $i_Z:Z=\ZZ\to X$ and $i_Y:Y\to X$ induce by restriction   maps on the space of forms, and since  they commute with the exterior derivative operators,  the following surjective chain maps (observe that on the smooth  boundary min/max=rel/abs.
\begin{align*}
\is_Z^\bu:& \DS(\d^\bu_{X,\pf})\to \DS(\d^\bu_{Z,\pf, \rm abs}),&
\is_Y^\bu:& \DS(\d^\bu_{X,\pf})\to \DS(\d^\bu_{Y,\rm abs}),\\
\is_Z^\bu:& \DS(\d^\bu_{X,\rm max})\to \DS(\d^\bu_{Z,\rm max}),&
\is_Y^\bu:& \DS(\d^\bu_{X,\rm max})\to \DS(\d^\bu_{Y,\rm max}).
\end{align*}

\begin{lem} We have the identifications: $\ker \is_Y^q=\DS(\d^q_{Z,\pf,\rm rel})(\DS(\d^q_{Z,\rm min}))$,  $\ker \is_Z^q=\DS(\d^q_{Y, \rm rel})(\DS(\d^q_{Y,\rm min}))$.
\end{lem}

\begin{prop}\label{longexactX} For each of the de Rham complexes above, the  inclusion  $i_Y:Y\to X$ and   $i_Z:\ZZ\to X$ naturally induces a  long exact sequence in co homology. For example 
\[
\xymatrix{\dots\ar[r]& H_q(\DS(\d^\bu_{Z,\rm min}), \d^\bu_{Z,\rm min})\ar[r]&H_q(\DS(\d^\bu_{X,\rm  min}),\d^\bu_{X,\rm min})\ar[r]^{\hspace{40pt}\is_Y^{*,q}}&\\
&\ar[r] &H_q(\DS(\d^\bu_{Y,\rm mx}),\d^\bu_{Y,\rm max})\ar[r]& H_{q+1}(\DS(\d^\bu_{Z, \rm min}),\d^\bu_{Z,\rm min})\ar[r]&\dots
}
\]

\end{prop}

\subsection{De Rham co homology}
\label{DRcohom}

In this section we compute the de Rham co homology of the de Rham complexes described in the previous Section \ref{drcomp}. We give details on one particular case.

\begin{lem} \label{formalsol-kerd} The formal solutions of the equation $D\om=0$ in degree $q$ on $\ZZ$ are:
\begin{enumerate}
\item[E.] $\te_E=\t\te_1$,
\item[O.] $\te_O=dx\wedge f_2 \t\te_2$, 
\item[I.] $\te_I=\t D \t\be_1$, 
\item[IV.] $\te_{IV}=dx\wedge f_2 \t D\t\be_2$, 
\item[X.] $\te_{X}=f_1 \t D \t\om_2+dx\wedge f'_1 \t\om_2$, $f_1\not=1$,
\end{enumerate}
where $\t\te_1$ and $\t\te_2$,  are harmonic, $\t\be_2$, $\t\be_1$ and $\t\om_2$ co exact forms  of the section. 
\end{lem}
\begin{proof} The proof is by direct verification, compare with \cite[4.6.1]{Spr12}. 
\end{proof}


\begin{prop}\label{kerd} The kernel of $\d^{(q)}_{\rm max}$ on $\ZZ$ is the following space
\[
\ker \d^{(q)}_{\rm max}=\left\{\begin{array}{ll} (K_E\oplus K_I)+ (K_O\oplus K_{IV})+K_X,& q< \frac{m+1}{2},\\
(K_O\oplus K_{IV})+K_X,& q\geq \frac{m+1}{2}.\end{array}\right.
\]
where
\begin{align*}
K_E=&\{\t\theta\in\H^{q}(W) \},\\
K_O=&\{dx\wedge f\t\theta~|~ \t\theta\in \H^{q-1}(W),  f\in AC_{\rm loc}((0,l]), f\in L^2((0,l],h^{1-2\al_{q-1}})\},\\
K_I=&\{\t\om\in\Im \t D^{q-1} \},\\
K_{IV}=&\{dx\wedge f\t\om~|~ \t\om\in\Im\t  D^{q-2},  f\in AC_{\rm loc}((0,l]), f\in L^2((0,l],h^{1-2\al_{q-1}})\},\\
K_X=&\{f\t \d\t\om+dx\wedge f^I\t\om~|~ \t\om\in\Im \t\DD^{q}, f\in AC_{\rm loc}((0,l]), f\in L^2((0,l],h^{1-2\al_q}), \\
&f^I\in L^2((0,l],h^{1-2\al_{q-1}})\}.
\end{align*}
\end{prop}

\begin{proof} Let $\om=f_1\t\om_1+dx\wedge f_2\t \om_2\in  \LL_1(C_{(0,l]}(W))$. Since 
\[
D\omega  =f_1 \tilde d\tilde\omega_1 + d x\wedge\left(  
f'_1 \tilde\omega_1 - f_2  \tilde D\tilde\omega_2\right),
\]
we have that $\om\in \DS(\d^{(q)}_{\rm max})$ if and only if it satisfies the following  conditions:
\begin{enumerate}
\item[(I)] $f_1\t\om_1\in L^2((0,l],h^{1-2\al_q})\otimes L^2\ga(W,\Lambda^q T^* W)$, and  
$ f_1 \tilde D\tilde\omega_1\in L^2((0,l],h^{1-2\al_{q+1}})\otimes L^2\ga(W,\Lambda^{q+1} T^* W)$,
\item[(II)]
$f_2\t\om_2\in L^2((0,l],h^{1-2\al_{q-1}})\otimes L^2\ga(W,\Lambda^{q-1} T^* W)$, and  
$ f'_1 \tilde\omega_1 - f_2  \tilde D\tilde\omega_2\in L^2((0,l],h^{1-2\al_{q}})\otimes L^2\ga(W,\Lambda^q T^* W)$.
\end{enumerate}


We investigate what of the  harmonics given in Lemma \ref{formalsol-kerd} satisfy these relations.
\begin{enumerate}

\item The forms of types E and I require only condition (I), and this condition is satisfied if and only if $1\in  L^2((0,l],h^{1-2\al_q})$ that is true if and only if $\ka (1-2\al_q)>-1$, i.e. if and only if $\al_q<1$, i.e. if and only if $q<\frac{m+1}{2}$. This gives the space
$
K_E\oplus K_I
$ 
with $q<\frac{m+1}{2}$. 

\item The forms of type O and IV satisfy condition (I), and satisfy condition (II) if and only if $f_2\in  L^2((0,l],h^{1-2\al_{q-1}})$. 
This gives the space
$
K_O\oplus K_{IV}$. 

\item The forms of type X always satisfy condition (II), and satisfy condition (I) if and only if $f\in  L^2((0,l],h^{1-2\al_q})$, and $f'\in  L^2((0,l],h^{1-2\al_{q-1}})$. 
This gives 
$K_X$. 

\end{enumerate}
 
\end{proof}

\begin{lem}\label{formalsol-de} The formal solutions of the equation $\DD\om=0$ in degree $q$ are:
\begin{enumerate}
\item[E.] $\te_E=f_1\t\te_1$,
\item[O.] $\te_O=h^{2\al_{q-1}-1}dx\wedge  \t\te_2$, 
\item[I.] $\te_{I}=f_1 \t \DD\t\be_1$,
\item[IV.] $\te_{IV}=h^{2\al_{q-1}-1}dx\wedge  \t \DD\t\be_2$,
\item[Y.] $\te_{Y}=h^{2\al_{q-1}+1} (h^{1-2\al_{q-1}}f_2)' \t\om_1+dx\wedge f_2 \t \DD\t\om_1$,  $f_2\not=h^{2\al_{q-1}-1}$, 
\end{enumerate}
where $\t\te_1$ and $\t\te_2$ are harmonic, $\t\be_1$, $\t\be_2$ and $\t\om_1$ are exact forms of the section. 
\end{lem}

\begin{prop}\label{kerde} The kernel of $\dr^{(q)}_{\rm max}$ is the following space
\[
\ker \dr^{(q)}_{\rm max}=\left\{\begin{array}{ll}  (H_E\oplus H_I)+H_Y,& q\leq \frac{m+1}{2},\\
(H_O\oplus H_{IV})+(H_E\oplus H_I)+H_Y,& q > \frac{m+1}{2}.\end{array}\right.
\]
where
\begin{align*}
H_E=&\{ f\t\te~|~ \t\te\in \H^{q}(W),  f\in AC_{\rm loc}((0,l]), f\in L^2((0,l],h^{1-2\al_{q}})\},\\
H_O=&\{h^{2\al_{q-1}-1}dx\wedge\t\theta~|~\t\te\in\H^{q-1}(W) \},\\
H_{I}=&\{ f\t\om~|~ \t\om\in\Im \DD^{q+1},  f\in AC_{\rm loc}((0,l]), f\in L^2((0,l],h^{1-2\al_{q}})\},\\
H_{IV}=& \{h^{2\al_{q-1}-1}dx\wedge\t\om~|~\t\om\in\Im \DD^{q} \},\\
H_Y=&\left\{h^{2\al_{q-1}+1} (h^{1-2\al_{q-1}}f)^I \t\om+dx\wedge f \t \dr\t\om ~|~ \t\om\in\Im D^{q-1},\right.\\
&\left. f\in AC_{\rm loc}((0,l]), f\in L^2((0,l],h^{1-2\al_{q-1}}), (h^{1-2\al_{q-1}} f)^I\in L^2((0,l],h^{2\al_{q}-1})\right\}.
\end{align*}
\end{prop}

\begin{theo}\label{coZ} If $(m,\pf)=(2p-1,\pf), (2p,\mf^c)$, then:
\begin{align*}
I^{\pf}H_{ \rm DR}^q(\ZZ,E_\rho)&=\left\{\begin{array}{ll}H^{q}(W), & 0\leq q\leq p-1,\\
\{0\},& p\leq q\leq m+1,
\end{array}\right.\\
I^{\pf}H_{ \rm DR}^q(\ZZ,W,E_\rho)&=\left\{\begin{array}{ll}\{0\}, & 0\leq q\leq p,\\
H^{q-1}(W),& p+1\leq q\leq m+1.
\end{array}\right.
\end{align*}

If $(m,\pf)=(2p,\mf)$, then:
\begin{align*}
I^{\mf}H_{ \rm DR}^q(\ZZ,E_\rho)&=\left\{\begin{array}{ll}H^{q}(W), & 0\leq q\leq p,\\
\{0\},& p+1\leq q\leq 2p+1,
\end{array}\right.\\
I^{\mf}H_{ \rm DR}^q(\ZZ,W,E_\rho)&=\left\{\begin{array}{ll}\{0\}, & 0\leq q\leq p+1,\\
H^{q-1}(W),& p+2\leq q\leq 2p+1.
\end{array}\right.
\end{align*}

If $m=2p-1$, then:
\begin{align*}
H_{ \rm DR}^q(\ZZ,E_\rho)&=\left\{\begin{array}{ll}H^{q-1}(W), & 0\leq q\leq p-1,\\
\{0\},& p\leq q\leq 2p,
\end{array}\right.\\
H_{ \rm DR}^q(\ZZ,W,E_\rho)&=\left\{\begin{array}{ll}\{0\}, & 0\leq q\leq p,\\
H^{q-1}(W),& p+1\leq q\leq 2p.
\end{array}\right.
\end{align*}

If $m=2p$, then:
\begin{align*}
H_{ \rm DR}^q(\ZZ,E_\rho)&=\left\{\begin{array}{ll}H^{q-1}(W), & 0\leq q\leq p,\\
\{0\},& p+1\leq q\leq 2p+1.
\end{array}\right.\\
H_{ \rm DR}^q(\ZZ,W,E_\rho)&=\left\{\begin{array}{ll}\{0\}, & 0\leq q\leq p,\\
H^{q-1}(W),& p+1\leq q\leq 2p+1.\end{array}\right.
\end{align*}

In all cases the isomorphisms are naturally induced by the inclusion $i:W\to C_{(0,l]}(W)$: and act as follows, in the absolute case $\t\te\mapsto [ \t\te]$, and in relative case $\t\te\mapsto [x^{(2\al_{q-1}-1)\ka} dx\wedge \t\te]$.
\end{theo}

\begin{proof} We give details for $H_q(\DS( \d_{ \rm max}^{\bu}), \d_{ \rm max}^{\bu})$, and $H_q(\DS( \d_{ \rm min}^{\bu}), \d_{ \rm min}^{\bu})$. 
The kernel of $\d_{\rm max}$ has been identified  in Lemma \ref{kerd}, now we discuss what of the forms in $\ker \d^{q}_{\rm max}$ are in the image of $\d^{q+1}_{\rm max}$. 

The forms of type I are obviously in the image of $\d^{q}_{\rm max}$, for
$
\te_I=\t D\t\be=D \be,\
$ 
with $\be=\t\be$, and since $1\in L^2((0,l),h^{1-2\al_q})$, it follows, by Remark \ref{rem1},  that $1\in L^2((0,l),h^{1-2\al_{q-1}})$. Similarly, it is clear that
$
\te_{IV}=dx\wedge f \t D\t\be=-D(dx\wedge f\t\be)=D \be,
$ 
and since $f\in L^2((0,l),h^{1-2\al_{q-1}})$, and $1-2\al_{q-2}\geq 1-2\al_{q-1}$, by Remark \ref{rem1},  $f\in L^2((0,l),h^{1-2\al_{q-2}})$, and therefore $\be\in \DS(\d^{(q-1)}_{\pf, \rm rel})$, since the boundary conditions do not affect the forms of this type (note that $\DD\te_{IV}\not=0$). Whence the forms of type IV are boundaries. Consider (where recall that $\t\om$ is co exact)
$
\te_{X}=f \t \D \t\om+dx\wedge f' \t\om=D(f\t\om)=D\be.
$ 
Again, since $f\in L^2((0,l),h^{1-2\al_{q}})$, and $1-2\al_{q-1}>1-2\al_q$, we have $f\in L^2((0,l),h^{1-2\al_{q-1}})$. This proves that $\be$ is in the maximal domain if $\te_X$ is,  and  the forms of type X are boundaries. 

Consider the forms of type O. Let $\te_O\in \ker \d^{(q)}_{ \rm max}$, then
$
\te_O=f dx\wedge \t\te,
$ 
with $\t\te$ harmonic of the section. We look for solutions of the equation
$
\te_O=\d_{\rm max} \om,
$ 
with $\om\in \DS(\d^{(q-1)}_{ \rm max})$. Writing $\om=F_1\t\om_1+F_2 dx\wedge\t\om_2$, then we have the equation
\[
F_1\t D\t\om_1+dx\wedge (F_1'\t\om_1-F_2\t D\t\om_2)=f dx\wedge\t\te.
\]

The solution $F_1=0$ is not acceptable since it  would require that $\t D\t\om_2=\t\te$, that is impossible since $\t\te$ is an harmonic. This also requires that either $F_2=0$ or $\t D\t\om_2=0$; we may chose the second option $F_2=0$ without loss of generality. So $\om$ is a solution only if $\t \om_1=\t\te$, $F_2=0$, and $F_1=F$ with $F'=f$, i.e. if 
$
\om=F\t\te,
$ 
with $F'=f$. We need to discuss the  initial value problem, with some $F_0$, 
\beq\label{init1}
\begin{aligned}
F'&=f, & F(l)&=F_0.
\end{aligned}
\eeq

As well known (see for example \cite[1.2.1]{Zet}), there exists a unique solution of this problem, so let $F$ denotes this solution. Since $\te_O\in \DS(\d_{ \rm max}^{(q)})$, and $\t D \t \te=0$,  $f\in L^2((0,l],h^{1-2\al_{q-1}})$, by Lemma \ref{lem1}. Whence,  
\[
F(x)=-\int_x^l f dx,
\]
with 
\begin{align*}
\left|\int_x^l f dx\right| &\leq \left|\int_x^l h^{\al_{q-1}-\frac{1}{2}}h^{\frac{1}{2}-\al_{q-1}}f\right| 
\leq\sqrt{\int_x^l h^{2\al_{q-1}-1}} dx
\sqrt{\int_0^l h^{1-2\al_{q-1}} f^2}dx \\
&\leq\sqrt{\frac{l^{\ka(2\al_{q-1}-1)+1}-x^{\ka(2\al_{q-1}-1)+1}}{\ka(2\al_{q-1}-1)+1}}\|f\|.
\end{align*}

It follows that the unique (absolutely continuous and locally integrable) solution of the initial value problem in equation (\ref{init1}) has the following behaviour near $x=0$:
\begin{enumerate}
\item if $\al_{q-1}<\frac{1}{2}-\frac{1}{2\ka}$ (note that equality is impossible since $\ka>1$), then
$
F(x)=O(x^{\left(\al_{q-1}-\frac{1}{2}\right)\ka+\frac{1}{2}}),
$ 
\item if $\al_{q-1}>\frac{1}{2}-\frac{1}{2\ka}$, then
$
F(x)=C+O(x^{\left(\al_{q-1}-\frac{1}{2}\right)\ka+\frac{1}{2}}),
$ 
\end{enumerate}
where $C\not=0$ is a constant that depends on $f$. However, by readjusting the constants, we may find an other primitive with $C'=0$, and therefore for all $\al_{q-1}$, 
$
F(x)=O(x^{\left(\al_{q-1}-\frac{1}{2}\right)\ka+\frac{1}{2}}).
$ 
It follows that  
\[
\int_0^l h^{1-2\al_{q-1}} F^2 dx\leq \int_0^l x^\ka dx,
\]
and hence $F\in L^2((0,l],h^{1-2\al_{q-1}})$, and $\om\in \DS(\d_{\rm max}^{(q-1)})$.   Therefore all closed forms of type $O$ are exacts. 

Eventually consider the forms of type $E_0$,
$
\te_{E_0}=\t\te,
$ 
where $\t\te$ is an harmonic of the section. It is clear that these forms may not be boundaries. On the other side, they are square integrable if $1\in L^2((0,l),h^{1-2\al_q})$, i.e. if 
$
\al_q<\frac{1}{2}+\frac{1}{2\ka}.
$ 

This condition is satisfied if and only if $\al_q\leq \frac{1}{2}$, that gives $q\leq p-1$, if $m=2p-1$, and $q\leq p$, if $m=2p$. 
Whence, in that degrees, $[\te_{E_0}]$ is a non trivial co homology class. It is clear that any other forms of this type is a constant multiple of $\te_{E_0}$, and therefore is in the same class.

Consider now $\d_{\rm min}^{(q)}$. It is clear that the forms of types E ad I are not in its kernel since they do not satisfy the bc at $x=l$. Also, for the forms of type IV the bc are vacuum since the tangent component is zero, and this is true also for the forms whose these  forms are boundaries (i.e. the exterior derivatives). Whence the analysis is the same as that for $\d_{\rm max}$, and all the exact forms of type IV are boundaries. 
A similar argument works for the forms of type X: these forms have tangent component, but we observe that if $\te_X=D\be$, then $\te_X$ and $\be$ have the same tangent component, and therefore they either both satisfy or not satisfy the relative bc at $x=0$ and at $x=l$. Whence all closed forms of type X are boundaries. 

It remains to consider the forms of type O. The bc are vacuum for these forms. We can run across the same path as above for $\d_{\rm max}$, but now the initial value problem has fixed initial value: $F_0=0$. This means that we may not readjust the constant to cancel C, and therefore the closed forms of type O are boundaries if $\al_{q-1}<\frac{1}{2}-\frac{1}{2\ka}$, i.e. $\al_{q-1}\leq 0$, that gives $q\leq p$ (for all $m$). Otherwise, namely when $\al_{q-1}>\frac{1}{2}-\frac{1}{2\ka}$, i.e. $\al_{q-1}\geq \frac{1}{2}$,  $F$ is square integrable, only if
\[
\int_0^l h^{1-2\al_{q-1}} dx<\infty,
\]
i.e. if $\al_{q-1}<\frac{1}{2}+\frac{1}{2\ka}$. This means that the unique possibility is $\al_{q-1}=\frac{1}{2}$, and otherwise, i.e. for $\al_{q-1}>\frac{1}{2}$, the closed forms of type O are not boundaries. Whence for $m=2p-1$ and $q\geq p+1$, and for $m=2p$ and $q\geq p+2$, the homology is not trivial. Before considering the particular case $\al_{q-1}=\frac{1}{2}$, we show that in all degrees there is just one homology class. Suppose we have two different closed forms $\te_{O,1}$ and $\te_{O,2}$ as above. Then, normalising the associated functions $F_1$ and $F_2$ by their constant value at $x=0$ and taking the difference, we have a boundary for the difference $\te_{O,1}-\te_{O,2}=D(\om_1-\om_2)$, that are therefore homologous.

It remains to consider the case $\al_{q-1}=\frac{1}{2}$, i.e. $m=2p$ and $q=p+1$. For these values, the closed form $\om$ is well defined and square integrable, we need  to verify the bc. The bc at $x=l$ is satisfied by construction (initial value problem). It remains the bc at $x=0$, 
\[
bv^{(q-1)}_{\rm rel}(x_0)(\om)=\lim_{x\to 0^+} (h^{1-2\al_{q}}F g_2)(x)(\t\te,\t\vv_2)_W, 
\]
for all $\vv\in \DS(\dr^{(q)}_{\rm max})$, $\vv|_\CF=g_1\t\vv_1+dx\wedge g_2\t\vv_2$. With $\al_{q-1}=\frac{1}{2}$, and using Proposition \ref{lem2}(2),
\[
bv^{(p)}_{\rm rel}(x_0)(\om)=\lim_{x\to 0^+}  g_2(x)(\t\te,\t\vv_2)_W=C\not=0, 
\]
the bc is not satisfied. Whence $\om$ is not in the domain of $\d_{\rm min}$ and $\te_O$ is not a boundary, but represent a non trivial homology class in degree $p+1$.

\end{proof}

\begin{theo}\label{coX} If $(m,\pf)=(2p-1,\pf), (2p,\mf^c)$, then:
\begin{align*}
I^{\pf}H_{ \rm DR}^q(X,E_\rho)&=\left\{\begin{array}{ll}H^q_{\rm DR}(Y,E_\rho), & 0\leq q\leq p-1,\\
\ker (j^{*,p}:H^{p}_{\rm DR}(Y,E_\rho)\to H^{p}(W,E_\rho)),&q=p\\
H^q_{\rm DR}(Y,\b Y, E_\rho),& p+1\leq q\leq m+1.
\end{array}\right.
\end{align*}

If $(m,\pf)=(2p,\mf)$, then:
\begin{align*}
I^{\pf}H_{ \rm DR}^q(X,E_\rho)&=\left\{\begin{array}{ll}H^q_{\rm DR}(Y,E_\rho), & 0\leq q\leq p,\\
\ker (j^{*,p+1}:H^{p+1}_{\rm DR}(Y,E_\rho)\to H^{p+1}(W,E_\rho)),&q=p+1\\
H^q_{\rm DR}(Y,\b Y, E_\rho),& p+2\leq q\leq 2p+1.
\end{array}\right.
\end{align*}

If $m=2p-1$, then:
\begin{align*}
H_{ \rm DR}^q(X,E_\rho)&=\left\{\begin{array}{ll}H^q_{\rm DR}(Y,E_\rho), & 0\leq q\leq p-1,\\
\ker (j^{*,p}:H^{p}_{\rm DR}(Y,E_\rho)\to H^{p}(W,E_\rho)),&q=p\\
H^q_{\rm DR}(Y,\b Y, E_\rho),& p+1\leq q\leq 2p.
\end{array}\right.
\end{align*}

If $m=2p$, then (in degrees $q=p,p+1$ the result depends on the particular case):
\begin{align*}
H_{ \rm DR}^q(X,E_\rho)&=\left\{\begin{array}{ll}H^q_{\rm DR}(Y,E_\rho), & 0\leq q\leq p-1,\\
H^q_{\rm DR}(Y,\b Y, E_\rho),& p+2\leq q\leq 2p+1.
\end{array}\right.
\end{align*}

\end{theo}

\begin{proof} The proof is based on the long  homology exact sequence of the pairs $(\ZZ,X)$ and $(Y,X)$, in Proposition \ref{longexactX}. We give some details for the case $m=2p$. By the cohomology sequence of $(\ZZ,X)$, using the calculation of the relative cohomology of $\ZZ$ given in Theorem \ref{coZ}, we have the isomorphism in degrees $0\leq q\leq p$. By the cohomology sequence of $(Y,X)$, using the calculation of the cohomology of the $\ZZ$ given in Theorem \ref{coZ}, we have the isomorphism in degrees $p+2\leq q\leq 2p$. It remains the degree $q=p+1$, that requires a little bit more work. The commutative square 
\[
\xymatrix{
W\ar[r]^j\ar[d]_i&Y\ar[d]^{i_Y}\\
\ZZ\ar[r]_{i_C}&X
}
\]
induces a cohomology ladder (we give details for the case $m=2p$, $\pf=\mf$)
\[
\xymatrix{
&&I^\mf H^{p+1}_{\rm DR}(\ZZ,W,E_\rho)\ar[d]^{e_C^{*,p+1}}&&\\
\dots\ar[r]&H^{p+1}_{\rm DR}(Y,\b Y,E_\rho)\ar[d]\ar[r]^{e_Y^{*,p+1}}&I^\mf H^{p+1}_{\rm DR}(X,E_\rho)\ar[d]_{\is_Y^{*,p+1}}\ar[r]^{i^{*,p+1}_C}&I^\mf H^{p+1}_{\rm DR}(\ZZ,E_\rho)\ar[d]^{i^{*,p+1}}\ar[r]&\dots\\
\dots\ar[r]&H^{p+1}_{\rm DR}(Y,\b Y,E_\rho)\ar[r]_{\xi_Y^{p+1}}&H^{p+1}_{\rm DR}(Y,E_\rho)\ar[r]_{j^{*,p+1}}&H^{p+1}(W)\ar[r]&\dots
}
\]
where in the bottom line is the long cohomology exact sequence of the pair $(Y,W)$. By Theorem \ref{coZ}, $I^\mf H^{p+1}_{\rm abs}(\ZZ)=0$, and therefore $e_Y^{*,p+1}$ is onto; by the same theorem $I^\mf H^{p+1}_{\rm DR}(\ZZ,W)=0$, and therefore $\is_Y^{*,p+1}$ is injective. So, we just need to show that  $\Im \is_Y^{*,p+1}=\Im \xi_Y^{p+1}$. 
For take $x\in I^\mf H^{p+1}_{\rm DR}(X,E_\rho)$. Then, $j^{*,p+1}\is_Y^{*,p+1}(x)=0$, and therefore $i_Y^{*,p+1}(x)\in \ker j^{*,p+1}=\Im \xi_Y^{p+1}$, i.e. $\Im i_Y^{*,p+1}\subseteq\Im \xi_Y^{p+1}$. Conversely, take $y\in \Im \xi_Y^{p+1}$. Then, $y=\xi_Y^{p+1}(u)$. Since $i_Y^{*,p+1} e_Y^{*,p+1}(u)=\xi_Y^{p+1}(u)=y$, it follows that 
$y\in  \Im \is_Y^{*,p+1}$, i.e. that $\Im \xi_Y^{p+1}\subseteq \Im \is_Y^{*,p+1}$. 
\end{proof}

\vspace{10pt}

\section{Hodge Laplace operator and Hodge theory}
\label{hodge}

We have two approaches to Hodge theory: from one side we may consider the Hodge Laplace operator associated to the  de Rham complexes introduced in Section \ref{DR}. From the other, we may define directly an Hodge Laplace operator as a suitable self adjoint extension of the formal operator. The main point is that the two approach coincide,  Theorem \ref{equiv}.


\subsection{Hodge Laplace operators associated to the de Rham complexes: $\Delta_{ \d}$, $\Delta_{ \dr}$, and  $\Delta_{ \pf, \d}$, $\Delta_{ \pf, \dr}$}

Let start with the Hodge Laplace operators  associated (see the end of Sections \ref{ss3.1} and \ref{ss3.10}) to the de Rham complexes $(\DS(\d^\bu_{\rm min}), \d^\bu_{\rm min})$ and $(\DS(\d^\bu_{\rm max}), \d^\bu_{\rm max})$:

\begin{align*}
\Delta^{(q)}_{ \d}&=\d_{ \rm min}^{(q-1)}(\d_{\rm min}^{(q-1)})^\da
+(\d_{\rm min}^{(q)})^\da \d_{ \rm min}^{(q)}=\d_{ \rm min}^{(q-1)}\dr_{\rm max}^{(q)}
+\dr_{\rm max}^{(q+1)} \d_{ \rm min}^{(q)},\\
\Delta^{(q)}_{ \dr}&=\dr_{ \rm min}^{(q-1)}(\dr_{\rm min}^{(q-1)})^\da
+(\dr_{\rm min}^{(q)})^\da \dr_{ \rm min}^{(q)}=\dr_{ \rm min}^{(q+1)}\d_{\rm max}^{(q)}
+\d_{\rm max}^{(q-1)} \dr_{ \rm min}^{(q)};
\end{align*}
and   the   Hodge Laplace operators associated to the intersection complexes, Section \ref{ss3.16}: the one associated to the de Rham complex $(\D(\d^{(\bu)}_{\pf, \rm rel}), \d^{(\bu)}_{\pf, \rm rel})$ 
\begin{align*}
\Delta^{(q)}_{ \d, \pf, \rm rel}&=\d_{\pf, \rm rel}^{(q-1)}(\d_{\pf, \rm rel}^{(q-1)})^\da
+(\d_{\pf, \rm rel}^{(q)})^\da \d_{\pf, \rm rel}^{(q)}=\d_{\pf, \rm rel}^{(q-1)}\dr_{\pf, \rm rel}^{(q)}
+\dr_{\pf, \rm rel}^{(q+1)} \d_{\pf, \rm rel}^{(q)},
\end{align*}
and the one associated to the de Rham complex $(\DS(\dr^{(\bu)}_{\pf, \rm abs}), \dr^{(\bu)}_{\pf, \rm abs})$  
\begin{align*}
\Delta^{(q)}_{ \dr, \pf, \rm abs}&=\d_{\pf, \rm abs}^{(q-1)} (\d_{\pf, \rm abs}^{(q-1)})^\da
+(\d_{\pf, \rm abs}^{(q)})^\da\d_{\pf, \rm abs}^{(q)}
=\d_{\pf, \rm abs}^{(q-1)} \dr_{\pf, \rm abs}^{(q)}
+\de_{\pf, \rm abs}^{(q+1)}\d_{\pf, \rm abs}^{(q)}.
\end{align*}

These operators are densely defined and closable (by definition), and self adjoint. 

\begin{rem} \label{domLap} An explicit description of the domain of the operators $\Delta_{ \d}$, $\Delta_{ \dr}$,  $\Delta^{(q)}_{\d, \pf, \rm bc}$, and $\Delta^{(q)}_{\dr, \pf, \rm bc}$ follows directly by the definition of the operators.
\end{rem}

\begin{rem}\label{nobound} If the boundary of $X$ is empty, then the Hodge Laplace operators of the intersection complex coincide, namely $\Delta^{(q)}_{\d, \pf}=\Delta^{(q)}_{\dr, \pf}$.
\end{rem}


\subsection{The formal Hodge-Laplace operator near the singularity} 
\label{H0}

On $W$ we have the  classical Hodge decomposition (see for example 
\cite[Section 1.5.2]{Gil}) 
$
\Omega^q(W,V_\rho)=\H^q (W,V_\rho)\oplus \t d^{(q-1)}\Omega^q(W,V_\rho)\oplus 
\t\de^{(q+1)}\Omega^q(W,V_\rho),
$ 
where $\H^q (W,V_\rho)=\ker \t\Delta^{(q)}$. The operator $\t\Delta$ is self 
adjoint and has a spectral resolution, namely the spectrum $\Sp(\t\Delta)$ is 
real pure point and discrete (with unique accumulation point at infinite), 
and non negative, i.e. coincides with the set of the eigenvalues of 
$\t\Delta$, each eigenspace has finite dimension, all eigenfunctions are 
smooth, by elliptic regularity, and the set of the eigenfunctions determines a complete orthonormal 
basis of $L^2(W,V_\rho)$. 

We proceed now to ``diagonalise'' the formal Hodge-Laplace operator 
$\Delta^{(q)}$ on the horn. Let $\t\la$ be an eigenvalue of the Hodge-Laplace 
operator on the section: $\t\Delta^{(q)} \t\psi=\t\la\t\psi$. Then, with the convention that $\t\E^{(q)}_{0}=\H^q(W)$, 
$\t\E^{(q)}_{\t\la>0}=\t\E^{(q)}_{\t\la,{\rm cex}}\oplus 
\t\E^{(q)}_{\t\la,{\rm ex}}$, and 
\begin{align*}
\tilde\E^{(q)}_{\tilde\la,{\rm ex}}&=\left\{\tilde\psi\in\Omega^{q}(W)~|~\tilde\Delta^{(q)}\tilde\psi=\tilde 
\lambda\tilde\psi, \, \tilde\omega=\tilde d\tilde\alpha\right\},\\
\tilde\E^{(q)}_{\tilde\la,{\rm cex}}&=\left\{\tilde\psi\in\Omega^{q}(W)~|~\tilde\Delta^{(q)}\tilde\psi=\tilde
 \lambda\tilde\psi, \, \tilde\psi=\tilde \de\tilde\alpha\right\},
\end{align*}
the Hodge decomposition reads 
$
\Omega^q(W)=\bigoplus_{\t\la=0} \t\E^{(q)}_{\t\la}.
$ 
Using this in the identification in equation (\ref{ide1}), we have the direct 
sum decomposition
\begin{align*}
L^2(\CF)=&L^2((0,l+\ep),h^{1-2\al_{q}}\oplus  h^{1-2\al_{q-1}})
\otimes
\left( \bigoplus_{\t\la_1=0} \tilde\E^{(q)}_{\tilde\lambda_1}
\oplus \bigoplus_{\t\la_2=0} \tilde\E^{(q-1)}_{\tilde\lambda_2}\right),
\end{align*}
and 
$
\Delta^{(q)}=\bigoplus_{  \t\la_1,\t\la_2=0}^\infty 
\Delta^{(q)}_{\t\la_1,\t\la_2},
$ 
where the terms are the projections
\beq\label{pro}
\Delta^{(q)}_{\t\la_1,\t\la_2}=
\Delta^{(q)}|_{L^2((0,l],h^{1-2\al_{q}}\oplus  
h^{1-2\al_{q-1}})\otimes\left(\t\E^{(q)}_{\tilde\lambda_{1}}\oplus  
\tilde\E^{(q-1)}_{\tilde\lambda_{2}}\right)}.
\eeq

Moreover
$
\Delta^{(q)}_{\t\la_1,\t\la_2}
=  \bigoplus_{j_1,j_2} \Delta^{(q)}_{\t\la_1,j_1,\t\la_2,j_2},
$ 
where 
\begin{align*}
\Delta^{(q)}_{\t\la_1,j_1,\t\la_2,j_2}
=&\Delta^{(q)}|_{L^2((0,l],h^{1-2\al_{q}}\oplus  
h^{1-2\al_{q-1}})\otimes\left(\langle\t\psi^{(q)}_{\t\la_1,j_1}\rangle\oplus 
\langle\t\psi_{\t\la_2,j_2}\rangle\right)}\\
=&\left(\begin{matrix} \tilde\psi^{(q)}_{\t\la_1,j_1} \FF^q_{1,\t\la_1}& 
-2\frac{h'}{h} \tilde d\tilde\psi^{(q-1)}_{\t\la_2,j_2}\\
 -2\frac{h'}{h^3} \tilde \de \t\psi^{(q)}_{\t\la_1,j_1}& 
 \tilde\psi^{(q-1)}_{\t\la_2,j_2}\FF^{q}_{2,\t\la_2}\end{matrix}\right),
\end{align*}
acts on the space $L^2((0,l],h^{1-2\al_{q}}\oplus  h^{1-2\al_{q-1}})$, and 
$\FF^q_{1,\t\la} =\FF_{1,\al_q,\t\la}$, 
$\FF^q_{2,\t\la}=\FF_{2,\al_{q},\t\la}$, where the last are defined in 
equation (\ref{F}) (see also the end of Section \ref{explicit}).


\subsection{Extensions of the formal Hodge Laplace operator: $\Delta_{\rm max}$, and $\Delta_{\pf, \rm bc}$} \label{defLap}

We proceed to introduce  self adjoint extensions of the formal Hodge Laplace operator.  We define the pre minimal operator  by 
\begin{align*}
\D(\Delta^q_0)&=\Omega_0^q(\mathring{X},E_\rho),& \Delta_0^q\om=(\de^{q+1}d^q+d^{q-1}\de^q)\om,
\end{align*}
and the maximal operator by 
\begin{align*}
\D(\Delta^q_{\rm max})&=\{\om\in L^2 (X,E_\rho)~|~(\DD^{q+1}D^q+D^{q-1}\DD^q)  \om\in L^2 (X,E_\rho)\},\\
 \Delta_{\rm max}^q\om&=(\DD^{q+1}D^q+D^{q-1}\DD^q)\om.
\end{align*}

Note that the definition is well posed, for the weak derivative of a square integrable form is locally square integrable. 
Analogously to the operator $\d^q_0$  the operators $\Delta^q_0$ is densely defined. Note that by definition $  \DS(\d_{\rm max}^{(q)})\cap  \DS(\de_{\rm max}^{(q)})\subseteq \DS(\Delta^q_{\rm max})$. 



\begin{prop}\label{greenformulaX} Let $\om,\vv\in \D(\Delta_{\rm max}^q)$, such that $\om|_\CF=f_1\t\om_1+dx\wedge f_2\t\om_2$, $\vv|_\CF=g_1\t\vv_1+dx\wedge g_2\t\vv_2$, then we have the following Green formula
\[
\langle \Delta_{\rm max}^{(q)}\omega, \vv\rangle=\lp D^{(q)}\omega, D^{(q)}\vv\rp+\lp\DD^{(q)}\omega, \DD^{(q)}\vv\rp+B_{x_0}(\om, \vv)+B_{\b X}(\om, \vv),
\]
where
\begin{align*}
B_{x_0}(\om, \vv)=&\lim_{x\to 0^+} \left(-\left(h^{1-2\al_q} f_1^Ig_1\right)(x) \|\tilde\omega_1\|^2
-\left(\left(h^{1-2\al_{q-1}}f_2\right)^Ig_2\right)(x) \|\tilde\omega_2\|^2\right.\\
&\left.+\left(h^{1-2\al_q}f_2g_1\right)(x) \langle\tilde D\tilde \omega_2, \tilde\vv_1\rangle
+\left(h^{1-2\al_q} f_1 g_2\right)(x) \lp\t\omega_1,\t D \t\vv_2\rp\right),\\
B_{\b X}(\om, \vv)=&\int_{\b X} t\om\wedge\star n D^{(q)} \vv-\int_{\b X} t \DD^{(q)}\om\wedge\star n\vv.
\end{align*}
\end{prop}


Here $t\om=(\om_{\rm tan})|_{\b X}$ and $n\om=( \om_{\rm norm})|_{\b X}$ 
denote the restriction of the tangential and of the normal component of a 
form $\om$ in a collar neighbourhood $\DF$ of  the boundary $\b 
X$ (see \cite[Section 3]{RS} and
\cite[Proposition 1.2.6]{Sch}). Note that using local coordinates on $\DF$ 
like near the tip of 
the horn, we may write $\om|_\DF=f_1\t\om_1+dx f_2\t\om_2$ (where the section 
is not necessarily $W$!), and then $t \om=f_1\t\om_1$, and $n 
\om=f_2\t\om_2$. 
It follows that  
 $\Delta_0^{(q)}\leq (\Delta_0^{(q)})^*=\Delta^{(q)}_{\rm max}$, whence $\Delta_0^{q}$ has a densely defined adjoint, and therefore it is closable \cite{Schm}, and $\Delta^{(q)}_{\rm max}$ is densely defined and, as adjoint, it is closed.  
By the Green formula, taking $\vv\in\DS(\Delta_0)$, we have that
\[
\lp \Delta_{\rm max}\om, \vv\rp=\lp D^{(q)}\omega, D^{(q)}\vv\rp+\lp\DD^{(q)}\omega,\DD^{(q)}\vv\rp=\lp\omega,\dr^{(q)}_{0}\d^{(q)}_{0}\vv\rp+\lp\omega,\d^{(q)}_{\rm 0}\dr^{(q)}_{\rm 0}\vv\rp
=\lp \om, \Delta_0\vv\rp.
\]
thus $(\Delta^*)_{\rm max}\leq \Delta_0^\da$ (where the $*$ denotes the formal adjoint). Since obviously $\Delta^*=\Delta$, 
$\Delta_{\rm max}\leq \Delta_0^\da$. Since the first is maximal, the converse inclusion is also true, and therefore $\Delta_{\rm max}\leq \Delta_0^\da$, and $(\Delta_{\rm max})^\da\leq \overline{\Delta_0}$. 
Also, by the Green formula, we see that it is possible to describe the domain 
of all the self adjoint extensions of $\Delta_0$ by 
suitable boundary conditions at $x=0$ and on $\b X$. 
The analysis on the regular boundary $\b X$ is classical (we refer to 
\cite{Sch,Gil,RS}).

The analysis  at the tip of 
the horn requires some work. It may be developed by applying the 
decomposition in equation (\ref{pro}). Then, we may realise all these 
extensions by using on each term of the sum the extensions  of the induced 
pair of Sturm-Liouville operators. More precisely, locally near the tip of the horn, the operator 
$\Delta^{(q)}_{\t\la_1,j_1,\t\la_2,j_2}$ is essentially a pair of Sturm 
Liouville operators $(\FF^q_{1,\t\la}, \FF^q_{2,\t\la})$ acting on the 
space $L^2((0,l+\ep),h^{1-2\al_{q}}\oplus  h^{1-2\al_{q-1}})$. These two 
operators are of the type described in Appendix \ref{SL} (see also Section \ref{explicit}), and their self 
adjoint extensions classified in Theorem \ref{t3.23}. 
An analogous classification for the self adjoint extensions of $\Delta$ 
follows.

However, here we are interested in some particular extensions, so we proceed 
describing the construction in details in those cases. On one side, at the 
boundary $\b X$, we consider the following two types of boundary values: 
\begin{align*}
BV_{\rm abs}(\b X)(\om)&=\left\{\begin{array}{l}n \om=g_2(0)\t\om_2,\\
n(d\om)=g_1'(0)\t\om_1,
\end{array}\right.&
BV_{\rm rel}(\b X)(\om)&=\left\{\begin{array}{l}t \om=f_1(0)\t\om_1,\\
t(\de\om)=(h^{1-2\al_{q-1}} f_2)'.
\end{array}\right.
\end{align*}
where $\om\in \DS(\Delta_{\rm max}^q)$, $\om|_\DF=f_1\t\om_1+dx\wedge 
f_2\t\om_2$.

On the other side, at the tip of the horn, writing locally $\om\in  
\DS(\Delta_{\rm max}^{(q)})$, as $\om|_\CF=f_1\t\om_1+dx\wedge f_2\t\om_2$, 
and decomposing the forms on the section as in equation (\ref{pro}), 
$\t\om_k=\sum_{\t\la_k,j_k}\lp \t\om_k,\t\psi_{\t\la,j_k}\rp 
\t\psi_{\t\la_k,j_k}$, 
we introduce the  boundary values on each eigen space, compare with Definition \ref{BC}
\[
BV^q_{\t\la_1,\t\la_2,\pm}(x_0)(\om)=\left\{\begin{array}{l}h^{1-2\al_q}\left(f_1
 (\ff_{\al_q,\t\la_1,\pm})'-f_1'\ff_{\al_q,\t\la_1,\pm}\right),\\
h^{1-2\al_{q-1}}\left(f_2 
(\ff_{\al_{q-1},\t\la_2,\pm})'-f_2'\ff_{\al_{q-1},\t\la_2,\pm}\right),
\end{array}\right.
\]
where the $\ff_{\al_q,\t\la_1, \pm}$ are a fundamental system of solutions of 
the equation $\FF_{1,\al_q,\t\la_1}f=0$, and
$\ff_{\al_{q-1},\t\la_2, \pm}$ are a fundamental system of solutions of the 
equation $\FF^q_{2,\al_q,\t\la_2}f=0$.

We can define the operators $\Delta^{(q)}_{\t\la_1,\t\la_2, \pm, \rm bc}$ by
\begin{align*}
\DS&(\Delta^{(q)}_{\t\la_1,\t\la_2, \pm, \rm bc})\\
&=\left\{\om\in 
\DS(\Delta^{(q)}_{\rm max}) ~|~\om|_\CF=f_1\t\om_1+dx\wedge f_2\t\om_2,  
BV^q_{\t\la_1,\t\la_2,\pm}(x_0)(\om)=                BV_{\rm  bc}(\b 
X)(\om)=0\right\},
\end{align*}
and summing up the operators 
\begin{align*}
\Delta^{(q)}_{ \pm, \rm bc}&=\bigoplus_{\t\la_1,\t\la_2} 
\Delta^{(q)}_{\t\la_1,\t\la_2, \pm, \rm bc}.
\end{align*}

Using these operators we define the self adjoint extension of $\Delta_0$ we will work with. In dimension $m=2p-1$, 
we define the operators
\begin{align*}
\Delta_{\mf, \rm bc}&=\Delta_{\mf^c, \rm bc}=\bigoplus_{q=0}^{2p-1} 
\Delta^{(q)}_{+, \rm bc}.
\end{align*}

In dimension $m=2p$, we define the operators
\begin{align*}
\Delta_{\mf^c, \rm bc}&=\bigoplus_{q=0}^{p} \Delta^{(q)}_{+,\rm bc}
\oplus  \left(\Delta^{(p+1)}_{0,0,-,\rm bc}\oplus \bigoplus 
\Delta^{(p+1)}_{\t\la_1>0, \t\la_2>0, +,\rm bc}\right)
\oplus \bigoplus_{q=p+2}^{2p} \Delta^{(q)}_{+,\rm bc},\\
\Delta_{\mf, \rm bc}&=\bigoplus_{q=0}^{p-1} \Delta^{(q)}_{+, \rm bc}
\oplus \left(\Delta^{(p)}_{0,0,-,\rm bc}\oplus \bigoplus 
\Delta^{(p)}_{\t\la_1>0, \t\la_2>0, +,\rm bc}\right)
\oplus  \bigoplus_{q=p+1}^{2p} 
\Delta^{(q)}_{+, \rm bc}.
\end{align*}

Observe that boundary conditions at the tip of the horn are necessary only when the end point $x=0$ of 
the induced Sturm-Liouville problem on $(0,l+\ep)$ is either a regular end point or a regular singular  limit circle end point, and this, according to Corollary \ref{cor1}, may happen only 
in  a limited number of cases for some (small) values of $\al_q$ and $\t\la$, according to the following list:
\begin{enumerate}
\item (regular) $3\leq\ka$, 
\begin{enumerate}
\item $\t\la_1=0$ and $\al_q=\frac{1}{2}$, i.e. $q=\frac{m}{2}$, i.e. $m=2p$ and $q=p$,
\item  $\t\la_2=0$ and $\al_{q-1}=\frac{1}{2}$, i.e. $q=\frac{m}{2}+1$, i.e. $m=2p$ and $q=p+1$,
\end{enumerate}

\item (regular singular)  $\frac{3}{2}\leq\ka< 3$, 
\begin{enumerate}
\item $\t\la_1=0$ and 1) $\al_q=0$, i.e. $m=2p-1$ and $q=p-1$, 2) $\al_q= \frac{1}{2}$, i.e. $m=2p$ and $q=p$,
\item  $\t\la_2=0$ and 1) $\al_{q-1}=0$, $m=2p-1$ and $q=p$, 2) $\al_{q-1}= \frac{1}{2}$, i.e. $m=2p$ and $q=p+1$,
\end{enumerate}

\item (regular singular)  $1< \ka< \frac{3}{2}$,
\begin{enumerate}
 \item $\t\la_1=0$ and 1) $\al_q=- \frac{1}{2}$, i.e.  $m=2p$ and $ q=p-1$, 2) $\al_q= 0$, i.e. $m=2p-1$ and $q=p-1$, 3) $\al_q= \frac{1}{2}$, i.e. $m=2p$ and $q=p$,
\item $\t\la_2=0$ and 1) $\al_{q-1}=- \frac{1}{2}$, i.e. $m=2p$ and $ q=p$, 2) $\al_{q-1}= 0$, i.e. $m=2p-1$ and $q=p$, 3) $\al_{q-1}= \frac{1}{2}$, i.e. $m=2p$ and $q=p+1$.
\end{enumerate}
\end{enumerate}

Whence, in the other cases the bc at the tip of the horn is vacuum. 

The next result depend on some facts proved in the following Section \ref{horn}.

\begin{lem} Let $\om$ be a solution of the eigenvalue equation $\Delta^{(q)}\om=\la\om$. Write $\om_\pm$ according to the local types described in Lemmas \ref{eigenforms} and \ref{harmonics}.  Assume that $\om$ is square integrable. Then, $\om$ satisfies all boundary conditions at the tip of the horn, namely: $BV_{\pm}(0)(\om_\pm)=0$, and $ bv_\pm(0)(\om)=0$. 
\end{lem}
\begin{proof} It is obvious that $BV_{\pm}(0)(\om_\pm)=0$. Consider the other bc. Write locally $\om|_\CF=f_1\t\om_1+dx\wedge f_2\t\om_2$.  Then, $\om$ is square integrable according to Lemma \ref{squareint}.  It follows that $\om|_\CF$ is of type +, beside a small number of cases. The behaviour of $\om$ near the tip of the horn is determined by the  behaviour near $x=0$ of the solution $f$ of the corresponding SL problem, and the behaviour of such solution is described in Theorem \ref{olv1}, when $b>0$, and in Theorem \ref{b0}, when $b=0$. In the first case, it is clear that $ bv_\pm(0)(\om)=0$, thanks to the presence of the exponential. The second case requires explicit verification using the behaviour of the solutions and the characterisation of the forms in the domain of the maximal extensions given in Lemmas \ref{lem1} and \ref{lem2}.
\end{proof}

\subsection{Spectral properties}
\label{spectralpropX}

\begin{prop} \label{prop1.3} The operators $\Delta^{(q)}_{ \pf, \rm bc}$ are non negative self-adjoint operators. 
\end{prop}
\begin{proof} $\Delta$ is formally self-adjoint, $\Delta_{\rm min}$ is densely defined and therefore symmetric, with adjoint $\Delta_{\rm max}$, so  $\Delta_{ \pf, \rm bc}$ that is a restriction of $\Delta_{\rm max}$ is self-adjoint. Non negativity follows by the Green formula,   Propositions \ref{greenformulaX}. 
\end{proof}

\begin{prop}\label{compresX} The operators $\Delta^{(q)}_{ \pf, \rm bc}$ have compact resolvent. 
\end{prop}
\begin{proof} Let $r(z_1, z_2)$ be the  standard kernel for the resolvent of the Hodge-Laplace operator on the compact manifold $Y$ (see for example \cite{Gil}). On the collar, we may use local coordinate and separate the variables, moreover we can use the decomposition on the complete system of eigenforms on the section, so we reduce exactly to the situation described for the finite horn in the proof of Theorem  \ref{compresC}, so $k$ takes the explicit form described in that theorem.  Therefore the minimal operators (on forms with compact support) coincide. Since both the maximal extensions are continuous (and densely defined) they coincide. Since the resolvent on the finite horn is defined by the extension of the resolvent of the collar, this defines the  resolvent of $\Delta^{(\bu)}_{\pf, \rm bc}$, and shows that  it has a continuous square integrable kernel, since it is square integrable on $Y$ by construction, and it is square integrable on the horn by Theorem \ref{compresC}.
\end{proof}


\begin{corol} The operators $\Delta^{(\bu)}_{\pf, \rm bc}$ have discrete non negative spectrum with simple eigenvalues. The eigenfunctions determine a complete orthonormal  basis of $L^2(X)$. 
\end{corol}



\begin{corol}\label{heattraceX} The heat operator $\e^{-t\Delta^{(q)}_{\pf, \rm bc}}$ associated to $\Delta^{(q)}_{\pf, \rm bc}$ is of trace class. Its trace has an asymptotic expansion for small $t$ of the following form
\[
\Tr \e^{-t\Delta^{(q)}_{\pf, \rm bc}}=\sum_{k=0}^\infty a_k(\Delta^{(q)}_{\pf, \rm bc}) t^\frac{m+1-k}{2}.
\]

\end{corol}
\begin{proof} By the theorem, the resolvent has a continuous square integrable kernel whose restriction on $Y$ is the kernel of the Hodge-Laplace on $Y$, and therefore is of trace class and has got the stated local asymptotic expansion,  whose restriction on the horn is the resolvent of the Hodge-Laplace operator on the horn discussed in Theorem \ref{compresC}. This kernel is of trace class since its trace is given by Laplace transform of  the associated logarithmic Gamma function \cite[2.3]{Spr9}, and the last is the one studied in Section \ref{s6.1}. Moreover, the asymptotic expansion exists and is given by the explicit twisted product of the one on the section and the one on the line. The coefficients of the expansion on $X$ are thus given by the sum of the integrals of the local coefficients on the two parts. 
\end{proof}

\subsection{Hodge decomposition for the operators $\Delta_{\d, \rm bc}$ and $\Delta_{\dr, \rm bc}$}

Consider the intersection de Rham complexes $(\DS(\d^{(\bu)}_{\rm min}), \d^{(\bu)}_{ \rm min})$, $(\DS(\d^{(\bu)}_{\rm max}), \d^{(\bu)}_{ \rm max})$, $(\DS(\d^{(\bu)}_{\pf, \rm rel}), \d^{(\bu)}_{\pf, \rm rel})$ and  $(\DS(\dr^{(\bu)}_{\pf, \rm abs}), \dr^{(\bu)}_{\pf, \rm abs})$ introduced in Section \ref{drcomp}. This are complexes of  Hilbert spaces and closed operators (compare \cite{BL0} and references therein). By Theorem  \ref{coZ}, the homology modules of these complexes are finite, therefore  we have the following results.

\begin{theo}\label{Hodgedec} We have the identification of closed vector spaces
\begin{align*}
H_{\rm DR}^q(X,  \b X,E_{\rho})&=H_q(\DS( \d_{ \rm min}^{(\bu)}), \d_{\rm min}^{(\bu)})=\ker \d_{ \rm min}^{(q)}\cap \ker \dr_{\rm max}^{(q)}=\ker \Delta^{(q)}_{ \d},\\
H_{\rm DR}^q(X,  E_{\rho})&=H_q(\DS( \d_{ \rm max}^{(\bu)}), \d_{\rm max}^{(\bu)})=\ker \d_{ \rm max}^{(q)}\cap \ker \dr_{\rm min}^{(q)}=\ker \Delta^{(q)}_{ \dr},\\
I^\pf H_{\rm DR}^q(X, \b X, E_{\rho})&=H_q(\DS( \d_{\pf, \rm rel}^{(\bu)}), \d_{\pf, \rm rel}^{(\bu)})=\ker \d_{\pf, \rm rel}^{(q)}\cap \ker \dr_{\pf, \rm rel}^{(q)}=\ker \Delta^{(q)}_{ \d, \pf,\rm rel},\\
I^\pf H^q_{ \rm DR}(X, E_{\rho})&=H_q(\DS( \dr_{\pf, \rm abs}^{(\bu)}), \dr_{\pf, \rm abs}^{(\bu)})
=\ker \dr_{\pf, \rm abs}^{(q)}\cap \ker \d_{\pf, \rm abs}^{(q)}=\ker \Delta^{(q)}_{ \dr, \pf, \rm abs}.
\end{align*}
and the direct sum decomposition of orthogonal vector spaces
\begin{align*}
L^2 (X, E_{\rho})&=\ker \d_{ \rm min}^{(q)}\cap \ker \dr_{ \rm max}^{(q)}\oplus \Im \d_{ \rm min}^{(q-1)}\oplus \Im \dr^{(q+1)}_{\rm max},\\
L^2 (X, E_{\rho})&=\ker \d_{ \rm max}^{(q)}\cap \ker \dr_{ \rm min}^{(q)}\oplus \Im \d_{ \rm max}^{(q-1)}\oplus \Im \dr^{(q+1)}_{\rm min},\\
L^2 (X, E_{\rho})&=\ker \d_{\pf, \rm rel}^{(q)}\cap \ker \dr_{\pf, \rm rel}^{(q)}\oplus \Im \d_{\pf, \rm rel}^{(q-1)}\oplus \Im \dr^{(q+1)}_{\pf, \rm rel},\\
L^2 (X, E_{\rho})&=\ker \dr_{\pf, \rm abs}^{(q)}\cap \ker \d_{\pf, \rm abs}^{(q)}\oplus \Im \d_{\pf, \rm abs}^{(q-1)} \oplus \Im \dr^{(q+1)}_{\pf, \rm abs}.
\end{align*}
\end{theo}


\section{Analysis on the finite horn}

In this section we study the spectral properties of the Hodge Laplace operators $\Delta_{\rm bc}$ and $\Delta_{\mf, \rm bc}$ on the finite horn $\ZZ$.

\subsection{Formal solutions of the harmonic and of the eigenvalue equation}

\begin{lem}\label{harmonics}   
Let $\left\{\t\psi^{(q)}_{{\rm har},l}, \tilde\psi_{\t\la,{\rm cex},j}^{(q)},\tilde\psi_{\t\la,{\rm ex},k}^{(q-1)}\right\}$  an orthonormal  basis 
of $\Omega^q(W)$ consisting of harmonic forms,   coexact and exact eigenforms, where $\t\la\not= 0$ is an eigenvalue of $\t\Delta^{(q)}$. 
Then, the solutions of the equation $\Delta^{(q)} \omega=0$ are  (sums of solutions) of the following six types:
\begin{enumerate}
\item
$\psi^{(q)}_{E}(x,0)=f \t\psi^{(q)}_{{\rm har}}$, 
where $f$ is a solution of the equation $\FF^q_{1,0} f=0$;
\item
$
\psi^{(q)}_{I,\t\la}(x,0)=f \t\psi^{(q)}_{\t\la,{\rm cex}}
$, 
where $f$ is a solution of the equation $\FF^q_{1,\t\la} f=0$, $\t\la\not=0$;
\item
$
\psi^{(q)}_{O}(x,0)=f dx\wedge\t \psi^{(q-1)}_{{\rm har}}
$, 
where $f$ is a solution of the equation $\FF^q_{2,0} f=0$;
\item
$
\psi^{(q)}_{IV,\t\la}(x,0)=f  dx \wedge\t\psi^{(q-1)}_{\t\la,{\rm ex}}
$, 
where $f$ is a solution of the equation $\FF^q_{2,\t\la} f=0$, $\t\la\not=0$;
\item 
$
\psi_{II,\t\la}^{(q)}(x,0)=f \t d \t\psi^{(q-1)}_{\t\la, {\rm cex}}+f'  dx\wedge \tilde\psi^{(q-1)}_{\t\la, {\rm cex}}
$, 
where $f$ is a solution of $\FF^{q-1}_{1,\t\la} f=0$, $\t\la\not=0$; 
\item 
$
\psi_{III,\t\la}^{(q)}(x,0)
=-h^{2\al_q-1}\left(h^{1-2\al_q}f\right)' \t\psi^{(q)}_{\t\la, \rm ex}-\frac{f}{h^2} dx\wedge \t \de \t\psi^{(q)}_{\t\la, \rm ex}
$, 
where $f$ a solution of the equation $\FF^{q+1}_{2,\t\la}f=0$, with $\t\la\not=0$. 
\end{enumerate}

The operators $\FF^q_{j,b}$ are described in equation (\ref{F0}), and equation (\ref{F}) of Appendix \ref{SL}. 
Note  that the forms of types II and III actually coincide, see Remark \ref{XY} below. 
\end{lem}


\begin{lem}\label{formaleigen} 
With the notation of Lemma \ref{harmonics},  the solutions of the equation $\Delta^{(q)} \omega=\la \om$ are  (sums of solutions) of the following six types:
\begin{enumerate}
\item 
$
\psi^{(q)}_{E}(x,\la)=f \t\psi^{(q)}_{{\rm har}}
$, 
where $f$ is a solution of the equation $\FF^q_{1,0} f=\la f$;
\item 
$\psi^{(q)}_{I,\t\la}(x,\la)=f \t\psi^{(q)}_{\t\la,{\rm cex}}$,
where $f$ is a solution of the equation $\FF^q_{1, \t\la} f=\la f$, $\t\la\not=0$;
\item
$
\psi^{(q)}_{O}(x,\la)=f dx\wedge\t \psi^{(q-1)}_{{\rm har}}
$, 
where $f$ is a solution of the equation $\FF^q_{2,0} f=\la f$;
\item
$\psi^{(q)}_{IV,\t\la}(x,\la)=f  dx \wedge\t\psi^{(q-1)}_{\t\la,{\rm ex}}$, 
where $f$ is a solution of the equation $\FF^q_{2,\t\la} f=\la f$, $\t\la\not=0$;
\item 
$
\psi_{II,\t\la}^{(q)}(x,\la)=f \t d \t\psi^{(q-1)}_{\t\la, {\rm cex}}+f'  dx\wedge \tilde\psi^{(q-1)}_{\t\la, {\rm cex}}
$, 
where $f$ is a solution of $\FF^{q-1}_{1,\t\la} f=\la f$, $\t\la\not=0$; 
\item
$\psi_{III,\t\la}^{(q)}(x,\la)
=-h^{2\al_q-1}\left(h^{1-\al_q}f\right)' \t\psi^{(q)}_{\t\la, \rm ex}-\frac{f}{h^2} dx\wedge \t \de \t\psi^{(q)}_{\t\la, \rm ex}
$, 
where $f$ a solution of the equation $\FF^{q+1}_{2,\t\la}f=\la f$, with $\t\la\not=0$. 
\end{enumerate}
\end{lem}

\begin{rem}\label{smooth} Note that all the solutions of the eigenvalue and of the harmonic equation are smooth. The solutions on the section are smooth by regularity of elliptic operators, the functions $f$ by Sturm Liouville theory.
\end{rem}

\begin{rem}\label{XY} Consider the equation (with $\t\om$ co exact) in degree $q$
\[
\om=h^{2\al_{q-1}+1} (h^{1-2\al_{q-1}}g)' \t\vv+dx\wedge g \t \DD\t\vv=A(f\t D\t\om+dx\wedge f'\t\om),
\]
for some constant $A$. A straightforward calculation shows that this equation has the unique solution (that is independent to $A$): $g= f'$, $\t\om=\t\DD\t\vv$, with $f$ a solution of the equation (compare with equations (\ref{F}) and (\ref{F}))
\[
\FF^{q-1}_{1,\t\la} f=-h^{2\al_{q-1} -1}\left(h^{1-2\al_{q-1} }f'\right)' +\frac{\t\la}{h^2} f=0
\]

For $h^{2\al_{q-1}+1} (h^{1-2\al_{q-1}}f')'=\t\la f$, and $\t\la \t\vv=\t\Delta\t\vv=\t D\t\om$. Note also that, choosing $\be=-\frac{h^2}{\t\la} g dx \wedge \t D\t\om$, in degree $q+1$, we have $\DD\be=\om$. 
\end{rem}

\begin{lem}\label{squareint} Let $\psi_\pm$ be a solution either of  the eigenvalue equations or of the harmonic equation, according to Lemmas \ref{formaleigen} and \ref{harmonics}. Then, $\psi_+, d\psi_+$ and $\de\psi_+$ are square integrable for all $q$. On the other side, $\psi_-, d\psi_-$ and $\de\psi_-$ are square integrable only in the cases described in the following tables. If $m=\dim W=2p-1$:

\begin{center}
\begin{tabular}{|c|c|c|c|}
\hline
forms type&  $\psi^{(q)}\in L^2$ &$d\psi^{(q)}\in L^2$&$\de \psi^{(q)} \in L^2$\\
\hline
$E_-$ & $q=p-1$& $q=p$& $\forall q$\\
$O_-$ &$q=p+1$& $\forall q$& $q =p$\\
$I_-$ &  $q=p-1, \tilde\la_{p-1,k} <1$  & $\not \exists q$   & $\forall q$\\
$IV_-$ &$q=p+1, \tilde\la_{p-1,k}<1$& $\forall q$ & $\not\exists q$ \\
\hline
\end{tabular}
\end{center}

If $m=\dim W=2p$:

\begin{center}
\begin{tabular}{|c|c|c|c|c|}
\hline
forms type&  $\psi^{(q)}\in L^2$ &$d\psi^{(q)}\in L^2$&$\de \psi^{(q)}\in L^2$\\
\hline
$E_-$& $q=p-1$ or $q=p$& $q=p$ or $q=p+1$& $\forall q$\\
$O_-$&$q= p+1$ or $q=p+2$&$\forall q$& $q=p$ or $q=p+1$\\
$I_-$&$\begin{array}{c}q=p-1,  \tilde\la_{p-1,k} < \frac{3}{4} \\ q=p, \tilde\la_{p,k} < \frac{3}{4}\end{array}$&$\not\exists q$ & $\forall q$\\
$IV_-$& $\begin{array}{c} q=p+1,  \tilde\la_{p-1,k} < \frac{3}{4} \\  q=p+2,  \tilde\la_{p,k} < \frac{3}{4} \end{array} $   & $\forall q$& $\not\exists q$\\
\hline
\end{tabular}
\end{center}

\end{lem}

\subsection{Spectral decomposition of the operators  $\Delta_{\pf, \rm bc}$}

Restricted on the finite horn of length $l$, the analysis of the operators $\Delta_{\pf, \rm bc}$ reduces to the analysis of a pair of Sturm Liouville operators. More precisely,   on $\ZZ$, each form globally decomposes as $\om=f_1\t\om_1+dx \wedge f_2\t\om_2$, and the operator 
$\Delta^{(q)}_{\t\la_1,j_1,\t\la_2,j_2}$ consists  in some self adjoint extension of the  pair of Sturm 
Liouville operators $(\FF_{1,\al_q,\t\la_1},\FF_{2,\al_{q},\t\la_2})\equiv (\FF_{1,\al_q,\t\la_1},\FF_{1,\al_{m+1-q},\t\la_2})$,  acting on the 
space $L^2((0,l+\ep),h^{1-2\al_{q}}\oplus  h^{1-2\al_{q-1}})$, see Remark \ref{sss}. These two 
operators are  described in Appendix \ref{SL}. The precise correspondence is described in the following remark.

\begin{rem}\label{formuleLap}
We have the following  description of the operators defined in Section \ref{defLap}. 
\begin{enumerate}
\item $3\leq \ka$:
\begin{enumerate}
\item $m=2p-1$:
\begin{align*}
\Delta_{\mf^c, \rm abs}=\Delta_{\mf, \rm abs}&
\equiv\bigoplus_{q=0}^{2p} \bigoplus_{\t\la_1,\t\la_2}
\left(F_{\al_q,\t\la_1;\frac{\pi}{2}},F_{\al_{2p-q},\t\la_2;0}\right);
\end{align*}

\item $m=2p$:
\begin{align*}
\Delta_{\mf^c, \rm abs}
\equiv&\bigoplus_{q=0,q\not=p,p+1}^{2p+1} \bigoplus_{\t\la_1,\t\la_2}
\left(F_{\al_q,\t\la_1;\frac{\pi}{2}},F_{\al_{2p+1-q},\t\la_2;0}\right)\\
&~\oplus \bigoplus_{\t\la_2} \left(F_{\al_p=\frac{1}{2},0;0,\frac{\pi}{2}},F_{\al_{p+1}=\frac{3}{2},\t\la_2;0}\right)
\oplus\bigoplus_{\t\la_1\t\la_2>0}\left(F_{\al_p=\frac{1}{2},\t\la_1;\frac{\pi}{2}},F_{\al_{p+1}=\frac{3}{2},\t\la_2;0}\right)\\
&~\oplus \bigoplus_{\t\la_1} \left(F_{\al_{p+1}=\frac{3}{2},0;\frac{\pi}{2}},F_{\al_p=\frac{1}{2},0;0,0}\right)
\oplus\bigoplus_{\t\la_1\t\la_2>0}\left(F_{\al_{p+1}=\frac{3}{2},\t\la_1;\frac{\pi}{2}},F_{\al_p=\frac{1}{2},\t\la_2;0}\right);\end{align*}

\begin{align*}
\Delta_{\mf, \rm abs}
\equiv&\bigoplus_{q=0,q\not=p,p+1}^{2p+1} \bigoplus_{\t\la_1,\t\la_2}
\left(F_{\al_q,\t\la_1;\frac{\pi}{2}},F_{\al_{2p+1-q},\t\la_2;0}\right)\\
&~\oplus \bigoplus_{\t\la_2} \left(F_{\al_p=\frac{1}{2},0;\frac{\pi}{2},\frac{\pi}{2}},F_{\al_{p+1}=\frac{3}{2},\t\la_2;0}\right)
\oplus\bigoplus_{\t\la_1\t\la_2>0}\left(F_{\al_p=\frac{1}{2},\t\la_1;\frac{\pi}{2}},F_{\al_{p+1}=\frac{3}{2},\t\la_2;0}\right)\\
&~\oplus \bigoplus_{\t\la_1} \left(F_{\al_{p+1}=\frac{3}{2},0;\frac{\pi}{2}},F_{\al_p=\frac{1}{2},0;\frac{\pi}{2},0}\right)
\oplus\bigoplus_{\t\la_1\t\la_2>0}\left(F_{\al_{p+1}=\frac{3}{2},\t\la_1;\frac{\pi}{2}},F_{\al_p=\frac{1}{2},\t\la_2;0}\right);\end{align*}
\end{enumerate}

\item $\frac{3}{2}\leq\ka< 3$:
\begin{enumerate}

\item $m=2p-1$:
\begin{align*}
\Delta_{\mf^c, \rm abs}=\Delta_{\mf, \rm abs}
\equiv&\bigoplus_{q=0,q\not=p-1,p}^{2p} \bigoplus_{\t\la_1,\t\la_2}
\left(F_{\al_q,\t\la_1;\frac{\pi}{2}},F_{\al_{2p-q},\t\la_2;0}\right)\\
&~\oplus \bigoplus_{\t\la_2} \left(F_{\al_{p-1}=0,0;0,\frac{\pi}{2}},F_{\al_{p}=1,\t\la_2;0}\right)
\oplus\bigoplus_{\t\la_1\t\la_2>0}\left(F_{\al_{p-1}=0,\t\la_1;\frac{\pi}{2}},F_{\al_{p}=1,\t\la_2;0}\right)\\
&~\oplus \bigoplus_{\t\la_1} \left(F_{\al_{p}=1,0;\frac{\pi}{2}},F_{\al_{p-1}=0,0;0,0}\right)
\oplus\bigoplus_{\t\la_1\t\la_2>0}\left(F_{\al_{p}=1,\t\la_1;\frac{\pi}{2}},F_{\al_{p-1}=0,\t\la_2;0}\right);
\end{align*}

\item $m=2p$:
\begin{align*}
\Delta_{\mf^c, \rm abs}
\equiv&\bigoplus_{q=0,q\not=p,p+1}^{2p+1} \bigoplus_{\t\la_1,\t\la_2}
\left(F_{\al_q,\t\la_1;\frac{\pi}{2}},F_{\al_{2p+1-q},\t\la_2;0}\right)\\
&~\oplus \bigoplus_{\t\la_2} \left(F_{\al_p=\frac{1}{2},0;0,\frac{\pi}{2}},F_{\al_{p+1}=\frac{3}{2},\t\la_2;0}\right)
\oplus\bigoplus_{\t\la_1\t\la_2>0}\left(F_{\al_p=\frac{1}{2},\t\la_1;\frac{\pi}{2}},F_{\al_{p+1}=\frac{3}{2},\t\la_2;0}\right)\\
&~\oplus \bigoplus_{\t\la_1} \left(F_{\al_{p+1}=\frac{3}{2},0;\frac{\pi}{2}},F_{\al_p=\frac{1}{2},0;0,0}\right)
\oplus\bigoplus_{\t\la_1\t\la_2>0}\left(F_{\al_{p+1}=\frac{3}{2},\t\la_1;\frac{\pi}{2}},F_{\al_p=\frac{1}{2},\t\la_2;0}\right);
\end{align*}

\begin{align*}
\Delta_{\mf, \rm abs}
\equiv&\bigoplus_{q=0,q\not=p,p+1}^{2p+1} \bigoplus_{\t\la_1,\t\la_2}
\left(F_{\al_q,\t\la_1;\frac{\pi}{2}},F_{\al_{2p+1-q},\t\la_2;0}\right)\\
&~\oplus \bigoplus_{\t\la_2} \left(F_{\al_p=\frac{1}{2},0;\frac{\pi}{2},\frac{\pi}{2}},F_{\al_{p+1}=\frac{3}{2},\t\la_2;0}\right)
\oplus\bigoplus_{\t\la_1\t\la_2>0}\left(F_{\al_p=\frac{1}{2},\t\la_1;\frac{\pi}{2}},F_{\al_{p+1}=\frac{3}{2},\t\la_2;0}\right)\\
&~\oplus \bigoplus_{\t\la_1} \left(F_{\al_{p+1}=\frac{3}{2},0;\frac{\pi}{2}},F_{\al_p=\frac{1}{2},0;\frac{\pi}{2},0}\right)
\oplus\bigoplus_{\t\la_1\t\la_2>0}\left(F_{\al_{p+1}=\frac{3}{2},\t\la_1;\frac{\pi}{2}},F_{\al_p=\frac{1}{2},\t\la_2;0}\right);\end{align*}

\end{enumerate}

\item $1<\ka< \frac{3}{2}$:
\begin{enumerate}

\item $m=2p-1$:
\begin{align*}
\Delta_{\mf^c, \rm abs}=\Delta_{\mf, \rm abs}
\equiv&\bigoplus_{q=0,q\not=p-1,p}^{2p} \bigoplus_{\t\la_1,\t\la_2}
\left(F_{\al_q,\t\la_1;\frac{\pi}{2}},F_{\al_{2p-q},\t\la_2;0}\right)\\
&~\oplus \bigoplus_{\t\la_2} \left(F_{\al_{p-1}=0,0;0,\frac{\pi}{2}},F_{\al_{p}=1,\t\la_2;0}\right)
\oplus\bigoplus_{\t\la_1\t\la_2>0}\left(F_{\al_{p-1}=0,\t\la_1;\frac{\pi}{2}},F_{\al_{p}=1,\t\la_2;0}\right)\\
&~\oplus \bigoplus_{\t\la_1} \left(F_{\al_{p}=1,0;\frac{\pi}{2}},F_{\al_{p-1}=0,0;0,0}\right)
\oplus\bigoplus_{\t\la_1\t\la_2>0}\left(F_{\al_{p}=1,\t\la_1;\frac{\pi}{2}},F_{\al_{p-1}=0,\t\la_2;0}\right);
\end{align*}

\item $m=2p$:
\begin{align*}
\Delta_{\mf^c, \rm abs}
\equiv&\bigoplus_{q=0,q\not=p-1,p,p+1,p+2}^{2p+1} \bigoplus_{\t\la_1,\t\la_2}
\left(F_{\al_q,\t\la_1;\frac{\pi}{2}},F_{\al_{2p+1-q},\t\la_2;0}\right)\\
&~\oplus \bigoplus_{\t\la_2} \left(F_{\al_{p-1}=-\frac{1}{2},0;0,\frac{\pi}{2}},F_{\al_{p+2}=\frac{5}{2},\t\la_2;0}\right)
\oplus\bigoplus_{\t\la_1\t\la_2>0}\left(F_{\al_{p-1}=-\frac{1}{2},\t\la_1;\frac{\pi}{2}},F_{\al_{p+2}=\frac{5}{2},\t\la_2;0}\right)\\
&~\oplus \bigoplus_{\t\la_2} \left(F_{\al_p=\frac{1}{2},0;0,\frac{\pi}{2}},F_{\al_{p+1}=\frac{3}{2},\t\la_2;0}\right)
\oplus\bigoplus_{\t\la_1\t\la_2>0}\left(F_{\al_p=\frac{1}{2},\t\la_1;\frac{\pi}{2}},F_{\al_{p+1}=\frac{3}{2},\t\la_2;0}\right)\\
&~\oplus \bigoplus_{\t\la_1} \left(F_{\al_{p+1}=\frac{3}{2},0;\frac{\pi}{2}},F_{\al_p=\frac{1}{2},0;0,0}\right)
\oplus\bigoplus_{\t\la_1\t\la_2>0}\left(F_{\al_{p+1}=\frac{3}{2},\t\la_1;\frac{\pi}{2}},F_{\al_p=\frac{1}{2},\t\la_2;0}\right)\\
&~\oplus \bigoplus_{\t\la_1} \left(F_{\al_{p+2}=\frac{5}{2},0;\frac{\pi}{2}},F_{\al_{p-1}=-\frac{1}{2},0;0, 0}\right)
\oplus\bigoplus_{\t\la_1\t\la_2>0}\left(F_{\al_{p+2}=\frac{5}{2},\t\la_1;\frac{\pi}{2}},F_{\al_{p-1}=-\frac{1}{2},\t\la_2;0}\right).
\end{align*}

\begin{align*}
\Delta_{\mf, \rm abs}
\equiv&\bigoplus_{q=0,q\not=p-1,p,p+1,p+2}^{2p+1} \bigoplus_{\t\la_1,\t\la_2}
\left(F_{\al_q,\t\la_1;\frac{\pi}{2}},F_{\al_{2p+1-q},\t\la_2;0}\right)\\
&~\oplus \bigoplus_{\t\la_2} \left(F_{\al_{p-1}=-\frac{1}{2},0;0,\frac{\pi}{2}},F_{\al_{p+2}=\frac{5}{2},\t\la_2;0}\right)
\oplus\bigoplus_{\t\la_1\t\la_2>0}\left(F_{\al_{p-1}=-\frac{1}{2},\t\la_1;\frac{\pi}{2}},F_{\al_{p+2}=\frac{5}{2},\t\la_2;0}\right)\\
&~\oplus \bigoplus_{\t\la_2} \left(F_{\al_p=\frac{1}{2},0;\frac{\pi}{2},\frac{\pi}{2}},F_{\al_{p+1}=\frac{3}{2},\t\la_2;0}\right)
\oplus\bigoplus_{\t\la_1\t\la_2>0}\left(F_{\al_p=\frac{1}{2},\t\la_1;\frac{\pi}{2}},F_{\al_{p+1}=\frac{3}{2},\t\la_2;0}\right)\\
&~\oplus \bigoplus_{\t\la_1} \left(F_{\al_{p+1}=\frac{3}{2},0;\frac{\pi}{2}},F_{\al_p=\frac{1}{2},0;\frac{\pi}{2},0}\right)
\oplus\bigoplus_{\t\la_1\t\la_2>0}\left(F_{\al_{p+1}=\frac{3}{2},\t\la_1;\frac{\pi}{2}},F_{\al_p=\frac{1}{2},\t\la_2;0}\right)\\
&~\oplus \bigoplus_{\t\la_1} \left(F_{\al_{p+2}=\frac{5}{2},0;\frac{\pi}{2}},F_{\al_{p-1}=-\frac{1}{2},0;0,0}\right)
\oplus\bigoplus_{\t\la_1\t\la_2>0}\left(F_{\al_{p+2}=\frac{5}{2},\t\la_1;\frac{\pi}{2}},F_{\al_{p-1}=-\frac{1}{2},\t\la_2;0}\right).
\end{align*}

\end{enumerate}

\end{enumerate}

In order to have the relative boundary conditions at $\b X$, we just need to change the last index, i.e. for example to write $(F_{\cdot, \cdot;0}, F_{\cdot,\cdot;\frac{\pi}{2}})$ in place of $(F_{\cdot, \cdot;\frac{\pi}{2}}, F_{\cdot,\cdot;0})$.

\end{rem}

\subsection{Harmonic forms and harmonic fields for the operators $\Delta_{\pf,\rm bc}$}

\begin{lem}  We have
$
\ker \Delta^{(q)}_{\rm max}=\{\psi_{A,\pm}\in L^2(\ZZ)\}
$, $A=E,O,I,II,III, IV$, and the forms $\psi_{A,+}$ are the forms of type + according to the description  in Proposition \ref{harmonics}, and the forms of type - are those described in Lemma \ref{squareint}. 
\end{lem}

\begin{proof} The kernel of the maximal extension is determined by the formal solution of the harmonic equation described in Lemma \ref{harmonics}, by imposing square integrability. 
\end{proof}

\begin{theo}\label{harmonicform}  Let $\H^q(W)=\langle \t\vv^{(q)}_{{\rm har}, j}\rangle$ be the space of the  harmonic fields of the Hodge Laplace operator $\t\Delta^{(q)}$ on the section. Then, the space of the harmonic forms of the operator $\Delta^{(q)}_{\pf,\rm bc}$ on the finite horn $\ZZ$ is given as follows. If $(m,\pf)=(2p-1,\pf),(2p,\mf^c)$, then:
\begin{align*}
\ker \Delta^{(q)}_{\pf,\rm abs}&=\left\{\begin{array}{ll} \langle \t\vv^{(q)}_{{\rm har}, j}\rangle,&0\leq q\leq p-1,\\ \{0\}, &p\leq q\leq m+1,\end{array}\right.\\
\ker \Delta^{(q)}_{\pf,\rm rel}&=\left\{\begin{array}{ll} \{0\},&0\leq q\leq p,\\ \langle h^{2\al_{q-1}-1}dx\wedge\t\vv^{(q-1)}_{{\rm har}, j}\rangle, &p+1\leq q\leq m+1.\end{array}\right.
\end{align*}

If $(m,\pf)=(2p,\mf)$, then:
\begin{align*}
\ker \Delta^{(q)}_{\mf,\rm abs}&=\left\{\begin{array}{ll} \langle \t\vv^{(q)}_{{\rm har}, j}\rangle,&0\leq q\leq p,\\ \{0\}, &p+1\leq q\leq 2p+1,\end{array}\right.\\
\ker \Delta^{(q)}_{\mf,\rm rel}&=\left\{\begin{array}{ll} \{0\},&0\leq q\leq p+1,\\ \langle h^{2\al_{q-1}-1}dx\wedge\t\vv^{(q-1)}_{{\rm har}, j}\rangle, &p+2\leq q\leq 2p+1.\end{array}\right.
\end{align*}

where the inclusion is  the extension by the constant function.
\end{theo}
\begin{proof} According to the description in Lemma \ref{harmonics}, the solution of the formal harmonic equation correspond to pair of solutions of the harmonic equations of a pair of Sturm 
Liouville operators $ (\FF_{1,\al_q,\t\la_1},\FF_{1,\al_{m+1-q},\t\la_2})$. Whence, the harmonic forms will be given by the harmonic forms of the corresponding pair of self adjoint extensions $(F_{\al_q,\t\la_{1};\be_1,\ga_1},F_{\al_{m+1-q},\t\la_{2};\be_1,\ga_2})$, as described in  Remark \ref{formuleLap}. The harmonic forms of all these extension are described in Theorem \ref{pos}. According to that proposition, the unique possibility is given when the solution $f$ in the line is constant, and the eigenvalue $\t\la_{1}=\t\la_{2}=0$. Whence, the possible harmonic forms are necessarily constant in the $x$ coordinate. and correspond to the zero eigenvalue on the section. According to Lemma \ref{harmonics} this means that we may have harmonic forms only of types E and O.

\begin{itemize}

\item[Type O] Let $\psi_O^{(q)}=fdx\wedge \t\vv^{(q-1)}_{\rm har}$ be a solution of type O. Then, $\FF^q_{2,0}f\equiv\FF_{\al_{m+1-q},0}f=0$. By Remark \ref{formuleLap} the concrete operator is  of type either $F_{\al,0;0}$ or $F_{\al, 0;0,0}$, with some $\al\not=\frac{1}{2}$, beside the particular cases where $\al=\frac{1}{2}$, that reduces to the regular case $F_{\frac{1}{2},0;\de,0}$, $\de=0,\frac{\pi}{2}$. By Theorem \ref{pos}, the kernel of all these operators is  trivial, so there are no harmonic forms of type O.

\item[Type E] Let $\psi_E^q=f\t\vv^{(q)}_{\rm har}$ be a solution of type E. Then, $\FF^q_{1,0}f=\FF_{\al_q,0}f=0$. By Remark \ref{formuleLap} the concrete operator is  of type either $F_{\al,0;\frac{\pi}{2}}$ or $F_{\al, 0;0,\frac{\pi}{2}}$, with some $\al\not=\frac{1}{2}$, beside the particular cases where $\al=\frac{1}{2}$, that reduces to the regular case $F_{\frac{1}{2},0;\de,\frac{\pi}{2}}$, $\de=0,\frac{\pi}{2}$. We proceed case by case according to enumeration in Remark \ref{formuleLap}, and suing the result for the kernel of the relevant operator given in Theorem \ref{pos}. The solutions of type E appear only in the following cases: 1) $3\leq \ka$, 2) $\frac{3}{2}\leq \ka<3$, and 3) $1< \ka<\frac{3}{2}$. Since the argument is always the same, we give details only for the hardest case 3). So let $1< \ka<\frac{3}{2}$. Then:

\begin{enumerate}

\item $m=2p-1$: 

\begin{enumerate}

\item $q\not=p-1$, the operator is $F_{\al_q,0;\frac{\pi}{2}}$ whose kernel is generated by the constant function if $\al_q<\frac{1}{2}$, i.e. $q\leq p-1$, and trivial otherwise,

\item $q=p-1$, the operator is $F_{0,0;0,\frac{\pi}{2}}$ whose kernel is generated by the constant function,

\end{enumerate}

\item $m=2p$: 

\begin{enumerate}

\item $\pf=\mf^c$:

\begin{enumerate}

\item $q\not=p$, the operator is $F_{\al_q,0;\frac{\pi}{2}}$ whose kernel is generated by the constant function if $\al_q<\frac{1}{2}$, i.e. $q\leq p-1$, and trivial otherwise,

\item $q=p$, the operator is $F_{\frac{1}{2},0;0, \frac{\pi}{2}}$ whose kernel is trivial,

\end{enumerate}

\item $\pf=\mf$:

\begin{enumerate}

\item $q\not=p$, the operator is $F_{\al_q,0;\frac{\pi}{2}}$ whose kernel is generated by the constant function if $\al_q<\frac{1}{2}$, i.e. $q\leq p-1$, and trivial otherwise,

\item $q=p$, the operator is $F_{\frac{1}{2},0;\frac{\pi}{2}, \frac{\pi}{2}}$ whose kernel is generated by the constant function,

\end{enumerate}

\end{enumerate}

\end{enumerate}

Collecting, we see that the kernel does not depend on $\ka$, and is always generated by the constant extension of the restriction to the boundary. This concludes the proof.


\end{itemize}

\end{proof}

\subsection{Eigenfunctions and spectrum for the operators $\Delta_{\pf,\rm bc}$}

\begin{prop}\label{l4}  Denote by $\ff_{{\alpha_q}, \t\la_{q,n},\pm}(x,\lambda)$  the two linearly independent solutions of the differential equation $\FF^q_{1,\t\la_{q,n}} f=\la f$ as described and normalised in Propositions \ref{olv1} and \ref{b0}. Denote by $m_{{\rm cex}, q, n}$  the rank of $\t\E^{(q)}_{\t\la_{q,n},\rm cex}$. When $\al\not=\frac{1}{2}$, denote by  $\ell_{\alpha,\t\la, k}$  the non vanishing  zeros of the  function 
$\ff_{\alpha,\t\la,+}(l,\lambda)$, and  by
$\hat \ell_{\alpha,\t\la,k}$  the non vanishing  zeros of the function $\ff_{\alpha,\t\la,+}'(l,\lambda)$. 
 When  $\al=\frac{1}{2}$, $\ell_{\frac{1}{2},0,+,k}=\left(\frac{\pi}{l}{k}\right)^2$, $\ell_{\frac{1}{2},-,k}=\left(\frac{\pi}{2l}(2k-1)\right)^2$. Then, the positive part of the spectrum of the intersection Hodge-Laplace operator  $\Delta^{(q)}_{ \pf, \rm abs}$  on $C_{0,l}(W)$, with absolute boundary conditions on $\b C_{0,l} (W)$ is as follows. If $m=2p-1$:
\begin{align*}
\Sp_+ \Delta_{ \mf, \rm abs}^{(q)} =&\Sp_+ \Delta_{  \mf^c, \rm abs}^{(q)}\\
=& \left\{m_{{\rm cex},q,n} : \hat 
\ell_{\alpha_q,\t\la_{q,n},k} \right\}_{n,k=1}^{\infty}
\cup
\left\{m_{{\rm cex},q-1,n} : \hat 
\ell_{\alpha_{q-1},\t\la_{q-1,n}k}\right\}_{n,k=1}^{\infty} \\
&\cup \left\{m_{{\rm cex},q-1,n} : 
\ell_{-\alpha_{q-1},\t\la_{q-1,n},k}\right\}_{n,k=1}^{\infty} \cup \left\{m 
_{{\rm cex},q-2,n} :
\ell_{-\alpha_{q-2},\t\la_{q-2,n},k}\right\}_{n,k=1}^{\infty} \\
&\cup \left\{m_{{\rm har},q}: 
\ell_{-\alpha_{q-1},0,k}\right\}_{k=1}^{\infty} \cup \left\{ 
m_{{\rm har},q-1}:
\ell_{-\alpha_{q-2},0,k}\right\}_{k=1}^{\infty}.
\end{align*}

If $m=2p$,  then $\Sp_+ \Delta_{  \mf, \rm abs}^{(q)} =\Sp_+ \Delta_{  \mf^c, \rm abs}^{(q)}$ is as in the odd case for all the eigenvalues with $(\al_q,\t\la_{q,n})\not=(\frac{1}{2},0)$, i.e. for $(m,q)\not=(2p, p), (2p,p+1)$.  When $(\al_q,\t\la_{q,0})=(\frac{1}{2},0)$ (i.e. $m=2p$, and $q=p$ and $q=p+1$), the sets of the eigenvalues  in the last line, i.e. those associated to the harmonics of the section,  are given as follows:
\begin{itemize}

\item in $\Sp_+ \Delta_{  \mf, \rm abs}^{(p)}$:   
$
\left\{ m_{{\rm har},p}: \ell_{\frac{1}{2},0,+,k}\right\}_{k=1}^{\infty}\cup \left\{ 
m_{{\rm har},p-1}:
\ell_{-\alpha_{p-2},0,k}\right\}_{k=1}^{\infty}
$;

\item in  $\Sp_+ \Delta_{  \mf, \rm abs}^{(p+1)}$:
$
\left\{m_{{\rm har},p+1}: 
\ell_{-\alpha_{p},0,k}\right\}_{k=1}^{\infty} \cup \left\{ m_{{\rm har},p}: \ell_{\frac{1}{2},0,+,k}\right\}_{k=1}^{\infty}
$;

\item in $\Sp_+ \Delta_{  \mf^c, \rm abs}^{(p)}$:   
$
\left\{ m_{{\rm har},p}: \ell_{\frac{1}{2},0,-,k}\right\}_{k=1}^{\infty}\cup \left\{ 
m_{{\rm har},p-1}:
\ell_{-\alpha_{p-2},0,k}\right\}_{k=1}^{\infty}
$;

\item in  $\Sp_+ \Delta_{  \mf^c, \rm abs}^{(p+1)}$:
$
\left\{m_{{\rm har},p+1}: 
\ell_{-\alpha_{p},0,k}\right\}_{k=1}^{\infty} \cup \left\{ m_{{\rm har},p}: \ell_{\frac{1}{2},0,-,k}\right\}_{k=1}^{\infty}
$.

\end{itemize}

\end{prop}
\begin{proof} The result follows by considering all those formal solutions  of the eigenvalues equation as described in Lemma \ref{formaleigen} that belong to the domain of the relevant operator: for they need to be  square integrable with their exterior derivative and their dual exterior derivative,  and to satisfy the suitable boundary conditions  according to  Section \ref{defLap}. We discussed square integrability in Subsection \ref{squareint}. We realise that all the + solutions are square integrable, while   square integrable - solutions are only of types E and O and only appear in some degrees in even dimension $m=2p$.  So consider the bc.  By the behaviour of $\ff_{\al,\t\la,+}$ (Proposition \ref{olv1} and \ref{b0}), all + solutions satisfy the bc at the tip of the horn. Consider then the bc at the boundary of the horn. 
First, assume  $m=2p-1$.  Applying the boundary condition at $x=l$ to a solution of type I+, we have the equation 
$
\ff_{\alpha,\t\la,+}'(x,\lambda)|_{x=l}=0,
$ 
and this gives the first set of eigenvalues $\hat \ell_{\alpha_q,\t\la_{q,n},k} $, $n,k=1, 2, \dots$. Similarly, the solutions of type 
II+ give the  eigenvalues $\hat \ell_{\alpha_q,\t\la_{q,n},k} $, $n,k=1, 2, \dots$, those of type III+  the eigenvalues 
$ \ell_{-\alpha_{q-1},\t\la_{q-1,n},k} $, $n,k=1,2,\dots$, and those of type IV+ the eigenvalues $\ell_{-\alpha_{q-2},\t\la_{q-2,n},k}$. The solutions of type E+ give the eigenvalues  $ \ell_{-\alpha_{q-1},0,k}$, $n,k=1,2,\dots$, and those of type O+ the eigenvalues $\ell_{-\alpha_{q-2},0,k}$, $n,k=1,2,\dots$. 
Next,  assume $m = 2p$.  In degree $p$,   the forms of type E are 
$\psi_{E,+}(x,\la)=\frac{1}{\sqrt{\la}}\sin(\sqrt{\la}x)\t\vv^{(p)}_{\rm har}$, and $\psi_{E,-}(x,\la)=\cos(\sqrt{\la} x)\t\vv^{(p)}_{\rm har}$, 
both square integrable. The forms E+ satisfy the boundary condition at $x=0$ for $\Delta_{\mf^c, \rm abs}^{(p)}$, and applying the boundary condition at $x=l$  we obtain the set  $\ell_{\frac{1}{2},0,-,k}$, $k=1,2,\dots$, the forms E- do not satisfy the boundary condition at $x=0$ for $\Delta^{(p)}_{ \mf^c, \rm abs}$. Similarly, the forms E- satisfy the boundary condition at $x=0$ for $\Delta^{(p)}_{ \mf, \rm abs}$, and applying the boundary condition at $x=l$  we obtain the set  $\ell_{\frac{1}{2},0,+,k}$, $k=1,2,\dots$. In degree $p+1$,   the forms of type O are 
$
\psi_{O,+}(x,\la)=\cos(\sqrt{\la}x)dx\wedge\t\vv^{(p)}_{\rm har}$, $\psi_{O,-}(x,\la)=\frac{1}{\sqrt{\la}}\sin(\sqrt{\la} x)dx\wedge\t\vv^{(p)}_{\rm har}
$, 
both square integrable. The forms O+ satisfy the boundary condition at $x=0$ for $\Delta_{\mf^c, \rm abs}^{(p+1)}$, and applying the boundary condition at $x=l$  we obtain the set  $\ell_{\frac{1}{2},0,+,k}$, $k=1,2,\dots$, the forms O- do not satisfy the boundary condition at $x=0$ for $\Delta^{(p+1)}_{ \mf^c, \rm abs}$. Similarly, the forms O- satisfy the boundary condition at $x=0$ for $\Delta^{(p+1)}_{ \mf, \rm abs}$, and applying the boundary condition at $x=l$  we obtain the set  $\ell_{\frac{1}{2},0,-,k}$, $k=1,2,\dots$. 
\end{proof}

\begin{rem}\label{duality1} Using Poincar\`e duality on the section: $\tilde \la_{q,n} = 
\tilde \la_{(m+1-q)-2,n}$ and $\al_q = -\al_{(m+1-q)-2}$, we find out that 
$\Sp_+ \Delta^{q}_{\pf, \rm abs} = \Sp_+ \Delta^{m+1-q}_{\pf^c, \rm rel}$. 
\end{rem}

\begin{prop}\label{eigenforms} With the notation of Proposition \ref{l4},  the   eigenforms  of the  operator $\Delta_{\pf,\dr}^{(q)}$ are given as follows. In dimension $m=2p-1$,
\begin{align*}
\Ei_+(\Delta^{(q)}_{\pf, \rm abs})
&=\left\{ \psi^{(q)}_{E, +}(x,\ell_{-\alpha_{q-1},0,k}),
\psi^{(q)}_{O,  +}(x,\ell_{-\alpha_{q-2},0,k}),\right.\\
& \hspace{20pt}\psi^{(q)}_{I, \tilde\lambda_{q,n}, +}(x,\hat \ell_{\alpha_q,\t\la_{q,n},k}), 
\psi^{(q)}_{II, \tilde\lambda_{q-1,n}, +}(x,\hat \ell_{\alpha_{q-1},\t\la_{q-1,n},k} ),\\
&\hspace{20pt}\left.
\psi^{(q)}_{III, \tilde\lambda_{q-1,n},+}(x, \ell_{-\alpha_{q-1},\t\la_{q-1,n},k} ), \psi^{(q)}_{IV, \tilde\lambda_{q-2,n},+}(x,\ell_{-\alpha_{q-2},\t\la_{q-2,n},k}).
\right\},
\end{align*}

In dimension $m=2p$, the eigenforms are the same as in the odd case ($m=2p-1$) in all degrees $q\not=p,p+1$. In the remaining cases, i.e. $m=2p$ and either $q=p$ or $q=p+1$, the eigenforms of types I,II, III, and IV are the same as in the odd case with the appropriate value of $q$, while the eigenforms of type E and O are as follows:
\begin{enumerate}

\item forms of type E and O in $\Ei_+(\Delta^{(p)}_{\mf, \rm abs})$: 
$\psi^{(p)}_{E,-}(x)=\cos \left(\frac{\pi k }{l} x\right) \t\vv^{(p)}_{\rm har}$,  and  \\
$\psi^{(p)}_{O,  +}(x,\ell_{-\alpha_{p-2},0,k})$; 
\item forms of type E and O in $\Ei_+(\Delta^{(p+1)}_{\mf, \rm abs})$: 
$\psi^{(p+1)}_{E, +}(x,\ell_{-\alpha_{p+1},0,k})$, and 
$\psi^{(p+1)}_{O,-}(x)=\sin \left(\frac{k \pi  }{ l} x\right) dx\wedge \t\vv^{(p)}_{\rm har}$; 

\item forms of type E and O in $\Ei_+(\Delta^{(p)}_{\mf^c, \rm abs})$:
$\psi^{(p)}_{E,+}(x)=\sin \left(\frac{(2k-1)\pi  }{2 l} x\right) \t\vv^{(p)}_{\rm har}$, and $ \psi^{(p)}_{O,  +}(x,\ell_{-\alpha_{p-2},0,k})$;

\item forms of type E and O in $\Ei_+(\Delta^{(p+1)}_{\mf^c, \rm abs})$:
$\psi^{(p+1)}_{E, +}(x,\ell_{-\alpha_{p+1},0, k})$, and 
$\psi^{(p+1)}_{O,+}(x)=\cos\left(\frac{(2k-1)\pi}{2l}x\right) dx\wedge \t\vv^{(p)}_{\rm har}$. 

\end{enumerate}

\end{prop}

\subsection{Spectral properties of the operators $\Delta_{\pf, \rm bc}$}

\begin{prop}\label{compresC} The operators $\Delta^q_{\pf,\rm bc}$ have compact resolvent.
\end{prop}
\begin{proof} Consider the single term  $\Delta^{(q)}_{\t\la_1,j_1,\t\la_2,j_2}$ in the decomposition described above, equation (\ref{pro}). This consists in  a pair of Sturm Liouville operators, whose kernel is given in Proposition \ref{c3.34}.  Direct evaluation of the integral using the behaviour of the solutions of the eigenvalues equations (compare with Corollary \ref{c3.36}), shows that these kernels are square integrable. It follows that the operator $\Delta^{(q)}_{\t\la_1,j_1,\t\la_2,j_2}$ has compact resolvent. Now, observe that 
the sequence of operators 
\[
\Delta^{(q)}_n=\bigoplus_{\t\la_1,\t\la_2=0}^{\t\la_n} \bigoplus_{j_1,j_2} \Delta^{(q)}_{\t\la_1,j_1,\t\la_2,j_2},
\]
converges to the operator $\Delta_{\pf,\rm bc}^{(q)}$    (in the $L^2$ norm),  i.e. $\left\| \Delta_{\pf,\rm bc} -\Delta^{(q)}_n\right\|\to 0$.  
But the resolvent of each of the operator $\Delta_n$ is a bounded finite rank operator. 
Then, compactness of $\Delta_{\pf,\rm bc}^{(q)}$ follows for example by \cite[Theorem 6.5]{Wei}.  
\end{proof}

\begin{prop}\label{resolutionC} The family of the eigenvalues described in Proposition \ref{l4} and the associated eigenforms described in Proposition \ref{eigenforms}, together with the harmonics described in Proposition \ref{harmonicform} determines a spectral resolution for the operator $\Delta^{(q)}_{\pf,\rm bc}$.
\end{prop}
\begin{proof} This follows for example by \cite[]{Sch}, since the resolvent is compact by Proposition \ref{compresC}. Alternatively, observe that  the projection $\Delta^{(q)}_{\t\la_1,\t\la_2}$ (see equation (\ref{pro})) of the Hodge Laplace operator on the finite dimensional space associated to the  pair of eigenvalues $(\t\la_1,\t\la_2)$ of the section has a spectral resolution by Corollary \ref{c3.36}. Then, taking the tensor product with the eigenforms of the spectral resolution of the Hodge Laplace operator on the section, we obtain a complete orthogonal basis for the space of the square integrable forms on the cone consisting in eigenforms of the Hodge Laplace operator, as desired. Note that all these forms are smooth. 
\end{proof}

\subsection{Harmonic forms and harmonic fields for the operators $\Delta_\d$, $\Delta_{\dr}$, and $\Delta_{ \d,\pf, \rm rel}$, $\Delta_{\dr, \pf, \rm abs}$}  

Note that harmonic forms and harmonic fields coincides for the operators $\Delta_\d$, $\Delta_{\dr}$, and $\Delta_{ \d,\pf, \rm rel}$, $\Delta_{\dr, \pf, \rm abs}$, see Corollary \ref{eqker}.

\begin{prop}\label{PPP} We have the identifications: 
if $m=2p-1$,  then
\begin{align*}
\ker \d_{ \rm max}^{q}\cap \ker \dr_{\rm max}^{q}&= L_{XY}^q\oplus \left\{\begin{array}{lc}\{\t\theta\in\H^{q}(W) \},&0\leq q\leq p-1, \\
\{0\},& q=p,\\
\{h^{2\al_{q-1}-1}dx\wedge\t\theta~|~\t\te\in\H^{q-1}(W) \},&p+1\leq q\leq 2p;
\end{array}\right. \\
\ker \Delta_{\d}&=\ker \d_{ \rm min}^{q}\cap \ker \dr_{ \rm max}^{q}=\ker \d_{\pf,  \rm rel}^{q}\cap \ker \dr_{\rm max}^{q}\\
&=\left\{\begin{array}{lc}
\{0\},& 0\leq q\leq p,\\
\{h^{2\al_{q-1}-1}dx\wedge\t\theta~|~\t\te\in\H^{q-1}(W) \},&p+1\leq q\leq 2p;
\end{array}\right.\\
\ker \Delta_{\dr}&=\ker \dr_{ \rm min}^{q}\cap \ker \d_{ \rm max}^{q}=\ker \d_{ \rm max}^{q}\cap \ker \dr_{\pf, \rm abs}^{q}\\
&=\left\{\begin{array}{lc}\{\t\theta\in\H^{q}(W) \},&0\leq q\leq p-1, \\
\{0\},& p\leq q\leq 2p;
\end{array}\right. 
\end{align*}
if $m=2p$, then
\begin{align*}
\ker \d_{ \rm max}^{q}\cap \ker \dr_{\rm max}^{q}&= L_{XY}^q\oplus \left\{\begin{array}{lc}\{\t\theta\in\H^{q}(W) \},&0\leq q\leq p, \\
\{h^{2\al_{q-1}-1}dx\wedge\t\theta~|~\t\te\in\H^{q-1}(W) \},&p+1\leq q\leq 2p+1;
\end{array}\right. \\
\ker \Delta_{\d}&=\ker \d_{ \rm min}^{q}\cap \ker \dr_{ \rm max}^{q}=\ker \d_{\pf,  \rm rel}^{q}\cap \ker \dr_{\rm max}^{q}\\
&=\left\{\begin{array}{lc}
\{0\},& 0\leq q\leq p,\\
\{h^{2\al_{q-1}-1}dx\wedge\t\theta~|~\t\te\in\H^{q-1}(W) \},&p+1\leq q\leq 2p+1;
\end{array}\right. \\
\ker \Delta_{\dr}&=\ker \dr_{ \rm min}^{q}\cap \ker \d_{ \rm max}^{q}=\ker \d_{ \rm max}^{q}\cap \ker \dr_{\pf, \rm abs}^{q}\\
&=\left\{\begin{array}{lc}\{\t\theta\in\H^{q}(W) \},&0\leq q\leq p, \\
\{0\},& p+1\leq q\leq 2p+1;
\end{array}\right. 
\end{align*}
where
\begin{align*}
L_{XY}^q=&\left\{f\t D\t\om+dx\wedge f^I\t\om~|~ \t\om\in\Im \t\DD^{q}, f\in AC_{\rm loc}((0,l]),\right.\\
&\left. f\in L^2((0,l],h^{1-2\al_q}), f^I\in L^2((0,l],h^{1-2\al_{q-1}}), \FF_{\al_{q-1}, \t\la}f=0\right\}.
\end{align*}

Moreover, if $m=2p-1$, then 
\begin{align*}
\ker \Delta_{\d,\pf,\rm rel}=\ker \d_{\pf,  \rm rel}^{q}\cap \ker \dr_{\pf, \rm rel}^{q}&=\ker \d_{\pf, \rm rel}^{q}\cap \ker \dr_{\rm max}^{q},\\
\ker \Delta_{\dr,\pf,\rm abs}= \ker \d_{\pf, \rm abs}^{q}\cap \ker \dr_{\pf, \rm abs}^{q}&=\ker \d_{\rm max}^{q}\cap \ker \dr_{\pf, \rm abs}^{q};
\end{align*} 
while if $m=2p$, then 
\begin{align*}
\ker \Delta_{\d,\mf,\rm rel}=\ker \d_{\mf,  \rm rel}^{q}\cap \ker \dr_{\mf, \rm  rel}^{q}&=\left\{\begin{array}{lc}
\{0\},& 0\leq q\leq p+1,\\
\{h^{2\al_{q-1}-1}dx\wedge\t\theta~|~\t\te\in\H^{q-1}(W) \},&p+2\leq q\leq 2p+1;
\end{array}\right. \\
\ker \Delta_{\dr,\mf,\rm abs}=\ker \d_{\mf, \rm abs}^{q}\cap \ker \dr_{\mf, \rm abs}^{q}&=
\ker \d_{\rm max}^{q}\cap \ker \dr_{\mf, \rm abs}^{q},
\end{align*}
and
\begin{align*}
\ker \Delta_{\d,\mf^c,\rm rel}=\ker \d_{\mf^c,  \rm rel}^{q}\cap \ker \dr_{\mf^c, \rm  rel}^{q}&=\ker \d_{\mf^c, \rm rel}^{q}\cap \ker \dr_{\rm max}^{q},\\
\ker \Delta_{\dr,\mf^c,\rm abs}=\ker \d_{\mf^c, \rm abs}^{q}\cap \ker \dr_{\mf^c, \rm abs}^{q}&=\left\{\begin{array}{lc}\{\t\theta\in\H^{q}(W) \},&0\leq q\leq p-1, \\
\{0\},& p\leq q\leq 2p+1.
\end{array}\right. 
\end{align*}

\end{prop}

\begin{proof}   By comparing the description of the kernels given in Propositions \ref{kerd} and \ref{kerde}, respectively, we immediately obtain
$
\ker \d_{ \rm max}^{q}\cap \ker \dr_{\rm max}^{q}=K_E\oplus H_O\oplus L_{XY},
$ 
with
\begin{align*}
K^q_E=&\{\t\theta\in\H^{q}(W)~|~ 1\in L^2(h^{1-2\al_q}) \},\\
H_O^q=&\{h^{2\al_{q-1}-1}dx\wedge\t\theta~|~\t\te\in\H^{q-1}(W)~|~h^{2\al_{q-1}-1}\in L^2(h^{1-2\al_{q-1}}) \},\\
L_{XY}^q=&K_X^q\cap H_Y^q.
\end{align*}

By taking into account the square integrability of the forms of types E and O, we obtain the two different formulas depending on the parity of the dimension $m$. By Remark \ref{XY}, the space $L_{XY}$ consists in the forms 
\[
\om=h^{2\al_{q-1}+1} (h^{1-2\al_{q-1}}f^I)^I \t\vv+dx\wedge f^I\t \DD\t\vv=f\t D\t\DD\t\vv+dx\wedge f^I\t\DD\vv,
\]
where $f$ is a solution of the equation
$
\FF_{\al_{q-1}, \t\la}f=0
$, and therefore it is smooth.   
This concludes the proof of the formulas for the maximal operators.

Next, we show that if $\om\in L_{XY}$ belongs either to the domain of $\d_{\rm min}$, $\dr_{\rm min}$, $\d_{\pf,\rm rel}$ or $\dr_{\pf, \rm abs}$,  
then  $\om=0$. For note that 
\begin{align*}
0=\int_0^l h^{1-2\al_{q-1}}(\FF_{\al_{q-1},\t\la}f) f &=[h^{1-2\al_{q-1}} f f']_0^l+\int_0^l h^{1-2\al_{q-1}}{f'}^2+\t\la^2\int_0^l h^{-1-2\al_{q-1}} f^2.
\end{align*}

Since $\t\la\not=0$, the behaviour of $f$ near $x=0$ is given (taking the + solution) by Theorem \ref{olv1}., and therefore 
\[
\lim_{x\to 0^+} (h^{1-2\al_{q-1}} \ff_{\al_{q-1},\t\la,+} \ff_{\al_{q-1},\t\la,+}')(x)=0.
\]

Whence the boundary term vanishes at $x=0$. Vanishing at $x=l$ follows by the bc at $x=l$ (relative bc implies that $f(l)=0$, absolute bc that $f'(l)=0$). This shows that $f'=0$, so $\t\la=0$, and therefore $\om=0$. 

The last point are the  bc on the other forms. On one side, the bc of $\d_{\pf,\rm rel}$ and $\d_{\rm  min}$ or $\dr_{\pf, \rm abs}$ and $\dr_{\rm min}$ at $x=l$  excludes respectively the E and the O solutions, accordingly to the given formulas. On the other, these forms trivially  verify  the bc at $x=0$, so this concludes the proof of the formulas. 

To conclude consider $\ker \d_{\pf,  \rm rel}^{q}\cap \ker \dr_{\pf, \rm rel}^{q}$. Take $\al\in \ker \d_{\pf, \rm rel}^{q}\cap \ker \dr_{\rm max}^{q}$, then $\al$ must verify the boundary conditions in the definition of $\dr_{\pf, \rm rel}$. This means 
\[
bv^{(q)}_{\rm abs}(x_0)(\al)=\lim_{x\to 0^+} (h^{1-2\al_{q-1}}f_1 a_2)(x)(\t\om_1,\t\al_2)_W,
\]
for all $ \om=f_1\t\om_1+dx\wedge f_2\t\om_2\in \DS(\d^{(q-1)}_{\rm max})$, where  $\al=a_1(x)\tilde\al_1(y)+a_2(x)dx\wedge \tilde\al_2(y)$, in the suitable degrees. In such degrees, $\al=h^{2\al_{q-1}-1} dx \wedge \t\te$, so we have
\[
bv^{(q)}_{\rm abs}(x_0)(\al)=\lim_{x\to 0^+} f_1(x)(\t\om_1,\t\te)_W.
\]

The behaviour of $f_1$ near $x=0$ is given by Lemma \ref{lem1}(2).  Whence, the limit is zero for all $\al_{q-1}\not=\frac{1}{2}$. The unique case where the limit may be non zero are thus when $m=2p$, and $q=p+1$. However,  if $(\pf,m)=(\mf^c,2p)$, then the bc is required only when $q\geq p+2$, and therefore not required for $\al_{q-1}=\frac{1}{2}$, and  $\al\in D(\dr_{\mf^c, \rm rel})$. The last case is    $(\pf,m)=(\mf,2p)$: then the bc is required when $q\geq p+1$, that gives $\al_{q-1}\geq \frac{1}{2}$. Whence, if  $\al_{q-1}=\frac{1}{2}$, i.e. $q=p+1$, then $f_1(x)$ is a non zero constant, and therefore the limit is not zero, and the bc is not satisfied, and $\al$ does not belong to $\ker \d_{\mf,  \rm rel}^{q}\cap \ker \dr_{\mf, \rm rel}^{q}$. This completes the proof for this identification. The proof of the other one is analogous.  
\end{proof}

In the following corollary we introduce a unified notation for the space of the harmonic fields.

\begin{corol} \label{eqker} 
\begin{align*}
\Ha^q_\pf(\ZZ)=\ker \Delta^{(q)}_{\pf,\rm abs}=\ker \Delta^{(q)}_{\dr, \pf,\rm abs}&=\ker \d_{\pf, \rm abs}^{q}\cap \ker \dr_{\pf, \rm abs}^{q},\\
\Ha^q_\pf(\ZZ, W)=\ker \Delta^{(q)}_{\pf,\rm rel}=\ker \Delta^{(q)}_{\d, \pf,\rm rel}&=\ker \d_{\pf, \rm rel}^{q}\cap \ker \dr_{\pf, \rm rel}^{q},\\
\Ha^q(\ZZ)=\ker \Delta^{(q)}_{\d}&=\ker \d_{ \rm min}^{q}\cap \ker \dr_{ \rm max}^{q},\\
\Ha^q(\ZZ,W)=\ker \Delta^{(q)}_{\dr}&=\ker \d_{ \rm max}^{q}\cap \ker \dr_{ \rm min}^{q}.
\end{align*}

\end{corol}

The next result is a corollary of the analysis of this section.

\begin{theo}\label{equiv} The pair of operators ($\Delta^{(q)}_{\d,\pf,\rm rel}$,  $\Delta^{(q)}_{\pf,\rm rel}$) and ($\Delta^{(q)}_{\dr,\pf,\rm abs}$, $\Delta^{(q)}_{\pf,\rm abs}$) have the same spectral resolutions, and therefore coincide.
\end{theo}
\begin{proof} The result follows by taking the spectral resolutions of the operators $\Delta^{(q)}_{\pf,\rm bc}$ described above, Proposition \ref{resolutionC},  and verifying that the eigenforms belong to the domains of the operators $\Delta^{(q)}_{\d,\pf,\rm rel}$ and $\Delta^{(q)}_{\dr,\pf,\rm abs}$, using the description given in Proposition \ref{domLap}. The equivalence of the kernels is in Corollary \ref{eqker}.
\end{proof}

\vspace{10pt}

\section{Harmonic fields, de Rham maps and RS torsion}
\label{DRX}

It is quite obvious that the combinatorial decomposition of the underling topological space $X$ does not depend on the type of the metric structure on it. Consequently, the same decomposition would work for any type of isolated singularity. An $n$-dimensional  pseudo manifold with isolated singularities and proper boundary  is  an $n$-dimensional pseudo manifold $K$ with boundary    such that there exists a (finite) subset of $0$-cells  $\Sigma=\{x_0, \dots, x_k\}$ in the interior of  $K$  such that $K-\Sigma$ is an $n$-dimensional   homology manifold with boundary. 
An $n$-dimensional  space  with isolated singularity and proper boundary is a topological space $X$ that admits a CW decomposition $K$, where $K$ is  $n$-dimensional pseudo manifold with isolated singularities and proper boundary \cite{Mun} \cite{Mas} \cite{Spa} \cite{Spr12}. An $n$-dimensional pseudo manifold $K$ with isolated singularities is the push out  of a regular CW complex $M$ that is an  homology $n$-manifold  with boundary $\b M=N_1\sqcup\dots \sqcup N_k$ by attaching the  cone $C (N_k)$ to $N_k$, for each $k$. We will call  the  decomposition 
$K=M\sqcup_{N_0} C(N_0)$ of $X$,   
where $N_0= {\rm Link}(x_0)$, $M=\overline{K-C(N_0)}$, the  canonical  decomposition  of $X$ (note that $N_0=\b M-(\b M\cap L)$ if $K$ has a  boundary $L$). If the link of the singular point is not a sphere, the classical chain complex induced by any combinatorial decomposition  will not satisfy the classical duality properties satisfied by the chain complexes induced by triangulation of manifolds. However, it is possible to define a suitable chain complex whose homology satisfies  duality. This is the intersection homology of Goresky and MacPherson \cite{GM1} \cite{GM2} \cite{KW}. Such a chain complex is defined by selecting the cells that intersect the singular point (we refer to construction in \cite{Spr12} for details, or to the original work \cite{GM1}, for an equivalent but slightly different construction). The chain complex and the homology groups depend on a graded integer, called perversity. We are mainly concerned with two perversities: the lower middle perversity $\mf=\{\mf_j=[j/2]-1\}_{j=2}^n$, and its  complementary perversity  $\mf^c_j=j-\mf_j-2$. 
Then, there exists a chain complex whose homology, called intersection homology, is as follows ($m+1=n$).

\begin{prop}\label{pop1}  Let $N$ be a regular oriented CW complex that is an homology $m$-manifold,   $C(N)$ the cone over $N$. If  $(m,\pf)=(2p-1,\pf), (2p, \mf^c)$, then 
\begin{align*}
I^\pf H_q(C(N))&=\left\{
\begin{array}{ll}
H_q(N), & 0\leq q\leq p-1,\\
0, & p \leq 	q\leq  m+1,
\end{array}\right.
\\
I^\pf H_q(C(N),N)&=\left\{
\begin{array}{ll}
0 ,& 0\leq q\leq p,\\
H_q(C(N),N) = H_{q-1}(N), & p+1 \leq q\leq m+1,
\end{array}
\right.
\end{align*}
if  $(m,\pf)=(2p, \mf)$, then 
\begin{align*}
I^\pf H_q(C(N))&=\left\{
\begin{array}{ll}
H_q(N), & 0\leq q\leq p,\\
0, & p+1 \leq 	q\leq  2p+1,
\end{array}\right.
\\
I^{\mf^c} H_q(C(N),N)&=\left\{
\begin{array}{ll}
0 ,& 0\leq q\leq p+1,\\
H_q(C(N),N) = H_{q-1}(N), & p+2 \leq q\leq 2p+1;
\end{array}
\right.
\end{align*}
\end{prop}

\begin{prop}\label{pop2} Let  $K$ be an $n$-pseudo manifold with an isolated  singularity and without boundary. Let $K=M\sqcup C(N)$ be the standard decomposition of $K$. If  $(m,\pf)=(2p-1,\pf), (2p, \mf^c)$, then
\begin{align*}
I^\pf H_q(K)&=\left\{\begin{array}{ll}H_q(M)&0\leq q\leq p-1,\\
\Im\left(p_*:H_{p}(M)\to H_{p}(M,\b M)\right),\\
H_q(K)&p+1\leq q\geq m+1,
\end{array}\right.
\end{align*}
if  $(m,\pf)=(2p,\mf^c)$, then
\begin{align*}
I^{\mf^c} H_q(K)&=\left\{\begin{array}{ll}H_q(M)&0\leq q\leq p,\\
\Im\left(p_*:H_{p+1}(M)\to H_{p+1}(M,\b M)\right),\\
H_q(K)&p+2\leq q\geq m+1,
\end{array}\right.
\end{align*}
where $p:M\to M/\b M$ is the canonical projection on the quotient, and the equivalence is induced by the inclusion $M\to K$. 
\end{prop}

For this homology, we have  the following duality isomorphisms (see \cite{HS6} for details, or the original work \cite{GM1}):
\begin{align*}
I^\pf \QQ'_{*,q}&:H_q(I^\pf \CS_\bu(C(N)))\to H^\da_{m+1-q}(I^{\pf^c}\CS_\bu(C( N),  N)),\\
I^\pf\QQ_{*,q}&:H_q(I^\pf \CS_\bu(K))\to H^\da_{m+1-q}(I^{\pf^c}\CS_\bu(K^*)).
\end{align*}

Recall that $C(N)$ is a triangulation for $\ZZ$, and therefore $K$ is a triangulation for $X$, thus the previous results give the intersection homology of a finite horn and of a space $X$ with an horn singularity. Comparing  Propositions \ref{pop1} and \ref{pop2} with Theorems \ref{coZ} and \ref{coX}, we see that (with real coefficients) the intersection homology of a space with an horn singularity coincides with the de Rham cohomology. We want to define de Rham maps on $X$ that extend the classical de Rham maps on $Y$.  On the finite horns, this follows by using the harmonic fields, see Corollary \ref{eqker}. The following proposition extends this result to $X$ (that we assume without boundary).

\begin{prop}\label{harmonicsX} If $(m,\pf)=(2p-1,\mf), (2p-1,\mf^c), (2p,\mf^c)$, then there is an isomorphism (here $i:\b Y\to Y$ denotes the natural inclusion of the boundary)
\begin{align*}
\ker \d^{(q)}_{ \pf}\cap \ker \dr^{(q)}_{ \pf}=\left\{\begin{array}{ll} \ker \d^{(q)}_{Y, \rm max}\cap \ker \dr^{(q)}_{Y, \rm min},&0\leq q\leq p-1,\\
\{\te\in \ker \d^{(q)}_{Y, \rm max}\cap \ker \dr^{(q)}_{Y, \rm min}~|~[i^*\te]=0\in H^p(\b Y)\},&q=p,\\
\ker \d^{(q)}_{Y, \rm min}\cap \ker \dr^{(q)}_{Y, \rm max},& p+1\leq q\leq m+1,
\end{array}
\right.
\end{align*}
if $(m,\pf)=(2p,\mf)$, then
\begin{align*}
\ker \d^{(q)}_{ \pf}\cap \ker \dr^{(q)}_{ \pf}=\left\{\begin{array}{ll} \ker \d^{(q)}_{Y, \rm max}\cap \ker \dr^{(q)}_{Y, \rm min},&0\leq q\leq p,\\
\{\te\in \ker \d^{(q)}_{Y, \rm max}\cap \ker \dr^{(q)}_{Y, \rm min}~|~[i^*\te]=0\in H^p(\b Y)\},&q=p+1,\\
\ker \d^{(q)}_{Y, \rm min}\cap \ker \dr^{(q)}_{Y, \rm max},& p+2\leq q\leq m+1,
\end{array}
\right.
\end{align*}
\end{prop}
\begin{proof} 

By Theorem \ref{Hodgedec},  $ \ker \Delta^{(q)}_{\d,\pf }=\ker \d^{(q)}_{ \pf}\cap \ker \dr^{(q)}_{ \pf}$. Moreover, it is clear enough that Corollary \ref{eqker} extends to $X$ and therefore we may define $\Ha^{(q)}_\pf(X)=\ker \Delta^{(q)}_{\pf }= \ker \Delta^{(q)}_{\d,\pf }$. Take $\om\in \ker \Delta^{(q)}_{\pf }$, then the restriction $\om|_\CF=f_1\t\om_1+f_2dx \wedge \t\om_2$ of $\om$ on $\CF$  is an harmonic form of $\CF$. 
Consequently, $\om|_\CF$ is of the types described in  Proposition 
\ref{harmonics}. Moreover, since $\om|_{\CF}$ is closed and co closed, it 
must be of  type $E_0$ (constant $f$), $O_0$ ($f=h^{1-2\al_{q-1}}$), or 
XY=II=III. Also, a direct verification shows that the bc at $x=x_0$ are 
always satisfied by these forms, when ever they are square integrable. 

Consider the forms that are locally  of type either $E_0$ or $O_0$. The forms of type $E_0$ are square integrable when $\al_q<1$, and having null normal component their restriction to $Y$  represents a non trivial absolute cohomology class of $Y$. The forms of type $O_0$ are square integrable for $\al_q>1$, and have null tangential component, so their restriction to $Y$ represents a non trivial relative cohomology class. Since the vice versa is as well clear, this proves the result in all non critical degrees. 

Next, consider the forms locally of type XY, $\om|_\CF=f \t D\t\om+dx\wedge f' \t\om$. Taking the + solution $f$, these forms are square integrable and satisfy the bc at $x=0$ (use  Theorem \ref{olv1}). Consider the restriction $\psi_Y=\om|_Y$. Assume first that  either $f_\pm(l)=0$, or $f'_\pm(l)=0$, then $ (\FF_{1,\t\la}^{q-1} f ,f)=0$, 
and integrating over $(0,l]$, this shows that  $f'=\t\la=0$, and therefore $\om|_C$ vanishes identically. But then, $f(l)=f'(l)=0$, and  $\om$ vanishes identically by \cite[2.4.1]{Sch}.
It remains the case $f(l), f'(l)\not=0$. 
Then, by the Friedrichs decomposition on $Y$,  $\psi_Y$ is the sum of an  harmonic $\psi_{Y, \rm abs}$ satisfying the absolute bc on $\b Y$  plus an exact form, and the sum of an  harmonic $\psi_{Y, \rm rel}$ satisfying the relative bc on $\b Y$  plus a co exact form.  This means that $\psi_Y$ and $\psi_{Y, \rm abs}$, and $\psi_Y$ and $\psi_{Y, \rm rel}$ are in the same cohomology classes, respectively. Whence, these forms are in classes of absolute cohomology that are images of classes in relative cohomology. 
\end{proof}

In particular, we may construct maps from the space of square integrable forms to the space of the chain that extend the classical de Rham maps defined in the smooth case. Consider as an example the following case:  $m=2p$, $\pf=\mf$,  and $0\leq q\leq p-1$. Then we can construct the following commutative diagram of isomorphisms

\centerline{
\xymatrix{\Ha^q_{ \mf}(X)\ar[r]^{\star}&\Ha^{m-q+1}_{ \mf^c}(X)\ar@{..>}[r]^{I^{\mf^c} \A^{m-q+1}}&I^{\mf^c} H^{m-q+1}(X)&I^\mf H_q(X)\ar[l]_{\hspace{40pt}I^\mf \QQ_{*,q}}\\
\FF^q_{\rm abs}( Y)\ar[r]^{\star_Y}\ar[u]_{k^*}&
\Ha^{m-q}_{\rm rel}( Y)\ar[r]^{\A^{m-q}_{\rm rel}}\ar[u]_{ k^*}&H^{m-q}(Y,W)\ar[u]&H_q( Y)\ar[l]_{\hat\QQ_{*, q}}\ar[u]
}}

The map $k^*$ is the inverse of inclusion in co homology,  
the map $\hat \QQ_{*,q}$ is standard Poincar\'e map, the map $I^\mf \QQ_{*,q}$ the corresponding extension to the singular case defined above, and the map  $\A^{m-q}_{\rm rel}$ is the classical de Rham map in  the smooth case \cite{RS}. Commutativity defines the isomorphism, also called de Rham map
\[
I^\pf \A_{q}: \Ha^q_{\pf}(X)\to I^{\pf} H_q(X).
\]

We can define and compute the RS torsion for the intersection chain complexes, denoted by $I^\pf \tau_{\rm RS}(\ZZ)$.

\begin{theo}\label{t7.29} Let $W$ be a compact connected orientable Riemannian manifold without boundary, of dimension $m$.  Let $\tilde\alphas_\bu$ a graded orthonormal basis of $\Ha^q(W)$, and $\ns_\bu$ the standard graded basis of $H_\bu(W)$. Then, ($r_q=\rk H_q(W;\Z)$)

\begin{align*}
I^\pf \tau_{\rm RS}(\ZZ)
=& \prod_{q=0}^{\af-2} \gamma_q^{\frac{(-1)^{q}}{2}r_q }
 \left|\det (\tilde\A_{q}(\tilde\alphas_q)/\ns_q)\right|^{(-1)^q }
\left(\# TH_q(W;\Z)\right)^{(-1)^q },\\
I^\pf \tau_{\rm RS}(\ZZ, W)
=& \prod_{q=\af-1}^{m}  \gamma_q^{\frac{(-1)^{q}}{2}r_q }
 \left|\det (\tilde\A_q(\tilde\alphas_q)/\ns_q)\right|^{(-1)^q}
 \left(\# TH_q(W;\Z)\right)^{(-1)^{q+1}}.
\end{align*}
\end{theo}

\begin{theo}\label{t7.37} Let $X=\ZZ \sqcup_W Y$ be a space with an horn singularity of dimension $n=m+1$,  where $(Y,W)$ is a compact connected orientable smooth Riemannian manifold of dimension $n$, with boundary $W$, then
\begin{align*}
I^\pf \tau_{\rm RS}(X)
=& I^\pf \tau_{\rm RS}(\ZZ)  \tau_{\rm RS}(Y,W)\\
&\tau_{\rm W}(I^\pf \ddot\Ha;I^\pf\A_{{\rm abs},\bu}(I^\pf \alphas_{C,\bu}), I^\pf \A_\bu(I^\pf \alpha_{X,\bu}),\A_{Y,{\rm rel}, \bu}(\alpha_{Y,{\rm rel},\bu})),
\end{align*}
where $I^\pf \alphas_{Z,q}$, $\alpha_{Y,{\rm rel},q}$, and $I^\pf \alpha_{X,q}$ are orthonormal bases of $ \H^q_\pf(\ZZ)$, $\H^q_{\rm rel}(Y)$, and  $\H^q_\pf(X)$, respectively; $I^\pf \A_{{\rm abs},q}$ is the absolute de Rham map on $\ZZ$,  $\A_{Y, {\rm rel}, q}$ the classical relative de Rham map on $Y$,  $I^\pf \A_q$ the de Rham map on $X$, and $\tau_W$ denotes the Whitehead torsion of the homology sequence (compare with \cite{Mil0})
\[
\xymatrix{
I^\pf \ddot\Ha:&  \dots\ar[r]&I^{\pf} H_q(C_{0,l}(W)) \ar[r]&I^\pf H_q(X)\ar[r]&H_q(Y,W)\ar[r]&\dots.
}
\]
\end{theo}

\section{The analytic torsion}\label{torsion}

\subsection{Torsion zeta function and analytic torsion}
\label{torsionzeta}

Let $X$ be a space with a horn singularity, and $\Delta_{\pf, \rm bc}$ one of the operators defined in Section \ref{defLap}. 
By the results of Section \ref{spectralpropX},  $\Sp_0 (\Delta_{\pf, \rm bc}^{(q)})$ is a sequence of  positive numbers with unique accumulation point at infinity, and since some power of the resolvent is of trace class, the logarithmic Gamma function 
$\log \Gamma(\la,\Sp_0 (\Delta_{ \pf, \rm bc}^{(q)}))$ is well defined \cite{Spr9}. Since the heat operator has an asymptotic expansion for small $t$, it follows that the logarithmic Gamma function has an asymptotic expansion for large Gamma. In particular, the associated zeta function is well defined \cite{Spr9}, 
\[
\zeta(s,\Delta_{\pf, \rm bc}^{(q)})=\sum_{\la\in \Sp_0 (\Delta_{\pf, \rm bc}^{(q)})} \la^{-s},
\]
and has an analytic continuation regular  at $s=0$. We define the torsion zeta function
\[
t_{\pf, \rm bc}(X)(s)=\sum_{q=0}^{m+1} (-1)^q q \zeta(s,\Delta_{\pf, \rm bc}^{(q)}),
\]
and the analytic torsion by $\left. \log T_{\pf, \rm bc}(X)=\frac{d}{d s}t_{\pf, \rm bc}(s)\right|_{s=0}$.   
Recall that the operators are defined on spaces of functions with value in some vector bundle $E_\rho$, where $\rho:\pi(\overline{X})\to O(V)$ is some orthogonal representation. However, to lighten  notation, we will omit explicit reference to the representation. Moreover, from now on we assume that $X$ has no boundary, so the bc at $\b X$ is vacuum and will be omitted.

We start by   showing that the classical variational result for the analytic torsion with respect to the metric \cite{RS} extend to the case of the space $X$. Let $g_{\mu}$ be a smooth family of metrics on $X$ such that the restriction on $Y$ is a smooth family of metrics on $Y$, and the restriction on $\ZZ$ is
$
g_{\mu}|_Z=g_{h_\mu}=\d x\otimes \d x+h_\mu^2(x) \tilde g_\mu,
$ 
where $h_\mu$ is a smooth family of functions on $[0,l]$ that coincide on $[0,l-\ep]$, for some positive $\ep<\iota$, and satisfy in $[0,l]$ all the hypothesis introduced in Section \ref{geo}. Denote the analytic torsion of the induced Hodge-Laplace operators  by $T_{\pf}(X,g_{\mu})$, respectively. The next two results are essentially in  \cite{RS} and \cite{Les2}, respectively.

\begin{theo}\label{variationX} With the notation above,  
\begin{align*}
\frac{\d}{\d\mu}\log T_{\pf}(X,g_{\mu})= -\frac{1}{2} \log \|\Det I^\pf \alphas_{g_{\mu}, \bullet}\|_{\Det \H^{\bullet}_{\pf}(X,g_{\mu})}^2
\end{align*}
where $I^\pf \alphas_{g_{\mu},q}$ is an orthonormal basis of $\H^{(q)}_{\pf}(X,g_{\mu})=\ker \Delta^{(q)}_{\pf, \mu}$.
\end{theo}

\begin{proof}    By Corollary \ref{heattraceX}, we have the 
\[
\frac{\partial}{\partial\mu} {\rm Tr} (\e^{-t \Delta^{(q)}_{ \pf, \mu}}) = - t {\rm Tr} \left( \left(\frac{\partial}{\partial \mu} 
\Delta^{(q)}_{\pf, \mu}\right) \e^{-t \Delta^{(q)}_{\pf, \mu}}\right),
\]
and proceeding  as in \cite[pg. 152-153]{RS} (here $\beta_\mu^{(q)}= (\star^{(q)}_\mu)^{-1}  \frac{\partial}{\partial_\mu} \star^{(q)}_\mu$),
\begin{equation*}
\begin{aligned}
\sum_{q=0}^{m+1} (-1)^{q+1} q {\rm Tr} \left( \left(\frac{\partial}{\partial \mu} 
\Delta^{(q)}_{ \pf, \mu}\right) \e^{-t \Delta^{(q)}_{ \pf, \mu}}\right)
 &= \sum_{q=0}^{m+1} (-1)^{q+1} {\rm Tr} (\beta^{(q)}_\mu \Delta^{(q)}_{ \pf, \mu} \e^{-t \Delta^{(q)}_{ \pf, \mu}})\\
&=\frac{\partial}{\partial t}  \sum_{q=0}^{m+1} (-1)^{q} {\rm Tr} (\beta^{(q)}_\mu \e^{-t \Delta^{(q)}_{ \pf, \mu}}).
\end{aligned}
\end{equation*}

By definition \cite[2.3]{Spr9}, the torsion zeta function of the family of operators $\Delta^{(q)}_{ \pf, \mu}$ is
\[
 t_{X,\pf}(s,\mu)= \frac{1}{2} \frac{1}{\Gamma(s)}\sum_{q=0}^{n+1} (-1)^q q \int_0^\infty t^{s-1} {\rm Tr} \left(\e^{-t \Delta^{(q)}_{ \pf, \mu}} - P_{\ker \Delta^{(q)}_{ \pf, \mu}}\right) dt. 
\] 

Then 
\[
\begin{aligned}
\frac{\partial }{\partial\mu}  t_{X,\pf}(s,\mu) 
&= \frac{1}{2} \frac{s}{\Gamma(s)}\sum_{q=0}^{m+1} (-1)^q  \int_0^\infty  t^{s-1}  {\rm Tr} \left(\beta_\mu^{(q)} \left(e^{-t \Delta^{(q)}_{\pf, \mu}} - P_{\ker \Delta^{(q)}_{ \pf, \mu}}\right)\right) dt.
\end{aligned}
\]

By definition of analytic torsion, using the asymptotic expansion of the heat kernel, 
\[
\frac{\b}{\b\mu} \log T_{\pf}(X,\g_{X,\mu})=\frac{\partial }{\partial\mu}   t'_{X,\pf}(0,\mu)
=\frac{1}{2} \sum_{q=0}^{m+1} (-1)^{q+1} \left(a_\frac{m+1}{2}(\Delta^{(q)}_{ \pf, \mu})-
\Tr \beta_\mu^{(q)}  P_{\ker \Delta^{(q)}_{ \pf, \mu}}\right).
\]

Since the coefficients  with half integer index vanish,  using   duality
$a_k(\Delta^{(m+1-q)}_{ \pf, \mu})=a_k(\Delta^{(q)}_{ \pf^c, \mu})$  \cite{Gil},  
the first term sums to zero, and we have
\[
\frac{\b}{\b\mu}\log T_{\pf}(X,\g_{X,\mu})=\frac{\partial }{\partial\mu}  t'_{X,\pf}(0,\mu)
=\frac{1}{2} \sum_{q=0}^{m+1} (-1)^q \Tr \beta_\mu^{(q)}  P_{\ker \Delta^{(q)}_{ \pf, \mu}}.
\]
 \end{proof}

\begin{theo}\label{les} Let $X=\ZZ\sqcup_W Y$ be a space with a horn singularity of dimension $m+1$,  where $(Y,W)$ is a compact connected orientable smooth Riemannian manifold of dimension $m+1$, with boundary $W$. Assume the metric $g_{0}$ on $X$ is a product in a collar neighbourhood of $W$.  
 Then, 
\[
\log T_{\pf}(X,g_0)=\log  T_{\pf,\rm abs}(\ZZ,g_0|_Z)+\log T_{\rm rel}(Y,g_0|_Y)-\frac{1}{2}\chi(W)\log 2+\log \tau(\H_0),
\]
where $\H_0$ is the long homology exact sequence induced by the inclusion 
$\ZZ\to X$, with some coherent orthonormal graded basis.
\end{theo}
\begin{proof} By Theorem \ref{equiv}, the operator $\Delta_{\pf}$ has the 
same spectrum as the Hodge Laplace operator $\Delta_{\dr, \pf}$ associated to the de Rham 
intersection complex $(\DS(\d_{\pf}),\d_\pf)$. This is an  
Hilbert complex with finite dimension spectrum according to \cite{BL0} and 
\cite[pg. 239, 3-25]{Les2}. All the hypothesis necessary in order to apply 
Theorem 6.1 of \cite{Les2} are verified. \end{proof}

\begin{corol}  Let $X=\ZZ\sqcup_W Y$ be a space with a horn singularity of dimension $m+1$,  where $(Y,W)$ is a compact connected orientable smooth Riemannian manifold of dimension $m+1$, with boundary $W$. 
 Then, 
\[
\log T_{\pf}(X)=\log  T_{\pf,\rm abs}(\ZZ)+\log T_{\rm rel}(Y)-\frac{1}{2}\chi(W)\log 2+\log \tau(\H_0),
\]
where $\H_0$ is the long homology exact sequence induced by the inclusion 
$\ZZ\to X$, with some coherent orthonormal graded basis. 
\end{corol}

\subsection{The analytic torsion of the finite horn}
\label{torhorn}

In this section we compute the analytic torsion of the finite horn. The calculation is based on the decomposition described in the following lemma, whose proof exploits  duality to rearrange the terms in the definition of the torsion zeta function. 

\begin{lem}\label{ppp}
The torsion zeta function of the finite horn $\ZZ$  with absolute boundary condition at the boundary  decomposes as follows:
\begin{equation*}
t_{\pf, \rm abs}(\ZZ) (s) = t^{(m)}_0(s) + t^{(m)}_1(s) + t^{(m)}_{2}(s) + t^{(m)}_{3,{\pf}}(s),
\end{equation*}
where
\begin{equation*}
\begin{aligned}
t^{(m)}_0(s)=&\frac{1}{2}\sum_{q=0}^{[\frac{m}{2}]-1}(-1)^{q}
\left((\mathcal{Z}_{q,-}(s) - \hat{\mathcal{Z}}_{q,+}(s)) 
 +(-1)^{m-1} 
(\mathcal{Z}_{q,+}(s)-\hat{\mathcal{Z}}_{q,-}(s))\right),\\
t^{(2p-1)}_1(s)=&(-1)^{p-1}\frac{1}{2} 
\left(\mathcal{Z}_{p-1,-}(s)-\hat{\mathcal{Z}}_{p-1,+}(s)\right),\hspace{30pt}
t^{(2p)}_1(s) = 0;\\
t^{(m)}_{2}(s)=&\frac{1}{2}\sum_{q=0}^{[\frac{m-1}{2}]} (-1)^{q+1} 
\left( z_{q-1,-} (s)+(-1)^{m} z_{q,+} (s)\right),\\
t^{(2p-1)}_3(s)=&0, \hspace{30pt} t^{(2p)}_{3,{\mf^c}}(s)=\frac{(-1)^{p+1}}{2}   \zeta_{-} (s), \hspace{10pt}
 t^{(2p)}_{3,{\mf}}(s)=\frac{(-1)^{p+1}}{2}  \zeta_{+} (s),
\end{aligned}
\end{equation*}
and
\begin{align*}
\mathcal{Z}_{q,\pm}(s) &= \sum_{n,k=1}^\infty m_{{\rm cex},q,n}\ell^{-s}_{\pm\alpha_q,\t\la_{q,n},k},&
\hat{\mathcal{Z}}_{q,\pm}(s) &= \sum_{n,k=1}^\infty m_{{\rm cex},q,n}\hat \ell^{-s}_{\pm\alpha_q,\t\la_{q,n},k},\\
z_{q,\pm}(s)&=\sum_{k=1}^{\infty} m_{{\rm har},q}\ell^{-s}_{\pm\alpha_q,0,k},&
\zeta_{\pm}(s)&=\sum_{k=1}^{\infty} m_{{\rm har},p}\ell^{-s}_{\frac{1}{2},\frac{1}{2},\pm,k}.
\end{align*}
\end{lem}

Next we proceed in three main steps.

\subsubsection{Contributions comping from simple series}

The analysis of $\zeta_\pm$ is classical.  Recalling that $\al_p=\frac{1}{2}$, and $\ell_{\frac{1}{2},\frac{1}{2},+,k}=\left(\frac{\pi}{l}{k}\right)^2$, and  
$\ell_{\frac{1}{2},\frac{1}{2},-,k}=\left(\frac{\pi}{2l}(2k-1)\right)^2$, we have that
\begin{align*}
\zeta_-(s)&=\left(\frac{\pi}{2l}\right)^{-2s}\sum_{k=0}^\infty (2k+1)^{-2s}=\left(\frac{\pi}{2l}\right)^{-2s}\left(\zeta_{\rm R}(2s)-2^{-2s}\zeta_{\rm R}(2s)\right),\\
\zeta_+(s)&=\left(\frac{\pi}{l}\right)^{-2s}\sum_{k=1}^\infty (2k)^{-2s}=\left(\frac{\pi}{l}\right)^{-2s}\zeta_{\rm R}(2s).
\end{align*}

Whence $\zeta_+'(0)=-\log 2l$, $ \zeta_-'(0)=-\log 2$, 
that gives
\begin{align*}
(t^{(2p)}_{3,\mf^c})'(0)&=\frac{(-1)^{p}}{2} m_{{\rm har},p}\log 2,&
(t^{(2p)}_{3,\mf})'(0)&=\frac{(-1)^{p}}{2} m_{{\rm har},p}\log 2l.
\end{align*}

It remains to treat $z_{q,\pm}$. Here, the function $z_{q,\pm}(s)$ is the zeta function associated to the 
simple sequence 
$
Y_{q,\pm}=\left\{m_{{\rm har},q}\ : \ \ell_{\pm\alpha_q,0,k}\right\},
$ 
where, by Lemma \ref{l4}, the numbers  
$\ell_{\al,0, k}$ are 
the zeros of the function $\ff_{ \al,0, +}(l,\lambda)$,  and  the eigenvalues of  $F_{\al_q, 0;0,0}$, see Section \ref{s222}.  
In all cases only  the  $+$ solution  of the relevant Sturm-Liouville equation appears, and therefore it is not necessary to distinguish the case where the roots of the indicial equation coincide. 
Note that in all case the kernel is trivial, by Lemma \ref{pos}.
By Lemma \ref{eigen} and its corollary,  the sequence  $ Y_{q,\pm}(a,b)$ has   genus 0.  We show that it is a sequence of spectral type. For consider the associated logarithmic spectral Gamma function:
\[
\log\Gamma(-\la,  Y_{q,\pm})=-\log\prod_{k=1}^\infty \left(1+\frac{-\la }{\ell_{\pm\alpha_q,0,k}}\right).
\]
Since $\log\Gamma(-\lambda, Y_{q,\pm})=\log \ff_{\pm\al_q,0,+}(l,0 )-\log 
\ff_{\pm\al_q,0,+}(l,\la )$, 
and  $\ff_{\al,0,+}(x,\lambda)$ has an asymptotic expansion for large 
$\la$ by the results in Section \ref{asym}, it follows that the 
sequence $\hat Y_{q,\pm}$ is of spectral type. Moreover, a simple calculation 
shows that it is indeed a regular sequence of spectral type. 
We  compute the constant term in the expansion of $\log\Gamma(-\la, Y_{q,\pm})$ for large $\la$, since by Theorem \ref{teo1-1},
\begin{align*}
z_{q,\pm}'(0)&=-\log \det F_{\al_q, 0;0,0}=-\Rz_{\la=+\infty} \log\Gamma(-\lambda, Y_{q,\pm})\\
&=-\ln \ff_{\pm\al_q,0,+}(l,0 )+\Rz_{\la=+\infty}\log \ff_{\pm\al_q,0,+}(l,\la )\\
&=-\ln \uf_{\pm\al_q,0,+}(l,0 )+\Rz_{\la=+\infty}\log \uf_{\pm\al_q,0,+}(l,\la ).
\end{align*}

According to Section \ref{asym}, we have the asymptotic expansion 
\begin{equation*}
\log \uf_{\pm\al_q,0,+}(x,\lambda)= \log \frac{2^{|c_{\al_q}| 
-\frac{1}{2}}\Gamma(1+|c_{\al_q}|)}{\sqrt{\pi}} - 
\frac{1}{2}\left(|c_{\al_q}|+\frac{1}{2}\right)\log(-\la) + l\sqrt{-\la} + 
O\left(\frac{1}{\sqrt{-\la}}\right),
\end{equation*}
for large $\lambda$, where $c_\al=\ka\left(\al-\frac{1}{2}\right)+\frac{1}{2}$, and hence
\beq\label{eee1}
\Rz_{\lambda \to \infty} \log \uf_{\al_q,0,+}(l,\lambda) = \log \frac{2^{|c_{\al_q}|-\frac{1}{2}}\Gamma(1+|c_{\al_q}|)}{\sqrt{\pi}}.
\eeq

Next, we compute the value of $\uf_{\al,0,+}(l,0 )$. By Remark \ref{solutions}, we have that
\begin{align}
\label{B1}
\uf_{\al_q,0,+}(l,0)&=((2\al_q-1)\ka+1) h^{\frac{1}{2}-\al_q}(l) \int_{0}^{l}h(x)^{2\al_q-1}dx, &\al_q>0,\\
\label{B2}
\uf_{\al_q,0,+}(l,0)&= h^{\frac{1}{2}-\al_q}(l), &\al_q\leq 0.
\end{align}

Observing that $0\leq q\leq p-1$ in the definition of  $t_2$, we find that
\[
z'_{q-1,-}(0) = \log \frac{2^{|c_{-\al_{q-1}}| -\frac{1}{2}}\Gamma(1+|c_{-\al_{q-1}}|)}{2c_{-\al_{q-1}}\sqrt{\pi}}h^{-\al_{q-1}-\frac{1}{2}}(l)-\log  \int_{0}^{l} h(x)^{1-2\al_{q-1}}dx,
\]%
and that
\[
z'_{q,+}(0) = \log \frac{2^{|c_{\al_{q}}| -\frac{1}{2}}\Gamma(1+|c_{\al_q}|)}{\sqrt{\pi}} h^{\al_q-\frac{1}{2}}(l).
\]

Plugging  these quantities in the formula of $t_2'(0)$, if $m=2p-1$, then
\begin{equation*}
\begin{aligned}
(t_2^{(2p-1)})'(0) =&\frac{1}{2}\sum_{q=0}^{p-1}(-1)^{q+1} m_{{\rm har},q} \log \frac{2^{c_{-\al_{q-1}}+c_{\al_q}}}{(2c_{-\al_{q-1}})} \frac{\Gamma(1+c_{-\al_{q-1}})}{\Gamma(1-c_{\al_q})}\\
&-\frac{1}{2}\sum_{q=0}^{p-1}(-1)^{q+1} m_{{\rm har},q} \log \gamma_q h(l)^{2\al_{q}-1}\\ 
=&\frac{1}{2}\sum_{q=0}^{p-1}(-1)^{q} m_{{\rm har},q} \log \gamma_q h(l)^{2q-2p+1}, 
\end{aligned}
\end{equation*}
since $c_{\al_{q-1}}+c_{\al_q}-1=0$, and where
\beq \label{eq-gammaq}
\gamma_q=\int_{0}^{l} h(x)^{-2\al_{q-1}-1} dx,
\eeq
while if $m=2p$, then
\begin{equation*}
\begin{aligned}
(t_2^{(2p)})'(0) =&\frac{1}{2}\sum_{q=0}^{p-1}(-1)^{q+1} m_{{\rm har},q} \log \frac{2^{c_{-\al_{q-1}}-c_{\al_q}-1}}{2c_{-\al_{q-1}} }\frac{\Gamma(1+c_{-\al_{q-1}})\Gamma(1-c_{\al_q})}{\pi}\\
=&\frac{1}{2}\sum_{q=0}^{p-1}(-1)^{q+1} m_{{\rm har},q} \log \frac{2^{\ka(1-2\al_{q})-2}}{\pi} \Gamma^2\left(\frac{1}{2}-\ka\left(\al_q-\frac{1}{2}\right)\right)\gamma_q.
\end{aligned}
\end{equation*}

\subsubsection{Contribution coming from double series: regular part}
\label{s6.1}

Consider the zeta functions $\mathcal{Z}$ and  $\hat{\mathcal{Z}}$. These are the zeta functions associated to the double sequences $S_{q,\pm}=\left\{m_{{\rm cex},q,n}\ : \  \ell_{\pm\al_q,\t\la_{q,n},k}\right\}$, and $\hat S_{q,\pm}=\left\{m_{{\rm cex},q,n}\ : \  \hat\ell_{\pm\alpha_q,\t\la_{q,n},k}\right\}$ ($\t\la_{q,n}\not=0$), and will be treated by the means described in Appendix \ref{backss}. We introduce the simple sequence $\tilde S_q=\left\{ m_{q,n}~:~\sqrt{\tilde\la_{q,n}}\right\}$: this  is a totally regular simple sequence of spectral type with infinite order, exponent of convergence and genus $\es(\tilde S_q)=\gs(\tilde S_q)=m=\dim W$, by \cite[Proposition 3.1]{Spr9}. The associated zeta function
$
\zeta(s,\tilde S_q)=\zeta_{\rm cex}\left(\frac{s}{2},\tilde\Delta^{(q)}\right)
$ 
has possible simple poles at $s=m-h$, $h=0,2,4, \dots, m-1$, \cite[Proposition 3.2]{Spr9}. 
Thus, by Lemma \ref{eigen} and its corollary, both the sequence $S_q$ and $\hat S_{q,\pm}$ are double sequences  of relative order   $\left(\frac{m+1}{2},\frac{m}{2},\frac{1}{2}\right)$ and relative genus $\left(\left[\frac{m+1}{2}\right],\left[\frac{m}{2}\right],0\right)$ see \cite[Section 3]{Spr9}.    We show that the sequences $S$ are spectrally decomposable (with power $2$) over the sequence $\tilde S_q$.  For,  consider the associated logarithmic spectral Gamma functions:
\[
\log \Gamma(-\la \t\la_{q,n},S_{q,\pm})=-\log\prod_{k=1}^\infty \left(1+\frac{-\la \t\la_{q,n}}{\ell_{\pm\al_q,\t\la_{q,n},k}}\right),
\]
and
\[
\log\Gamma(-\la \t\la_{q,n}, \hat S_{q,\pm})=-\log\prod_{k=1}^\infty \left(1+\frac{-\la \t\la_{q,n}}{\hat\ell_{\pm\alpha_q,\t\la_{q,n},k}}\right).
\]

Recall that, by Lemma \ref{l4}, and according to Sections \ref{s222} and \ref{other}, the numbers $\ell_{\pm\al_q,\t\la_{q,n},k}$ are the positive eigenvalues of the operator $F_{\al_q, \t\la_{q,n};0}$, 
and the zeros of the function (recall $\t\la\not=0$, we are in the irregular singular case, and $x=0$ is limit point)
\[
B_{\pm\al_q, \t\la_{q,n}, \frac{\pi}{2}}(l,\la )=\ff_{\pm\al_q,\t\la_{q,n},+}(l,\la),
\] 
the numbers $\hat \ell_{\pm\al_q,\t\la_{q,n},,k}$  are the positive eigenvalues of the operator $F_{\al_q,\t\la_{q,n};\frac{1}{2}}$, and the zeros of the function
\[
B_{\pm\al_q, \t\la_{q,n},0}(l,\la )=\ff'_{\pm\al_q,\t\la_{q,n},+}(l,\la),
\] 

We need the uniform (in $\la$) asymptotic expansion of these functions for large $\tilde\la_{q,n}$. Observe that by Proposition \ref{pos}, these operators have trivial kernel. Proceeding as at the end of Section \ref{other}, we have 
\begin{align*}
\log \Gamma(-\la \mu^2_{q,n},S_{q,\pm})
&=\log B_{\pm\al_q, \t\la_{q,n}, \frac{\pi}{2}}(l,0)-\log B_{\pm\al_q, \t\la_{q,n}, \frac{\pi}{2}}(l,\la\mu_{q,n}^2),\\
\log\Gamma(-\la \t\la_{q,n}, \hat S_{q,\pm})
&=\log B_{\pm\al_q, \t\la_{q,n},0}(l,0)-\log B_{\pm\al_q, \t\la_{q,n},0}(l,\la \t\la_{q,n}).
\end{align*}

Using the expansions of the solutions for large $\nu=\t\la_{q,n}$ and fixed $x=l$, see Section \ref{asym}, we obtain the required expansion of the logarithmic Gamma functions, and we complete the proof. We give details for  the sequence $S_{q,\pm}$:
\begin{align*}
\log\Gamma(-\la \t\la_{q,n},  S_{q,\pm})
=&\log B_{\pm\al_q, \t\la_{q,n}, \frac{\pi}{2}}(l,0)-\log \frac{2^{\t\la_{q,n}} \Gamma(\t\la_{q,n}+1)}{\sqrt{2\pi\t\la_{q,n}}(-\la)^\frac{\t\la_{q,n}}{2}}\\
&-\log \frac{\sqrt{h(l)}}{\left(1-\la h^2(l)\right)^\frac{1}{4}}
-\mu_{q,n} \mathlarger{\mathlarger{\int}} \frac{\sqrt{1-\la h^2(x)}}{ h(x)} dx\mathlarger{\mathlarger{\mathlarger{|}}}_{x=l}\\
&-\log \left(1+\sum_{j=1}^J U_{q,j,\pm}(l,i\sqrt{-\la}) \t\la_{q,n}^{-j}+ O\left(\frac{1}{\t\la_{q,n}^{J+1}}\right)\right),
\end{align*}

By Remark \ref{last}, we realise that there are not relevant logarithmic 
terms, so $L$, in \eqref{exp}, can take any value, while 
the relevant terms in powers of $\t\la_{q,n}$ are all negative,  $\sigma_h=m-h$, $h=0,1,2,\dots m-1$, and the functions $\phi_{\sigma_h}$ are given by the  equation
\[
\sum_{h=0}^{m-1} \phi_{q,m-h,\pm}(l,\la)\t\la_{q,n}^{-h}=-\log \left(1+\sum_{j=1}^{m-1} U_{q,j,\pm}(l,z) \t\la_{q,n}^{-j}\right)+ O\left(\frac{1}{\t\la_{q,n}^{m}}\right).
\]

This shows that  the sequence $S_{q,\pm}$ is spectral decomposable on the sequence $\tilde S_q$, according to Definition \ref{spdec}. We give here the values of the parameters appearing in the definition:
$(s_0,s_1,s_2)=\left(\frac{m+1}{2},\frac{m}{2},\frac{1}{2}\right)$, 
$(p_0,p_1,p_2)=\left(\left[\frac{m+1}{2}\right],\left[\frac{m}{2}\right],0\right)$, 
$r_0=m$, $q=m$,
$\ka=2$, $\ell=m$. 


We want now to compute the regular part according to the formulas in the Theorem \ref{sdl}, so we need to identify the quantities $A_{0,0}(0)$, and $A_{0,1}'(0)$. 
Before starting calculations, observe that all the coefficients $b_{\sigma_h, j,0/1}$ vanish. For observe that the expansions of the functions $\phi_{q,m-h,\pm}(l,\la)$ and $\hat \phi_{q,m-h,\pm}(l,\la)$ for large $\lambda$ have terms only with negative powers of $-\la$ and negative powers of $-\la$ times $\log (-\la)$, as follows by the asymptotic characterisation of the functions $U_j(x,i\sqrt{-\la})$ and $V_j(x,i\sqrt{-\la})$, Section \ref{asym}. Whence:
\begin{align*}
\mathcal{Z}'_{{\rm reg},q,\pm}(0) &=-A_{0,0,q,\pm}(0)-A'_{0,1,q,\pm}(0),&
 \hat{\mathcal{Z}'}_{{\rm reg},q,+}(0)=-\hat A_{0,0,q,\pm}(0)-\hat A'_{0,1,q,\pm}(0),
\end{align*}
where 
\begin{align*}
A_{0,0,q,\pm}(s)&=\sum_{n=1}^\infty m_{{\rm cex},q,n} a_{0,0,q,\pm}\t\la_{q,n}^{-s},&a_{0,0,q,\pm}=\Rz_{\la=\infty} \log\Gamma(-\la \t\la_{q,n},  S_{q,\pm}),\\ 
A_{0,1,q,\pm}(s)&=\sum_{n=1}^\infty m_{{\rm cex},q,n} a_{0,1,q,\pm}\t\la_{q,n}^{-s},&a_{0,1,q,\pm}=\Rz_{\la=\infty} \frac{\log\Gamma(-\la \t\la_{q,n},  S_{q,\pm})}{\log(-\la)},\\
\hat A_{0,0,q,\pm}(s)&=\sum_{n=1}^\infty m_{{\rm cex},q,n} a_{0,0,q,\pm}\t\la_{q,n}^{-s},&\hat a_{0,0,q,\pm}=\Rz_{\la=\infty} \log\Gamma(-\la \t\la_{q,n}, \hat S_{q,\pm}),\\
\hat A_{0,1,q,\pm}(s)&=\sum_{n=1}^\infty m_{{\rm cex},q,n} a_{0,1,q,\pm}\t\la_{q,n}^{-s},&\hat a_{0,1,q,\pm}=\Rz_{\la=\infty} \frac{\log\Gamma(-\la \t\la_{q,n}, \hat S_{q,\pm})}{\log(-\la)}.
\end{align*}

Consider for exampe
\begin{align*}
a_{0,0,q,\pm}&
=\log \ff_{\pm\al_q,\t\la_{q,n},+}(l,0)- \Rz_{\la=\infty}\log \ff_{\pm\al_q,\t\la_{q,n},+}(l,\la\t\la_{q,n}),\\
a_{0,1,q,\pm}&=- \Rz_{\la=\infty}\frac{\log \ff_{\pm\al_q,\t\la_{q,n},+}(l,\la\t\la_{q,n})}{\log(-\la)}.
\end{align*}

We need the asymptotic  expansions for large $\la$. By classical tools of asymptotic analysis (see for example \cite{Olv}), we can determine the asymptotic expansions of the fundamental solutions of the associated SL problems, $\ff_{\pm}$ and their derivatives for large $\la$, see Section \ref{asym}. Moreover, we find that the constant coefficient does not depend on $x$, and therefore it is the same in the function and in its derivative. Suppose that
\begin{align*}
\log \ff_{\pm\al_q,\t\la_{q,n},+}(l,\la )&=\dots +a_{1,\pm}\log(-\la)+a_{0,\pm}+\dots,\\
\log \ff'_{\pm\al_q,\t\la_{q,n},+}(l,\la)&=\dots +\hat a_{1, \pm}\log(-\la)+ a_{0, \pm}+\dots.
\end{align*}

Then, 
\begin{align*}
a_{0,0,q,\pm}=&\log \ff_{\pm\al_q,\t\la_{q,n},+}(l,0)+a_{0,\pm}+a_{1,\pm}\log \t\la_{q,n},&
a_{0,1,q,\pm}=&a_{1,\pm},\\
\hat a_{0,0,q,\pm}=&\log \ff_{\pm\al_q,\t\la_{q,n},+}(l,0)+ a_{0,\pm}+\hat a_{1,\pm}\log \t\la_{q,n},&
\hat a_{0,1,q,\pm}=&\hat a_{1,\pm}.
\end{align*}

Thus,
\begin{align*}
\mathcal{Z}'_{{\rm reg},q,\pm}(0) &=-A_{0,0,q,\pm}(0)-A'_{0,1,q,\pm}(0)
=-\sum_{n=1}^\infty m_{{\rm cex},q,n}\left( \log \ff_{\pm\al_q,\t\la_{q,n},+}(l,0)+ a_{0,\pm}\right),\\
\hat{\mathcal{Z}'}_{{\rm reg},q,\pm}(0)&=-\hat A_{0,0,q,\pm}(0)-\hat A'_{0,1,q,\pm}(0)
=-\sum_{n=1}^\infty m_{{\rm cex},q,n}\left( \log \ff'_{\pm\al_q,\t\la_{q,n},+}(l,0)+ a_{0,\pm}\right),
\end{align*}
and 
\[
\mathcal{Z}'_{{\rm reg},q,-}(0) - \hat{\mathcal{Z}'}_{{\rm reg},q,+}(0)=-\sum_{n=1}^\infty m_{{\rm cex},q,n} 
\log\frac{\ff_{-\al_q,\t\la_{q,n},+}(l,0)}{\ff'_{\al_q,\t\la_{q,n},+}(l,0)}.
\]

By Lemma \ref{f+-}, 
$
\ff'_{-\al_q,\t\la_{q,n},+}(x,\la)=C\ff_{\al_q,\t\la_{q,n},+}(x,\la) h^{2\al_q-1}(x),
$ 
for some constant $C$. In order to determine the value of $C$, we may compute the limit for $x\to0$, using Theorem \ref{olv1}. This gives $C=\sqrt{\t\la_{q,n}}$.  Thus, 
\[
\mathcal{Z}'_{{\rm reg},q,-}(0) - \hat{\mathcal{Z}'}_{{\rm reg},q,+}(0)=\sum_{n=1}^\infty m_{{\rm cex},q,n} 
\left(\frac{1}{2}\log \t\la_{q,n}-\log h(l)\right).
\]

We may now complete the determination of the regular contribution, 
First, suppose that $m= 2p-1$. Then,  we have

\begin{equation*}
\begin{aligned}
(t^{(2p-1)}_{0,{\rm reg}})'(0)&=\frac{1}{2}\sum_{q=0}^{p-2}(-1)^{q}
\left(\mathcal{Z}'_{{\rm reg},q,-}(0) - \hat{\mathcal{Z}}_{{\rm reg},q,+}'(0) 
+ 
\mathcal{Z}'_{{\rm reg},q,+}(0)-\hat{\mathcal{Z}}'_{{\rm reg},q,-}(0)\right)\\
&= \sum_{q=0}^{p-2}(-1)^{q} 
\left(-\frac{1}{2}\zeta'(0,\tilde{\Delta}^{(q)}_{\rm cex}) - 
\zeta(0,\tilde{\Delta}^{(q)}_{\rm cex}) \ \log h(l)\right).
\end{aligned}
\end{equation*}

Second, suppose that $m=2p$. Then, we have

\begin{equation*}
\begin{aligned}
(t^{(2p)}_{0,{\rm reg}})'(0)&=\frac{1}{2}\sum_{q=0}^{p-1}(-1)^{q}
\left((\mathcal{Z}_{{\rm reg},q,-}'(0) - \hat{\mathcal{Z}}_{{\rm reg},q,+}'(0)) 
-
(\mathcal{Z}_{{\rm reg},q,+}'(0)-\hat{\mathcal{Z}}_{{\rm reg},q,-}'(0))\right)=0.
\end{aligned}
\end{equation*}

The last step in the calculation of the regular part coming from the double sequences, is  the regular part of $t_1$. This is a particular case of the previous one, with $m=2p-1$, and  $q=p$. Since $\alpha_{p-1} = 0$,   we have,
\begin{equation*}
\begin{aligned}
(t_{1,{\rm reg}}^{(2p-1)})'(0) &= (-1)^{p-1} \frac{1}{2}\left((\mathcal{Z}_{{\rm reg},p-1,-}'(0) - 
\hat{\mathcal{Z}}_{{\rm reg},p-1,+}'(0)\right)\\
&=(-1)^{p-1} \left(-\frac{1}{4}\zeta'(0,\tilde{\Delta}^{(p-1)}_{\rm cex}) - 
\frac{1}{2 }\zeta(0,\tilde{\Delta}^{(p-1)}_{\rm cex}) \ \log h(l)\right)
\end{aligned}
\end{equation*}

\subsubsection{Contribution coming from double series: singular part}

The last point concerns the determination of the singular part of the torsion, according to the SDL \ref{sdl}. The process is long and requires resolution of some intricate technical points, but the idea is quite direct. Thus we will just outline it here, avoiding the details. We consider instead of the finite metric horn, a truncated one, i.e. a frustum with a horn shape $F(W)=[l_0,l]\times W$, with the restriction of the metric  in (\ref{g1}). This means to rewrite all the process followed for the finite horn, but with  Sturm Liouville problems on a compact interval $[l_0,l]$. In such a case, all the SL operators are regular ones, so we may treat them with the classical analytic tools. We introduce the suitable Hodge Laplace operators, that will "decompose" as products of regular SL operator on the interval times the Hodge Laplace operator on the section. The main difference with respect to the horn, is that now we have classical bc for the SL problems at both the end points. We produce a description of  the spectrum as we did for the finite horns, and we have a decomposition of the torsion zeta function as in Lemma \ref{ppp}:
\[
t_{ \rm rel, abs}(F(W)) (s) = w_0(s) + w_1(s) + w_{2}(s) + w_3(s),
\]

We may compute the derivative at $s=0$ as we did for the finite horn. In particular, the terms $w_0$ and $w_1$ are zeta functions associated to double series, and we treat them by the SDL \ref{sdl}. As a result we have the analytic torsion, where reg and sing correspond to the subdivision in the SDL:

\beq\label{te1}
\log T_{\rm rel,abs}(F(W)) 
=w'_{0,{\rm reg}}(0)+w'_{1,{\rm reg}}(0)+w'_{2}(0)+w'_{3}(0)+w'_{0,{\rm sing}}(0)+w'_{1,{\rm sing}}(0).
\eeq

But the frustum is a smooth manifold with boundary, so its analytic torsion is defined and has the usual property of the analytic torsion of smooth manifold. In particular, we know (see for example \cite{Luc} \cite {BM1}) the decomposition in global and boundary parts (we are taking relative bc at $x=l_0$ and absolute bc at $x=l$ (compare \cite{RS}):
\beq\label{te2}
\log T_{\rm rel,abs}(F(W))=\log T_{\rm global, 
rel,abs}(F(W))+\log T_{\rm boundary,rel,abs}(F(W)),
\eeq
where
\begin{align*}
\log T_{\rm global,rel,abs}(F(W))=&\log 
\tau_{RS}(F(W),\b_l F(W)),\\
\log T_{\rm boundary,rel,abs}(F(W))=&\frac{1}{4}\chi (\b_l F(W))\log 
2+A_{\rm BM, rel, abs}(\b_l F(W)).
\end{align*}
where $\b_l F(W)$ denotes the boundary at $x=l$ of $F(W)$, $\tau_{RS}(M,\b M)$ is the Reidemeister torsion of the pair $(M,\b M)$ with the Ray and Singer homology basis, 
$\chi(M)$ denotes the Euler characteristic of $M$, and $A_{\rm BM, bc}(\b M)$ the anomaly boundary term of \cite{BM1}.

Now, we can compute explicitly some of the terms in equations (\ref{te1}) and (\ref{te2}), in particular those appearing the following equality, so proving the equality itself:
\begin{align*}
\log T_{\rm global,rel,abs}(F(W))+\frac{1}{4}\chi (\b_l F(W))\log 
2
=&w'_{0,{\rm reg}}(0)+w'_{1,{\rm reg}}(0)+w'_{2}(0)+w'_{3}(0),
\end{align*}

As a consequence, we have the equality:
\[
\log T_{\rm boundary,rel,abs}(F(W))-\frac{1}{4}\chi (\b_l F(W))\log 
2=A_{\rm BM,rel, abs}( \b_l F(W))
=w'_{0,{\rm sing}}(0)+w'_{1,{\rm reg}}(0).
\]

The last point consists in comparing the singular term $w'_{0,{\rm sing}}(0)+w'_{1,{\rm reg}}(0)$ of the analytic torsion of the frustum with the corresponding singular term $t'_{0,{\rm sing}}(0)+t'_{1,{\rm reg}}(0)$ of the analytic torsion of the finite horn. It is not difficult to realise that these two terms are combinations of quantities computed by using the coefficients in the asymptotic expansions for large $n$ of the same functions, namely the solutions $\ff$ of the SL problems, see Section \ref{asym}. By rearranging these coefficients in the two singular terms, we can prove that
\[
w_{j,{\rm sing}}(s)=t_{j,{\rm sing}}(s)(l_0)+t_{j,{\rm sing}}(s)(l),
\]
and therefore that
\[
t'_{0,{\rm sing}}(0)+t'_{1,{\rm sing}}(0)=\frac{1}{2}A_{\rm BM,rel, abs}( \b_l F(W))=A_{\rm BM, abs}( \b \ZZ).
\]

\subsubsection {The analytic torsion and the CM theorem}

Collecting the terms computed in the previous sections, we have the following result.

\begin{theo}\label{t1}
Let $W$ be a compact connected oriented Riemannian manifold of dimension $m$ without boundary, and  $\ZZ$  the  finite horn over $W$. Then, the  the analytic torsion of the Hodge-Laplace operator $\Delta_{\pf, \rm bc}$ on $\ZZ$ (in the trivial representation), is
\[
\log T_{\pf,\rm bc}(\ZZ)=\log T_{\rm global,   \pf, \rm bc}(\ZZ)+\log T_{\rm bound, \pf, \rm bc}(\ZZ),
\]
where 
\[
\log T_{\rm bound, \pf, \rm bc}(\ZZ)=\frac{1}{4}\chi(W)\log 2+A_{\rm BM, bc}(\b\ZZ),
\]
is the boundary anomaly term, that only depends on the boundary,  coincides with the anomaly boundary term described  in \cite{BM1,BM2} for smooth manifolds, and vanishes if the boundary is totally geodesic. 
The global term depends on the parity of the dimension $m$ of $W$. If $m=2p-1$, $p\geq 1$, then
\begin{align*}
\log T_{\rm global,\pf,  abs}(\ZZ)= &\frac{1}{2}\log T(W) 
+ \frac{1}{2} \log \|\Det k^*_\bu(\tilde\alphas_\bu)\|_{\Det I^\pf H_\bu(\ZZ)}.
\end{align*}

If $m=2p$, $p\geq 1$,  then
\begin{align*}
\log T_{\rm global, \pf, abs}(\ZZ) =& \frac{1}{4} \chi(W)\log 2
+\frac{1}{2}\log \|\Det k^*_\bu(\tilde\alphas_\bu)\|_{\Det I^\pf H_\bu(\ZZ)}\\
&+\frac{1}{2}\sum_{q=0}^{p-1}(-1)^{q+1} r_q \log \frac{2^{\ka(1-2\al_{q})}}{\pi} \Gamma^2\left(\frac{1}{2}-\ka\left(\al_q-\frac{1}{2}\right)\right).
\end{align*}
where $\tilde \alphas_q$ is an orthonormal basis of $\H^q(W)$ and $k^*:\H^q(W)\to \H^q(\ZZ)$ the map induced by inclusion, $r_q=\rk H_q(W;\Z)$.
\end{theo}

Note that 
$\|\det k^*_q(\tilde\alphas_q)\|_{\det I^\mf H_q(\ZZ)}=\gamma_q^{r_q}.
$

\begin{theo}\label{tcm} With the notation of Theorem \ref{t1}, 
if $m=2p-1$, $p\geq 1$, then
\begin{align*}
\log T_{\pf, \rm abs}(\ZZ)=&\log I^\pf \tau_{\rm RS}(\ZZ)+A_{\rm BM, abs}(W).
\end{align*}

If $m=2p$, $p\geq 1$, then
\begin{align*}
\log T_{\pf, \rm abs}(\ZZ)=&\log I^\pf \tau_{\rm RS}(\ZZ)+\frac{1}{4}\chi(W)\log 2+A_{\rm BM, abs}(W)\\
&+A_{\rm comb, \pf}(W)+A_{\rm analy}(W),
\end{align*}
where the combinatoric anomaly term is:
\begin{align*}
A_{\rm comb, \mf}(W)&=-\sum_{q=0}^{p}(-1)^q 
\log \left(\# TH_q(W;\Z)\left|\det (\tilde\A_{q}(\tilde\alphas_q)/\ns_q)\right|\right),\\
A_{\rm comb, \mf^c}(W)&=-\sum_{q=0}^{p-1}(-1)^q 
\log \left(\# TH_q(W;\Z)\left|\det (\tilde\A_{q}(\tilde\alphas_q)/\ns_q)\right|\right),
\end{align*}
the analytic anomaly terms is:
\begin{align*}
A_{\rm analy}(W)=&\frac{1}{2}\sum_{q=0}^{p-1}(-1)^{q+1} r_q \log \frac{2^{\ka(1-2\al_{q})}}{\pi} \Gamma^2\left(\frac{1}{2}-\ka\left(\al_q-\frac{1}{2}\right)\right)+\frac{1}{4}\chi(W)\log 2.
\end{align*}
\end{theo}

\begin{theo}[Duality]\label{dt} 
Let $(W,g)$ be an oriented closed Riemannian manifold of dimension $m$ then
\[
\log T_{\pf, \rm abs} (\ZZ) = (-1)^m \log T_{\pf^c, \rm rel} (\ZZ).
\]
\end{theo}


\subsection{Analytic torsion and CM Theorem on $X$}

Collecting the results of the previous section, we may prove the following theorems.

\begin{theo}\label{tX} Let $X=\ZZ\sqcup_W Y$ be a space with a horn singularity of dimension $n=m+1$, and $\pf$ a middle perversity, then
\begin{align*}
\log T_{\pf}(X)=&\log  T_{\pf, \rm abs}(\ZZ)+\log T_{\rm rel}(Y)-\frac{1}{2}\chi(W)\log 2+\log \tau(\H),
\end{align*}
where $\H$ is the long homology exact sequence induced by the inclusion $\ZZ\to X$,  with some orthonormal graded bases. 
\end{theo}

\begin{theo}\label{tcmX} Let $X=\ZZ\sqcup_W Y$ be a space with a horn singularity of dimension $n=m+1$, and $\pf$ a middle perversity.  Then, if $m=2p-1$: 
\begin{align*}
\log T_{ \pf}(X)&=\log I^\pf \tau_{\rm RS}(X),
\end{align*}
if $m=2p$:
\begin{align*}
\log T_{ \pf}(X)=&\log I^\pf \tau_{\rm RS}(X)
+A_{\rm comb, \pf}(W)+A_{\rm analy}(W),\\
\end{align*}
where the  anomalies are given in Theorem \ref{tcm}.
\end{theo}

\vspace{10pt}

\appendix

\section{Irregular singular Sturm Liouville operators}
\label{SL}

We introduce some results in the irregular singular Sturm Liouville operators,  for the general theory we address to \cite{Zet}. 
Consider the interval $(0,l]$, for some positive $l$, and the  Sturm Liouville formal operator
\begin{align*}
\FF_{\al,b} f&=-h^{2\al -1}\left(h^{1-2\al }f'\right)' +\frac{b}{h^2} f
=-f''-(1-2\al)\frac{h'}{h}f'+\frac{b}{h^2} f,
\end{align*}
where $\al, b$ are real numbers where $\al$ is an half integer,  $b\geq 0$, and $h(x)=x^\ka H(x)$, with  $H$  a smooth positive  function   on $[0,l]$, with $H(0)=1$, and $\ka$ a real number, with $\ka>1$, $f\in \R^I$. When $b>0$, this is an irregular singular SL problem. It reduces to a regular singular problem when $b=0$ (see Proposition \ref{b0} and Remark \ref{solutions}). We study concrete realisations of $\FF_{\al,b}$ in the Hilbert space $L^2((0,l],h^{1-2\al})$.  The associated SL problem (eigenvalues equation) is
\beq
\label{prob1}
\begin{aligned}
\FF_{\al,b} f=\la f.
\end{aligned}
\eeq

Consider the invertible unitary operator
$
\TT_{\frac{1}{2}-\al}:L^2((0,l],w)\to L^2((0,l])$, $\TT_{\frac{1}{2}-\al}:f\mapsto h^{\frac{1}{2}-\al} f$, 
Setting $u=h^{\frac{1}{2}-\al}f$, we have the operator 
$
\UU_{\al,b}u= \TT_{\frac{1}{2}-\al}^{-1} \FF_{\al,b} \TT_{\frac{1}{2}-\al} u=-u''+r_{\al,b}u$. 
Note that the variable $f$ is in the space $L^2((0,l],h^{1-2\al})$, the variable $u$ in the space $L^2((0,l])$. 
The associated eigenvalues equations is
\begin{align*}
u''&=(r_{\al,b}-\la)u,
\end{align*}
with 
\begin{align*}
r_{\al,b}(x)=&\frac{b}{h^2}-\left(\al-\frac{1}{2}\right)\frac{h''}{h} +\left(\al^2-\frac{1}{4}\right) \frac{{h'}^2}{h^2}\\
=&\frac{b}{x^{2\ka}}+b\frac{1-H^2}{x^{2\ka} H^2}
+\ka\left(\al-\frac{1}{2}\right)\left(\ka\left(\al-\frac{1}{2}\right)+1\right) \frac{1}{x^2}
+2\ka\left(\al-\frac{1}{2}\right)^2 \frac{1}{x}\frac{H'}{H}\\
&-\left(\al-\frac{1}{2}\right)\frac{H''}{H}
+\left(\al^2-\frac{1}{4}\right)\frac{{H'}^2}{H^2}.
\end{align*}

\begin{rem} \label{R} For $x\to 0^+$, 
$
r_{\al, b}(x)=\frac{b}{x^{2\ka}}+g(x),
$ 
where the next term in the expansion depends on the value of $\ka$. 
In particular, if $b=0$ the SL equation reduces to a regular singular one, i.e. the singularity is of type $1/x^2$, see Remark \ref{solutions} below.  
\end{rem}

\begin{rem}\label{F} In this work we  consider a pair of formal operators $(\FF_{1,\al,b},\FF_{2,\al,b})$ acting on the Hilbert space
$L^2((0,l],h^{1-2\al})\times L^2((0,l],h^{1-2(\al-1)})$ and defined as follows. 
\begin{align*}
\FF_{1,\al,b}f&=\FF_{\al,b}=-h^{2\al -1}\left(h^{1-2\al }f'\right)' +\frac{b}{h^2} f
=-f''-(1-2\al)\frac{h'}{h}f'+\frac{b}{h^2} f,\\
\FF_{2,\al,b}f&=\FF_{1,\al-1,b}f-(1-2(\al-1))\left(\frac{h'}{h}\right)'f
=-\left(h^{2(\al-1)-1}\left(h^{1-2(\al-1)}f\right)'\right)'+\frac{b}{h^2} f,
\end{align*}



\end{rem}

\begin{rem}\label{sss} Note that 
\begin{align*}
\TT_{2(\al-1)-1}\FF_{1,1-(\al-1)}\TT_{1-2(\al-1)} f
&=\left(h^{2(\al-1)-1}\left(h^{1-2(\al-1)}f\right)'\right)'+\frac{b}{h^2} f
=\FF_{2,\al,b}f
\end{align*}

Whence
$
\FF_{2,\al_q,\t\la}=\TT_{2\al_{q-1}-1}\FF_{1,1-\al_{q-1},\t\la}\TT_{1-2\al_{q-1}}.
$ 
Since by Hodge duality $\t\la_{m-1+q}=\t\la_q$ and $\al_{m-1-q}=-\al_q$, and by definition $\al_{q-1}=\al_q-1$, it follows that $1-\al_q=-\al_{q-1}=\al_{m-q}$, and $1-\al_{q-1}=\al_{m+1-q}$. Whence
$
\FF_{1,\al_{m+1-q},\t\la}=\FF_{1,1-\al_{q-1},\t\la},
$ 
and 
$
\FF_{2,\al_q,\t\la}=\TT_{2\al_{q-1}-1}\FF_{1,\al_{m+1-q},\t\la}\TT_{1-2\al_{q-1}}.
$ 
Recalling that $\TT_\al$ is a  unitary operator, it follows that the spectral properties of the self adjoint extensions of $\FF_{2,\al_q, \t\la}$ and $\FF_{1,\al_{m+1-q},\t\la}$ coincide.
\end{rem}

\begin{theo}\label{olv1} If $b>0$,  the equation 
\[
u''=(r_{\al,b}-\la) u,
\]
has two linearly independent solutions
\[
\uf_{\al,b,\pm}(x,\la)=
x^\frac{\ka}{2}\e^{\pm \frac{\sqrt{b}}{1-\ka}x^{1-\ka}}(1+\vv_{\al,b,\pm}(x,\la)),
\]
with
$
\lim_{x\to 0^+}\vv_{\al,b,\pm}(x)= 0$, and $ \lim_{x\to 0^+}x^\ka \vv'_{\al,b,\pm}(x)= 0.
$ 
The corresponding solutions under under the isometry $\TT_{\al-\frac{1}{2}}$ are
\[
\ff_{\al,b,\pm}(x,\la)=\TT_{\al-\frac{1}{2}}\uf_{\al,b,\pm}(x,\la)=
x^\frac{\ka}{2}h^{\al-\frac{1}{2}}(x)\e^{\pm \frac{\sqrt{b}}{1-\ka}x^{1-\ka}}(1+\vv_{\al,b,\pm}(x,\la)).
\]

\end{theo}

\begin{proof} Let
$
f(x)=\frac{b}{x^{2\ka}}$,  $g(x)=r(x)-f(x)-\la$, 
and set $A=2\ka$ and $B=3$, then 
\begin{align*}
\lim_{x\to 0^+} x^{2A+2} f(x)&=\lim_{x\to 0^+}  x^{4\ka+2}\frac{1}{x^{2\ka}}=0,&
\lim_{x\to 0^+} x^{A-B+2} g(x)=\lim_{x\to 0^+} x^{2\ka-3+2}\frac{1}{x^{2\ka-1}}&=1.
\end{align*}

All the hypothesis of Theorem 6.2 of \cite{Olv}  are satisfied and thus the thesis follows. \end{proof}

\begin{prop}\label{b0}(Compare with Remark \ref{solutions}) If $b=0$,  and $\al\not=\frac{1}{2}$, the equation 
\[
u''=(r_{\al,0}-\la) u,
\]
has two linearly independent solutions
\begin{align*}
\uf_{\al,0,\pm}(x,\la)&=x^{\frac{1}{2}\pm \left|\ka\left(\al-\frac{1}{2}\right)+\frac{1}{2}\right|}\vv_{\al,0,\pm}(x,\la),&
\uf'_{\al,0,\pm}(x,\la)&=x^{\frac{1}{2}\pm \left|\ka\left(\al-\frac{1}{2}\right)+\frac{1}{2}\right|-1}\Phi_{\al,0,\pm}(x,\la),
\end{align*}
where $\vv$ and $\Phi$ are smooth in $[0,l]$, $\vv_{\al,0,\pm}(0)=1$, $\Phi_{\al,0,\pm}(0)=\pm\left|k\left(\al-\frac{1}{2}\right)+\frac{1}{2}\right|$. The corresponding solutions under under the isometry $\TT_{\al-\frac{1}{2}}$ are
\[
\ff_{\al,0,\pm}(x,\la)=\TT_{\al-\frac{1}{2}}\uf_{\al,0,\pm}(x,\la)
=x^{\frac{1}{2}\pm \left|\ka\left(\al-\frac{1}{2}\right)+\frac{1}{2}\right|}h^{\al-\frac{1}{2}}(x)\vv_{\al,0,\pm}(x).
\]

\end{prop}
\begin{proof} If $b=0$ the equation is a regular singular one.  The indicial equation is
\[
\mu(\mu-1)-\ka\left(\al-\frac{1}{2}\right)\left(\ka\left(\al-\frac{1}{2}\right)+1\right)=0,
\]
with solutions
$
\mu_\pm=\frac{1}{2}\pm \left|\ka\left(\al-\frac{1}{2}\right)+\frac{1}{2}\right|.
$ 
Note that $\mu_+-\mu_-=0$ if and only if $\al=\frac{1}{2}-\frac{1}{2\ka}$, that has not solutions in the assumed hypothesis. 
Then, the thesis follows by classical theory of Sturm Liouville equations, see for example \cite[B.2]{HS6}. 
 \end{proof}

Note that in both theorems the solutions  are univocally defined by the initial condition.

Also note that the solutions in the variable $f$ read
$
\ff_{\al,0,1}(x,\la)=x^{\ka(2\al-1)+1}\vv_{\al,0,1}(x)$, and 
$\ff_{\al,0,2}(x,\la)=\vv_{\al,0,2}(x)$, 
where $\vv_j$  are smooth in $[0,l]$, $\vv_{\al,0,j}(0)=1$, and $\vv_{\al,0,2}'(0)=0$.


\begin{rem}\label{regularcase} In the particular case $b=0$ and $\al=\frac{1}{2}$,  all the operators coincide 
$
\FF_{1,\frac{1}{2},0}=\FF_{2,\frac{1}{2},0}=\UU_{1,\frac{1}{2},0}=\UU_{2,\frac{1}{2},0},
$ 
and the eigenvalue equations are regular ones:
$f''=\la f$, and $ u''=\la u$, 
with  solutions 
$
\ff_{\frac{1}{2},0,+}(x,\la)=\uf_{\frac{1}{2},0,+}(x,\la)=\frac{1}{\sqrt{-\la}}\sh \sqrt{-\la} x$, 
and $\ff_{\frac{1}{2},0,-}(x,\la)=\uf_{\frac{1}{2},0,-}(x,\la)=\ch \sqrt{-\la} x$. 

\end{rem}

\begin{rem}\label{solutions} In the particular case of $b=0$, 
the SL problem reduces to a regular singular SL problem
\[
\FF_{\al, 0}f= -h^{2\al -1}\left(h^{1-2\al }f'\right)' =-f''-(1-2\al)\left(\frac{\ka}{x}+\frac{H'}{H}\right)f'=0.
\]

This problem has two explicit solutions (normalised as the $\ff_\pm$):
\begin{align*}
\ff_{1,I}&=1,&\ff_{1,II}(x)&=-((2\al-1)\ka+1)\int_x^l h^{2\al-1}(t) dt +\ff_{1,II}(l).
\end{align*}
 
The equation
\[
\FF_{2,\al,0}f= -\left(h^{2(\al-1) -1}\left(h^{1-2(\al-1) }f\right)'\right)' =0,
\]
has the solutions 
\begin{align*}
\ff_{2,I}(x)&=- h^{2(\al-1)-1} (x)\int_x^l h^{1-2(\al-1)}(t) dt,&\ff_{2,II}&=h^{2(\al-1)-1}.
\end{align*}
  
These solutions do not correspond to the $\pm$ solutions described in Theorem \ref{b0}. 
   
\end{rem}


We call the solutions $\ff_{\al,b,\pm}$, $\uf_{\al,b,\pm}$, a fundamental system of solutions of the SL problem $\FF_{\al,b}$.

\begin{lem}\label{f+-} If $\FF_{\al,b} f=0$, then $\FF_{-\al,b}(h^{1-2\al}f')=0$. 
\end{lem}
\begin{proof} Assume $\FF_{\al,b} f=0$ and set
$g=h^{1-2\al}f$, 
then
\[
g'=(1-2\al)h^{-2\al}h' f+h^{1-2\al}f''=(1-2\al)h^{-2\al}h' f+h^{1-2\al}\left(\frac{b}{h^2} f-(1-2\al)\frac{h'}{h}f'\right),
\]
that gives
$h^{1+2\al} g'=bf$. But then
\[
(1+2\al)h^{1+2\al}\frac{h'}{h}g'+h^{1+2\al} g''=bf',
\]
that gives
\[
g''+(1+2\al)h^{1+2\al}\frac{h'}{h}g'-\frac{b}{h^2}g=\FF_{-\al,b}g=0.
\]

\end{proof}

\begin{corol}\label{cor1} Consider the SL problem $\FF_{\al,b}$. Both end points are non oscillatory. The end point $x=l$ is  a regular end point. If $b>0$, then the end point $x=0$   is an irregular  singular end point, and in such a case the SL problem  $\FF_{\al,b}$ is in the limit point case. 
If $b=0$, and $\al=\frac{1}{2}$, then the end point $x=0$ is a regular end point. 
If $b=0$, and $\al\not=\frac{1}{2}$, then the end point $x=0$   is a regular  singular end point. In such a case, $x=0$ is limit circle in the following cases:
\begin{enumerate}
\item if $\frac{3}{2}\leq\ka< 3$ and $\al=0$,
\item if $1< \ka< \frac{3}{2}$ and $\al=- \frac{1}{2}, 0$,
\end{enumerate}
while $x=0$ is  limit point otherwise.
\end{corol}
\begin{proof} It is clear that if $b>0$ there is only one square integrable solution near $x=0$, so $x=0$ is limit circle. The case $b=0$, $\al=\frac{1}{2}$ follows by Remark \ref{regularcase}. Assume $\al\not=\frac{1}{2}$, then if $b=0$, among the solutions given in Theorem \ref{b0}, $\uf_+$ is square integrable, while $\uf_-$ is square integrable only if
\[
\int_0^l  |\uf_-|^2 \sim \int_0^l x^{1- 2\left|\ka\left(\al-\frac{1}{2}\right)+\frac{1}{2}\right|} <\infty,
\]
and this happens only if 
$
1- 2\left|\ka\left(\al-\frac{1}{2}\right)+\frac{1}{2}\right|>-1,
$ 
i.e. if
$
\frac{1}{2}-\frac{3}{2\ka}<\al<\frac{1}{2}+\frac{1}{2\ka}.
$ 
Since $\ka>1$, this gives
\[
-1<\frac{1}{2}-\frac{3}{2\ka}<\al<\frac{1}{2}+\frac{1}{2\ka}<1,
\]
and hence, for all $\ka$, we have $-1<\al<1$, i.e. $\al=0,\pm\frac{1}{2}$. But, if 
$
0\leq\frac{1}{2}-\frac{3}{2\ka}<\frac{1}{2},
$ 
i.e. if $\ka\geq 3$, on the left we have $\frac{1}{2}\leq \al$, if
$
-\frac{1}{2}\leq\frac{1}{2}-\frac{3}{2\ka}<0,
$ 
i.e. if $\frac{3}{2}\leq\ka< 3$, we have $0\leq \al$, and if 
$
-1<\frac{1}{2}-\frac{3}{2\ka}<-\frac{1}{2},
$ 
i.e. if $1< \ka< \frac{3}{2}$, we have $-\frac{1}{2}\leq \al$.
\end{proof}

\begin{rem}\label{squareint1} By the proof of the previous corollary, we deduce the square integrability of the solutions of a fundamental system:  $\uf_+$ ($\ff_+$) is always square integrable, while $\uf_-$ ($\ff_-$) is square integrable only in the following cases: 
\begin{enumerate}
\item if $b=0$ and $\al=\frac{1}{2}$,
\item if $b=0$, $\frac{3}{2}\leq\ka< 3$ and $\al=0$,
\item if $b=0$, $1< \ka< \frac{3}{2}$ and $\al=- \frac{1}{2}, 0$,
\end{enumerate}
namely in the regular case and in the regular singular limit circle case.
\end{rem}

\begin{rem}\label{Green} The Green formula for the formal operator $\FF_{\al,b}$ is 
\begin{align*}
(\FF_{\al,b} f,f)&=\int_0^l (\FF_{\al,b}f) f w
=[h^{1-2\al} f f']_0^l+\int_0^l h^{1-2\al}{f'}^2+b\int_0^l h^{-1-2\al} f^2.
\end{align*}
\end{rem}

\subsection{Self adjoint extensions}
\label{s222}

Consider the Hilbert space $L^2((0,l],h^{1-2\al})$ with inner product
\[
(u,v)=\int_0^l h^{1-2\al}(x) u(x)  v(x)   dx.
\]
 
Let $\FF$ a formal SL operator. Define the global domain:
\[
\DS(F_{\rm glob})=\left\{ y\in AC_0((0,l])~|~h^{1-2\al}y'\in AC_0((0,l])\right\},
\]
if $y\in D(F_{\rm glob})$, then $\FF y$ is defined and belong to $L_{\rm loc}((0,l],h^{1-2\al})$. This defines the operator $F_{\rm glob}$. Define the maximal operator $F_{\rm max}$ by the domain 
\[
\DS(F_{\rm max})=\left\{ y\in D(F_{\rm glob})~|~y, \FF y\in L^2((0,l],h^{1-2\al})\right\},
\]
and the pre minimal operator
\[
\DS(F_0)= \DS(F_{\rm max})\cap (\R^I)_0,
\]
where $(\R^I)_0$ denotes the subspace of the function with compact support. The operator $F_{0}$ is closable, and its closure, denoted by $F_{\rm min}$, is the adjoint of $F_{\rm max}$. These operators are densely defined: $\overline{\DS(F_0)}=\overline{\DS(F_{\rm max})}=L^2((0,l],h^{1-2\al})$.

\begin{defi}\label{BC} If $x_0$ is  a regular  end point for the SL formal operator $\FF$, we introduce the following boundary values at $x=x_0$,
\begin{align*}
bv_\FF(x_0)(y)&=y(x_0),&bv_\FF'(x_0)(u)&=(h^{1-2\al}y')(x_0),
\end{align*}
and
\begin{align*}
bv_{\FF;\be}(x_0)(y)&=\cos \be~ bv_\FF(x_0)(y)+\sin \be~ bv_\FF'(x_0)(y)=\cos \be ~y(x_0)+\sin \be ~ (h^{1-2\al} y')(x_0),
\end{align*}
where $0\leq\be< \pi$. If $x_0$ is  a non oscillatory regular singular limit circle   end point for the SL formal operator $\FF$, we introduce the following boundary values at $x=x_0$,
\begin{align*}
bv_{\FF;+}(x_0)(y)&=[y,\ff_+]_\FF(x_0),&bv_{\FF;-}(x_0)(u)&=[y,\ff_-]_\FF(x_0),
\end{align*}
where
$[y,f]_\FF= h^{1-2\al}(yf'-y'f)$,
and
\begin{align*}
\widehat{bv}_{\FF;\de}(x_0)(y)&=\cos \de~ bv_+(x_0)(y)+\sin \de~ bv_-(x_0)(y),
\end{align*}
where $0\leq\de< \pi$. 
\end{defi}

Note that the boundary values $bv_\pm$ reduce to (linear combinations) of $bv$ and $bv'$ if the end points happen to be regular one. For this reason we used a unified notation for the associated boundary conditions. It will be tacitly assumed that if the end point is singular, the boundary values $bv_\pm$ are used.


\begin{theo}\label{t3.23} \cite[10.4.2, 10.4.4]{Zet}  \cite[5.1]{NZ} Let $\FF_{\al,b}$ be  the SL problem in equation (\ref{prob1}).   
 Then, the  self adjoint extensions of the operator $F_0$ defined by separated boundary conditions are the operators $F_{\al,b,\be}$ with the following domain (compare with Corollary \ref{cor1})
\begin{enumerate}
\item if $b\not=0$ (irregular singular case), then 
\[
\DS(F_{\al,b;\be})=\left\{ y\in \DS(F_{\al,b,\rm max})~|~bv_{\FF;\be}(l)(y)=0\right\},
\]

\item if $b=0$ and either ($\ka\geq 3$ and $\al\not=\frac{1}{2}$) or ($\frac{3}{2}\leq\ka< 3$ and $\al\not=0, \frac{1}{2}$) 
or  ($1< \ka< \frac{3}{2}$ and $\al\not =- \frac{1}{2}, 0, \frac{1}{2}$) (regular singular limit point case), then
\[
\DS(F_{\al,0;\de,\be})=\left\{ y\in \DS(F_{\al,0,\rm max})~|~ bv_{\FF;\be}(l)(y)=0\right\},
\]

\item if $b=0$ and either 
 $\frac{3}{2}\leq\ka< 3$ and $\al=0$, 
or if $1< \ka< \frac{3}{2}$ and $\al=- \frac{1}{2}, 0$, (regular singular limit circle case), then
\[
\DS(F_{\al,0;\de,\be})=\left\{ y\in \DS(F_{\al,0,\rm max})~|~\widehat{bv}_{\FF;\de}(0)(y)=bv_{\FF;\be}(l)(y)=0\right\},
\]

\item if $b=0$, $\al=\frac{1}{2}$ (regular case), then
\[
\DS(F_{\frac{1}{2},0;\de,\be})=\left\{ y\in \DS(F_{\frac{1}{2},0,\rm max})~|~bv_{\FF;\de}(0)(y)=bv_{\FF;\be}(l)(y)=0\right\}.
\]
\end{enumerate}

\end{theo}

\subsection{Spectral properties and spectral functions}

\begin{theo}\label{tt.1}  Let $\FF_{\al,b}$ be  the SL problems in equation (\ref{prob1}). Then, the end point $x=0$ is non oscillatory, and any of the self adjoint extension $F=F_{\al,b;\be}, F_{\al,0;\delta,\be}$ of $\FF$ described in Theorem \ref{t3.23} is bounded below with discrete pure point spectrum $\Sp(F)=\{\la_n\}_{n=0}^\infty$, consisting in a set of real eigenvalues  with unique accumulation point at infinity, and each eigenvalue is simple.
\end{theo}
\begin{proof} It is clear that the spectrum is real. First we show that all the self adjoint extension  $F$ of $F_0$ described in  Theorem \ref{t3.23} are bounded below.  The formal operator $\FF$ has empty essential spectrum and is bounded below by \cite[XIII.10.C25]{DS2} since 
\[
\lim_{x\to 0^+} x^2 r_{\al,b}(x)=\lim_{x\to 0^+} \left(b x^{2(1-\ka)}
+k\left(\al-\frac{1}{2}\right)\left(k\left(\al-\frac{1}{2}\right)+1\right)\right)>-\frac{1}{4}.
\]

In fact, if $b>0$, the limit is $+\infty$, if $b=0$,  the limit is 
$
=\frac{1}{2}\left( k\left(\al-\frac{1}{2}\right)+1\right)^2,
$ 
that is positive, since it would vanish only if $\al$ were $\frac{1}{2}-\frac{1}{2\ka}$, that may never happen (see also \cite[6.32]{Wei}).  Then, by \cite[XIII.7.24]{DS2}, any self adjoint extension  of $\FF$ is bounded below, and by \cite[XIII.6.4]{DS2} its essential spectrum  is void (compare with \cite[Proposition 10.4.4 (4)]{Zet}). The residual spectrum is void by \cite[XI.8.2]{Yos}. 
Whence the spectrum is real pure point and bounded below. We use \cite[Theorem 10.12.1 (8)]{Zet} to conclude the proof.  
\end{proof}

\subsection{Spectral functions}

\begin{prop}\label{p3.33} Let $F$ be any of the self-adjoint extensions of the formal operator $\FF_{\nu,\al}$ described in Theorem \ref{t3.23}. Then, the resolvent $(\la I-F)^{-1}$ of $F$ is the integral operator with kernel 
\[
k(x,y;\la)=\frac{1}{h(x)^{1-2\al}W(\ff_+,\ff_-)}h^{\frac{1}{2}-\al}(x)h^{\frac{1}{2}-\al}(y)\left\{\begin{array}{ll}\ff_{l} (x)\ff_{0}(y),&x\geq y,\\ \ff_{0} (x)\ff_{l}(y),&x< y,\end{array}\right.
\]
where $\ff_{l}$ is  either the (unique) solution of the equation $\FF_{\al,b}f=\la f$ satisfying the boundary condition for $F$ at $x=l$ if we have the LCC at $x=l$, or the (unique) square integrable solution of the equation $\FF_{\al,b}f=\la f$ if we have the LPC at $x=l$; and $\ff_0$ is  either the (unique) solution of the equation $\FF_{\al,b}f=\la f$ satisfying the boundary condition for $F$ at $x=0$ if we have the LCC at $x=0$, or the (unique) square integrable solution   of the equation $\FF_{\al,b}f=\la f$ if we have the LPC at $x=0$. 
\end{prop}
\begin{proof} The construction of the kernel may be found in  \cite[XIII.3]{DS2}. Alternatively, using \cite[8.29]{Wei}, with $p=r=h^{1-2\al}$, $q(x)=\frac{b}{h^2(x)}$, then
$\FF_{\al,b} f=-\frac{1}{r}(pf')'+qf$.
\end{proof}

\begin{corol}\label{c3.34} Let $F$ be any of the self-adjoint extensions of the formal operator $\FF_{\nu,\al}$ described in Theorem \ref{t3.23}. Then, $F$ has compact resolvent.
\end{corol}
\begin{proof} The resolvent is an integral operator by the Theorem. So for it to be compact it is sufficient to verify that it is square integrable on $(0,l)\times (0,l)$ (see for example \cite[pg. 1331]{DS2}). Observe that that $h(x)^{1-2\al} W(\ff_+,\ff_-)$ is constant. Then we calculate
\begin{align*}
\int_0^l\int_0^l |k(x,y;\la)|^2 dx dy&=\int_0^l |\uf_{l} (x)|^2\int_0^x |\uf_{0}(y)|^2 dy dx+
\int_0^l |\uf_{0} (x)|^2\int_x^l |\uf_{l}(y)|^2dy.
\end{align*}

These functions are smooth, so problems appear only near $x=0$, and there the solutions $\uf_{l}$ are bounded, 
while the behaviour of $\uf_{0}$ is given in Propositions \ref{olv1} and \ref{b0}. In all cases, near zero 
\[
 \int_0^\ep\int_0^\ep |k(x,y)|^2 dx dy={\rm const}  \int_0^\ep x^2 dx<\infty.
\]
\end{proof}

\begin{corol}\label{c3.36} Let $F$ be any of the self-adjoint extensions of the formal operator $\FF_{\nu,\al}$ described in Theorem \ref{t3.23}. Then, there exists a spectral resolution of $F$, i.e. the spectrum $\Sp (F)$ of $F$ is real and pure point, i.e. coincides with the set of the eigenvalues, and discrete, i.e. the unique point of accumulation is infinite, and all eigenspaces have finite dimension. 
All the eigenvalues are simple. All eigenfunctions are smooth, and the set of the  eigenfunctions determines a complete orthonormal basis of $L^2(0,l)$.
\end{corol}
\begin{proof} See \cite[XIII.4.2, XIII.4.3]{DS2}. See also \cite[Proposition 5.12]{Schm}. See \cite[8.29]{Wei} for simplicity of eigenvalues. See also \cite[10.6.1]{Zet} for the case where one point is LCC. 
\end{proof}

\subsection{Some particular self adjoint extensions}
\label{particularoperators}

In the following we will work with the self adjoint extensions of the SL problems $\FF_{\al,b}$ described in Theorem \ref{t3.23}, except that   $F_{\al,0; \frac{\pi}{2}, \be}$, $\al\not=\frac{1}{2}$. We will  denote by $F$  these self adjoint extensions, excluding some times the  operator $F_{\frac{1}{2},0, \delta,\be}$ that is a regular one, and may be treated explicitly as in the next Remark \ref{regularop}.  Occasionally, we will consider the self adjoint extensions of the formal  operators $\UU$ corresponding to the $\FF$ according to Remark \ref{F}. We will use the letter $U$ for these operators.  We will call the cases $\be=0$ relative  bv and the case $\be=\frac{\pi}{2}$ absolute bv. Also note that $\de=0$ corresponds to $bv_+$ and $\de=\frac{\pi}{2}$ corresponds to $bv_-$. 

\begin{rem}\label{regularop} The formal solution of the eigevalues equation for  the operator $F_{\frac{1}{2},0, \delta,\be}$ are given in Remark \ref{regularcase}. The general solution is
$
y(x,\la)=c_+\frac{\sin \sqrt{\la} x}{\sqrt{\la}}+c_-\cos \sqrt{\la} x.
$ 
The boundary condition for  $F_{\frac{1}{2},0, 0,0}$ is 
$y(0)=0$, and $y(l)=0$, 
that gives $c_-=0$, and the equation $\sin \sqrt{\la}l=0$. Whence the eigenvalues are $\la_n=\frac{\pi^2n^2}{l^2}$, $n\in N_0$, and the eigenfunctions
$
y(x,\la_n)=\sin \frac{\pi n}{l} x.
$ 

The boundary condition for  $F_{\frac{1}{2},0, 0,\frac{\pi}{2}}$ is 
$y(0)=0$, and $y'(l)=0$, 
that gives $c_-=0$, and the equation $\cos \sqrt{\la}l=0$. Whence the eigenvalues are $\la_n=\frac{\pi^2}{4l^2}(2n+1)^2$, 
$n\in N$, and the eigenfunctions
$
y(x,\la_n)=\sin \frac{\pi (2n+1)}{2l} x.
$ 

The boundary condition for  $F_{\frac{1}{2},0, \frac{\pi}{2},0}$ is 
$y'(0)=0$, and $y(l)=0$, 
that gives $c_+=0$, and the equation $\cos \sqrt{\la}l=0$. Whence the eigenvalues are $\la_n=\frac{\pi^2}{4l^2}(2n+1)^2$, $n\in N$, and the eigenfunctions
$y(x,\la_n)=\cos \frac{\pi (2n+1)}{2l} x$. 

The boundary condition for  $F_{\frac{1}{2},0, \frac{\pi}{2}, \frac{\pi}{2}}$ is 
$y'(0)=0$, $y'(l)=0$, 
that gives $c_+=0$, and the equation $\sin \sqrt{\la}l=0$. Whence the eigenvalues are $\la_n=\frac{\pi^2n^2}{l^2}$, $n\in N$, and the eigenfunctions
$
y(x,\la_n)=\cos \frac{\pi n}{l} x$. 
\end{rem}

\begin{theo}\label{pos} The kernel of the operators $F$ are given as follows (where $\al\not=\frac{1}{2}$):
\begin{enumerate}
\item (Irregular singular case.) Let $b>0$. Then the operators $F_{\al,b; \frac{\pi}{2}}$ and $F_{\al,b; 0}$ are positive.

\item (Regular singular limit point case.)  The operators  $F_{\al,0;\be}$ are non negative with the following kernel:
\begin{align*}
\ker F_{\al,0;0}&=\langle 0\rangle,&
\ker F_{\al,0; \frac{\pi}{2}}&=\left\{\begin{array}{ll}\langle 1\rangle,&\al<\frac{1}{2},\\\langle 0\rangle,&\al>\frac{1}{2}.\end{array}\right.
\end{align*}

\item (Regular singular limit circle case.)  The operators  $F_{\al,0;0, \be}$  are non negative with the following kernel:
\begin{align*}
\ker F_{\al,0;0,0}&=\langle 0\rangle,&
\ker F_{\al,0; 0,\frac{\pi}{2}}&=\left\{\begin{array}{ll}\langle 1\rangle,&\al<\frac{1}{2},\\\langle 0\rangle,&\al>\frac{1}{2}.\end{array}\right.
\end{align*}

\item (Regular case.) The operators $F_{\frac{1}{2},0;\de,\be}$, $\delta,\be\in\left\{0,\frac{\pi}{2}\right\}$, are non negative with the following kernel:
\begin{align*}
\ker F_{\frac{1}{2},0; 0,\frac{\pi}{2}}&=\langle 0\rangle,&\ker F_{\frac{1}{2},0; \frac{\pi}{2},\frac{\pi}{2}}&=\langle 1\rangle,&
\ker F_{\frac{1}{2},0; 0,0}&=\langle 0\rangle, &\ker F_{\frac{1}{2},0; \frac{\pi}{2},0}&=\langle 0\rangle.
\end{align*}

\end{enumerate}
\end{theo}
\begin{proof} (1) Let $\be_0$ be either $\frac{\pi}{2}$ or $0$. Let $f_n(x)=f(x,\la_n)$ be an eigenfunction associated to the eigenvalue $\la_n$ of $F_{\al,b; \be_0}$. By the Green formula, Remark \ref{Green}, 
\begin{align*}
\la_n\|f_n\|^2=(F_{\al,b;\be_0} f_n,f_n)&=\left[h^{1-2\al} f_n f_n'\right]_0^l+\int_0^l h^{1-2\al}{f_n'}^2+b\int_0^l h^{-1-2\al} f_n^2.
\end{align*}

Since $f_n$ is an eigenfunction, it is a square integrable solution of the eigenvalue equation satisfying the boundary conditions. Therefore according to Theorem \ref{olv1}, $f_n$ decreases exponentially at $x=0$. This implies that all integrals are finite, and the first term vanishes at $x=0$. But it  also vanishes  at $x=l$ by the boundary condition. Thus, 
\begin{align*}
\la_n\|f_n\|^2=(F_{\al,b;\be_0} f_n,f_n)&=\int_0^l h^{1-2\al}{f_n'}^2,
\end{align*}
and $\la_n=0$ if and only if $f'_n=0$, i.e. if $f_n$ is a multiple of the constant function. But the constant function is not in the domain of the operators because it does not have the necessary behaviour near $x=0$.

(2) Let $f_n(x)=f(x,\la_n)$ be an eigenfunction associated to the eigenvalue $\la_n$ of $F_{\al,0; \be}$, the Green formula gives
\begin{align*}
\la_n\|f_n\|^2=(F_{\al,0; \be} f_n,f_n)&=\left[h^{1-2\al} f_n f_n'\right]_0^l+\int_0^l h^{1-2\al}{f_n'}^2.
\end{align*}

In all cases, $(h^{1-2\al} f_n f_n')(l)=0$, since $f_n$ satisfies the BC at $x=l$, and $\| f_n\|$ is finite, since $f_n$ is square integrable. We prove that also $(h^{1-2\al} f_n f_n')(0)=0$ and that the integral is finite. Since $f_n$ is a solution of the eigenvalue equation, it is a linear combination of the two fundamental solutions. 
We know the behaviour of the $f_n$ near $x=0$, by Theorem \ref{b0}.  For the  component of $f_n$ along $\ff_\pm$ we have:
\[
h^{1-2\al}(x) f_n(x) f_n'(x)\sim \int h^{1-2\al}(x) f_n(x) f_n'(x)\sim x^{\pm|\ka(2\al-1)+1|}.
\]

Whence for the component along $\ff_+$, the integral is finite and $(h^{1-2\al} f_n f_n')(0)=0$. 
Thus, $(h^{1-2\al} f_n f_n')(0)=0$ for $F_{\al, 0; \be}$, since $\ff_-$ is not in the domain of this operator, and 
\begin{align*}
\la_n\|f_n\|^2=(F f_n,f_n)&=\int_0^l h^{1-2\al}{f_n'}^2,
\end{align*}
so $\la_n\geq 0$,  
and, $\la_n=0$ if and only if $f'_n=0$. Since the constant function does not satisfy the boundary condition at $x=l$ when $\be=0$, this gives the result for $F_{\al, 0; 0}$. For $F_{\al, 0; \frac{\pi}{2}}$, we know that the constant function is $\ff_+$ and therefore satisfies the bc at $x=0$ for $\al\leq 0$, by Remark \ref{solutions}. This gives the result for $F_{\al, 0; \frac{\pi}{2}}$.

(3) Consider the operator $F_{\al,0; 0, \be}$. Observe that  the boundary condition at $x=0$ is: $\hat{bv}_{\FF;\de}(f)=0$. If $\de=0$, $\hat{bv}_{\FF;\de=0}(f)=bv_{+}(f)=0$, that is satisfied by $\ff_+$  but is not  satisfied by $\ff_-$.  Whence, the  result follows  by the same argument as in  point (2).

(4) This  follows directly by Remark \ref{regularop}.  
\end{proof}


\begin{lem}\label{eigen}  Let $\{\la_n\}$ be in the spectrum of either of the operators $F$ defined above, and $\la_n^0$ denote the corresponding eigenvalue of the associated   operator $F^0$, induced by $F$ setting $b=0$. Then, 
$|\lambda_{n}-\lambda^0_{n}|\leq M$, 
for all $n$ and  some positive constant $M$.
\end{lem}

\begin{proof} Let $\la_n$ denote an eigenvalue of the operator $F_{\al,b;\be}$ and $\la^0_n$ the corresponding eigenvalue of either the operator $F_{\al,0;0, \be}$ if $x=0$ is limit circle, or $F_{\al,0;\be}$ if $x=0$ is limit point (this depend on $\al$). It is easier to work with the operator $U^0$ that are the self adjoint extension  of the formal operators $\UU$ corresponding to the 
$\FF$ according to Remark \ref{F}. Denote by $u_n$ and $u^0_n$ some  eigenfunction of $\la_n$ and $\la_n^0$. Then,
\begin{align*}
\la_n\langle u^0_n, u_n\rangle=\langle u^0_n, S u_n\rangle 
=\langle S^0 u^0_n,  u_n\rangle+\langle u^0_n, r u_n\rangle,
\end{align*} 
where (see Remark \ref{R})
$
r_{\al, b}(x)=\frac{b}{x^{2\ka}}+g(x).
$ 
Since $u_n$ converges exponentially for small $x$, the integral
\[
\lp r u^0_n,u_n\rp=\int_0^l r(x) u^0_n(x) u_n(x) dx,
\]
is finite, and the result follows. 
\end{proof}

\begin{corol}\label{order} The sequences   $\Sp(F)$ have order $\frac{1}{2}$ and genus  $0$ (see  Appendix \ref{backss}).
\end{corol}

\subsection{Some other spectral functions}
\label{other}

Let $\Sigma$ any sector contained in $\C-\Sp_+(L)$ (where $L$ is any of the above operators), and $\Lambda$ the boundary of $\Sigma$. We assume the variable $\la$ always restricted in $\Sigma$. We denote by $-\la$ the complex variable in $\Sigma$ with $\arg (-\la)=0$ on the negative part of the real axes contained in $\Sigma$. With this convention, if $z=\sqrt{\la}$, then $iz =-\sqrt{-\la}$.

The eigenvalues of $F$  may be characterised as follows. If $Ff=\la f$, then $f$ is a solution of equation $\FF_{\al,b}f=\la f$  that belongs to $\DS(F)$. If there are no bc at $x=0$, there exists just one solution square integrable, namely $\ff_+$. If there is a bc at $x=0$, again only one of the square integrable solutions satisfies this bc. Whence, the eigenvalues of $F=F_{\al,b;\frac{\pi}{2}},F_{\al,b;0,\frac{\pi}{2}}, F_{\al,b;\frac{\pi}{2},\frac{\pi}{2}} $ with absolute bc at $x=l$, are the zeros of the function 
\begin{align*}
B_{\al, b;\frac{\pi}{2}}(l,\la )&=\ff_{\al,b,+}'(l,\la), &B_{\al, b;0,\frac{\pi}{2}}(l,\la )&=\ff_{\al,b,+}'(l,\la),
&B_{\al, b;\frac{\pi}{2},\frac{\pi}{2}}(l,\la )&=\ff_{\al,b,-}'(l,\la)
\end{align*}
as a function of $\la$, while the eigenvalues of $F=F_{\al,b;0},F_{\al,b;0,0}, F_{\al,b;\frac{\pi}{2},0}  $ with relative bc at $x=l$, are the zeros of the function 
\begin{align*}
B_{\al, b;0}(l,\la )&=\ff_{\al,b,+}(l,\la), &B_{\al, b;0,0}(l,\la )&=\ff_{\al,b,+}(l,\la),
&B_{\al, b;\frac{\pi}{2},0}(l,\la )&=\ff_{\al,b,-}(l,\la)
\end{align*}
as a function of $\la$. Since the solutions are analytic in $\la$, $B$ is an entire  function. 
By Lemma \ref{order},  $B$ has order  $\frac{1}{2}$, thus, we have the factorisation 
\[
B(l,\la )=B(l,0 ) \la^{\dim\ker F}\prod_{\la_{n}\in \Sp_0(F)} \left(1+\frac{-\la}{\la_{n}}\right),
\] 
where 
$
B(l,0)=\lim_{\la\to 0} \frac{B(l,\la)}{\la^{\dim\ker F}}. 
$ 
It follows that:
\beq\label{ww}
\begin{aligned}
\log\Gamma(-\lambda, F)=&\log\Gamma(-\lambda,\Sp_0 (F))
=-\log\prod_{\la_{n}\in \Sp_0(F)} \left( 1+\frac{-\lambda}{\lambda_{n}}\right)\\
=&\log B(l,0 )+ \dim\ker F\ln\la-\log B(l,\la ).
\end{aligned}
\eeq

\subsection{Asymptotic expansion of the solutions}
\label{asym}

One of the main technical tools we need in the determination of analytic torsion is information on the existence and of the type of the asymptotic expansion of some spectral functions, from one side for large values of $\la$ and on the other for large values of $n$. This information depends on the asymptotic expansions of the solutions of  the fundamental system  The proofs are given by classical methods of asymptotic analysis,  and a new approach consisting in introducing a perturbation of the flat case, main reference is Olver \cite{Olv} (see also \cite{Mur, Mar}).

\begin{lem}\label{l2.3} The  equation
\[
v''(x)+\left(z^2-\frac{\nu^2-\frac{1}{4}}{h^2(x)}- p(x)\right) v(x)=0,
\]
has two linearly independent solutions $v_\pm(x,z )$,   that for large $z$ (in the suitable sector) have the following asymptotic expansions
\begin{align*}
v_{\mp}(x,z,\nu )
=& \e^{\pm i x z}\left(1+\left(\mp\frac{i}{2}\left(\frac{1}{4}-\nu^2\right)\int \frac{1}{h^2(x)}\pm \frac{i}{2} \int p(x) dx\right)\frac{1}{z}+\dots\right).
\end{align*}
uniform  in $x$ for $x$ in any compact subset of $(0,l]$. The coefficients are functions of $\nu^2$.
\end{lem}
\begin{proof} We assume $v(x,z,\nu )=\e^{z w_0(x,\nu )+w_1(x,\nu )+z^{-1} w_2(x,\nu )+\dots}=\e^{F(x,z,\nu )}$, and substitute in the differential equation. The,  equating the coefficients of the powers we have the result \cite[Chapter 6]{Olv}.
\end{proof}

\begin{lem}\label{l3.5} The  equation
\[
w''(x)+\left(z^2\nu^2-\frac{\nu^2-\frac{1}{4}}{h(x)^2}- p(x)\right) w(x)=0,
\]
has two linearly independent solutions $w_{\pm}(x,z,\nu)$,   that for large $\nu$  have the following asymptotic expansions
\begin{align*}
w_{\pm}(x,z,\nu)
=\frac{\e^{\pm\nu\int \sqrt{\frac{1}{h^2(x)}-z^2} dx}}{\left(\frac{1}{h^2(x)}-z^2\right)^\frac{1}{4}}\left(\sum_{j=0}^m U_j(x,z) (\pm\nu)^{-j}+ O\left(\frac{1}{\nu^m}\right)\right).
\end{align*}
uniformely in $x$ for $x$ in any compact subset of $(0,l]$,   analytic and uniform in $z$, for $z\in \Sigma_{\theta,c}=\left\{z\in \C~|~|\arg(z-c)|> \frac{\theta}{2}\right\}$, $c,\theta>0$.
\end{lem}
\begin{proof} We proceed as in Chapter 10 of \cite{Olv}, with $f(x,z)=\frac{1}{h(x)^2}-z^2$, $g(x)= p(x)-\frac{1}{4 h(x)^2}$,
and with the change of variable 
$t=t(x)=\int \sqrt{\frac{1}{h(x)^2}-z^2} dx$. 
Then, we have the equation $W''+(\nu^2 +\psi)W=0$ and the thesis follows considering the series solution 
$W(t,z,\pm\nu)=\e^{\pm\nu t}\sum_{j=0}^\infty U_j(t,z) (\pm\nu)^{-j}$. 
\end{proof}

Using these lemmas we can compute the asymptotic expansion of the fundamental solution of the relevant SL problems. For example we have the following expansions for large $\nu$: 
\begin{align*}
\uf_{+}(x,\la\nu^2,\nu)=&\frac{2^{\nu} \Gamma(\nu+1)}{\sqrt{2\pi\nu}(-\la)^\frac{\nu}{2}}
\frac{\sqrt{h(x)}}{\left(1-\la h^2(x)\right)^\frac{1}{4}}\e^{\nu\mathlarger{\mathlarger{\int}} \frac{\sqrt{1-\la h^2(x)}}{ h(x)} dx}\left(\sum_{j=0}^J \frac{U_j(x,i\sqrt{-\la})}{ \nu^{j}}+ O\left(\frac{1}{\nu^{J+1}}\right)\right),\\
\end{align*}
uniformly   in $x$ for $x$ in any compact subset of $(0,l]$, analytic and uniform in $\la$, for $\la\in \Sigma_{\theta,c}=\left\{z\in \C~|~|\arg(z-c)|> \frac{\theta}{2}\right\}$, $c,\theta>0$, where  $U_0=1$, and
$
U_j(x,i\sqrt{-\la})=O\left(\frac{1}{(-\la)^\frac{k}{2}}\right),
$
for all $j>1$, with some $k>1$.

\vspace{10pt}

\section{Elements of zeta regularisation technique} 
\label{backss}

\subsection{Simple sequences and zeta determinant}

Let $S=\{a_n\}_{n=1}^\infty$ be a sequence
of  non negative real numbers with unique accumulation point at infinity. 
We denote by $S_0$ the positive part of $S$, i.e $S_0=S-\{0\}$. We assume that $S$ has finite exponent of convergence $\es(S)$, so that the associated Weierstrass canonical product converges uniformly and absolutely in any bounded closed region and is an integral function of order the genus $\gs(S)$ of the sequence $S$. In this setting, we define the function
\[
\frac{1}{\Gamma(-\lambda,S)}=\prod_{a_n\in S_0}\left(1+\frac{-\lambda}{a_n}\right)\e^{\sum_{j=1}^{\gs(S)}\frac{(-1)^j}{j}\frac{(-\lambda)^j}{a_n^j}},
\]
for $\la\in \rho(S)=\C-S$, that we call Gamma function associated to $S$.  Here $-\la$ denotes the complex variable defined on $\C- [0,+\infty)$, with $\arg{-\lambda=0}$ on $(-\infty, 0]$. 
We also introduce the zeta function, defined by the uniformly convergent series 
\[
\zeta(s,S)=\sum_{a_n\in S_0} a_n^{-s},
\]
for $\Re(s)>\es(S)$, and by analytic continuation elsewhere. We  assume that the sequence $S$ is a regular sequence of spectral type of non positive order, as defined in \cite[2.1, 2.6]{Spr9}. In this situation, the analytic extension of the zeta function associated to $S$ is regular at $s=0$, and we may define the zeta regularised determinant of $S$ by
$\det_\zeta S=\e^{-\zeta'(0,S)}$. 
Note that if $S_0\not= S$, the determinant is defined as the determinant of the positive part of $S$, since the determinant of $S$ in this case would obviously vanish. Moreover, we have the following result \cite[2.11]{Spr9}.

\begin{theo} \label{teo1-1}
If $S$ is a regular sequence of spectral type of non positive order, the logarithmic Gamma function has an asymptotic expansion for large $\la$ in $\C-\Sigma_{c,\theta}$, where $\Sigma_{c,\theta}$ is some sector $\left\{z\in \C~|~|\arg(z-c)|\leq \frac{\theta}{2}\right\}$, with $c>0$ and $0<\theta<\pi$, that contains $S$, and  we have 
\[
\log \det_\zeta S=-\zeta'(0,S)=\Rz_{\la=+\infty} \log \Gamma(-\la,S).
\]
\end{theo}

\subsection{Double sequences and the spectral decomposition lemma}

We give some formulas to deal with zeta invariants of double series that follow as particular instances of the general results given in \cite{Spr9} (see also \cite{HS3}).

Given a double sequence $S=\{\lambda_{n,k}\}_{n,k=1}^\infty$  of non
vanishing complex numbers with unique accumulation point at the
infinity, finite exponent $s_0=\es(S)$ and genus $p=\gs(S)$, we use the notation $S_n$ ($S_k$) to denote the simple sequence with fixed $n$ ($k$), we call the exponents of $S_n$ and $S_k$ the {\it relative exponents} of $S$, and we use the notation $(s_0=\es(S),s_1=\es(S_k),s_2=\es(S_n))$; we define {\it relative genus} accordingly.

\begin{defi}\label{spdec} Let $S=\{\lambda_{n,k}\}_{n,k=1}^\infty$ be a double
sequence with finite exponents $(s_0,s_1,s_2)$, genus
$(p_0,p_1,p_2)$, and positive spectral sector
$\Sigma_{\theta_0,c_0}$. Let $U=\{u_n\}_{n=1}^\infty$ be a totally
regular sequence of spectral type of infinite order with exponent
$r_0$, genus $q$, domain $D_{\phi,d}$. We say that $S$ is
spectrally decomposable over $U$ with power $\kappa$, length $\ell$ and
asymptotic domain $D_{\theta,c}$, with $c={\rm min}(c_0,d,c')$,
$\theta={\rm max}(\theta_0,\phi,\theta')$, if there exist positive
real numbers $\kappa$, $\ell$ (integer), $c'$, and $\theta'$, with
$0< \theta'<\pi$,   such that:
\begin{enumerate}
\item the sequence
$u_n^{-\kappa}S_n=\left\{\frac{\lambda_{n,k}}{u^\kappa_n}\right\}_{k=1}^\infty$ has
spectral sector $\Sigma_{\theta',c'}$, and is a totally regular
sequence of spectral type of infinite order for each $n$;
\item the logarithmic $\Gamma$-function associated to  $S_n/u_n^\kappa$ has an asymptotic expansion  for large
$n$ uniformly in $\lambda$ for $\lambda$ in
$D_{\theta,c}$, of the following form
\beq\label{exp}
\log\Gamma(-\lambda,u_n^{-\kappa} S_n)=\sum_{h=0}^{\ell}
\phi_{\sigma_h}(\lambda) u_n^{-\sigma_h}+\sum_{l=0}^{L}
P_{\rho_l}(\lambda) u_n^{-\rho_l}\log u_n+o(u_n^{-r_0}),
\eeq
where $\sigma_h$ and $\rho_l$ are real numbers with $\sigma_0<\dots <\sigma_\ell$, $\rho_0<\dots <\rho_L$, the
$P_{\rho_l}(\lambda)$ are polynomials in $\lambda$ satisfying the condition $P_{\rho_l}(0)=0$, $\ell$ and $L$ are the larger integers 
such that $\sigma_\ell\leq r_0$ and $\rho_L\leq r_0$.
\end{enumerate}
\end{defi}

Define the following functions, ($\Lambda_{\theta,c}=\left\{z\in \C~|~|\arg(z-c)|= \frac{\theta}{2}\right\}$, oriented counter clockwise):
\beq\label{fi1}
\Phi_{\sigma_h}(s)=\int_0^\infty t^{s-1}\frac{1}{2\pi i}\int_{\Lambda_{\theta,c}}\frac{\e^{-\lambda t}}{-\lambda} \phi_{\sigma_h}(\lambda) d\lambda dt.
\eeq

By Lemma 3.3 of \cite{Spr9}, for all $n$, we have the expansions:
\beq\label{form}\begin{aligned}
\log\Gamma(-\lambda,S_n/{u_n^\kappa})&\sim\sum_{j=0}^\infty a_{\alpha_j,0,n}
(-\lambda)^{\alpha_j}+\sum_{k=0}^{p_2} a_{k,1,n}(-\lambda)^k\log(-\lambda),\\
\end{aligned}
\eeq
for large $\lambda$ in $D_{\theta,c}$. We set (see Lemma 3.5 of \cite{Spr9})
\beq\label{fi2}
\begin{aligned}
A_{0,0}(s)&=\sum_{n=1}^\infty \left(a_{0, 0,n} 
-\sum_{h=0}^\ell b_{\sigma_h,0,0}u_n^{-\sigma_h}\right)u_n^{-\kappa s},\\
A_{0,1}(s)&=\sum_{n=1}^\infty \left(a_{0, 1,n} 
-\sum_{h=0}^\ell b_{\sigma_h,j,1}u_n^{-\sigma_h}\right)
u_n^{-\kappa s},& 0&\leq j\leq p_2.
\end{aligned}
\eeq

\begin{theo} \label{sdl} Let $S$ be spectrally decomposable over $U$ as in Definition \ref{spdec}. Assume that the functions $\Phi_{\sigma_h}(s)$ have at most simple poles for $s=0$. Then,
$\zeta(s,S)$ is regular at $s=0$, and
\begin{align*}
\zeta(0,S)=&\zeta_{\rm reg}(0,S)+\zeta_{\rm sing}(0,S),&
\zeta'(0,S)=&\zeta'_{\rm reg}(0,S)+\zeta'_{\rm sing}(0,S),
\end{align*}
where the regular and singular part are 
\begin{align*}
\zeta_{\rm reg}(0,S)=&-A_{0,1}(0),\hspace{80pt}
\zeta_{\rm sing}(0,S)=\frac{1}{\kappa}{\sum_{h=0}^\ell} \Ru_{s=0}\Phi_{\sigma_h}(s)\Ru_{s=\sigma_h}\zeta(s,U),\\
\zeta_{\rm reg}'(0,S)=&-A_{0,0}(0)-A_{0,1}'(0),\\
\zeta'_{\rm sing}(0,S)=&\frac{\gamma}{\kappa}\sum_{h=0}^\ell\Ru_{s=0}\Phi_{\sigma_h}(s)\Ru_{s=\sigma_h}\zeta(s,U)\\
&+\frac{1}{\kappa}\sum_{h=0}^\ell\Rz_{s=0}\Phi_{\sigma_h}(s)\Ru_{s=\sigma_h}\zeta(s,U)+{\sum_{h=0}^\ell}{^{\displaystyle
'}}\Ru_{s=0}\Phi_{\sigma_h}(s)\Rz_{s=\sigma_h}\zeta(s,U),
\end{align*}
and the notation $\sum'$ means that only the terms such that $\zeta(s,U)$ has a pole at $s=\sigma_h$ appear in the sum.
\end{theo}



\begin{rem}\label{last} Observe that  in the expansion at point (2) of Definition \ref{spdec}, the terms where the functions $\phi_{\sigma_h}(\la)$ and the polynomial $P_{\rho_l}(\la)$ are constant (in $\la$) do not enter in the subsequent  results. For this reason it is sufficient in the enumeration of this terms (namely in the indices $\sigma_h$ and $\rho_l$) to consider only the other terms.
\end{rem}

\begin{corol} \label{c} Let $S_{(j)}=\{\lambda_{(j),n,k}\}_{n,k=1}^\infty$, $j=1,...,J$, be a finite set of  double sequences that satisfy all the requirements of Definition \ref{spdec} of spectral decomposability over a common sequence $U$, with the same parameters $\kappa$, $\ell$, etc., except that the polynomials $P_{(j),\rho}(\lambda)$ appearing in condition (2) do not vanish for $\lambda=0$. Assume that some linear combination $\sum_{j=1}^J c_j P_{(j),\rho}(\lambda)$, with complex coefficients, of such polynomials does satisfy this condition, namely that $\sum_{j=1}^J c_j P_{(j),\rho}(\lambda)=0$. Then, the linear combination of the zeta function $\sum_{j=1}^J c_j \zeta(s,S_{(j)})$ is regular at $s=0$ and satisfies the linear combination of the formulas given in Theorem \ref{sdl}.
\end{corol}


\vspace{10pt}

{\bf Declaration}

Data availability statement: data sharing not applicable to this article as no datasets were generated or analysed during the current study.

There is not competing interest and funding.

\bibliographystyle{plain}
\bibliography{HarSprBibliography.bib}

\begin{thebibliography}{10}

\bibitem{BL0}
J.~Br\"{u}ning and M.~Lesch.
\newblock Hilbert complexes.
\newblock {\em J. Funct. Anal.}, 108(1):88--132, 1992.

\bibitem{BM1}
J.~Br\"{u}ning and X.~Ma.
\newblock An anomaly formula for {R}ay-{S}inger metrics on manifolds with
  boundary.
\newblock {\em Geom. Funct. Anal.}, 16(4):767--837, 2006.

\bibitem{BM2}
J.~Br\"{u}ning and X.~Ma.
\newblock On the gluing formula for the analytic torsion.
\newblock {\em Math. Z.}, 273(3-4):1085--1117, 2013.

\bibitem{BS1}
J.~Br\"{u}ning and R.~Seeley.
\newblock Regular singular asymptotics.
\newblock {\em Adv. in Math.}, 58(2):133--148, 1985.

\bibitem{BS2}
J.~Br\"{u}ning and R.~Seeley.
\newblock The resolvent expansion for second order regular singular operators.
\newblock {\em J. Funct. Anal.}, 73(2):369--429, 1987.

\bibitem{Che0}
J.~Cheeger.
\newblock Analytic torsion and the heat equation.
\newblock {\em Ann. of Math. (2)}, 109(2):259--322, 1979.

\bibitem{Che1}
J.~Cheeger.
\newblock On the spectral geometry of spaces with cone-like singularities.
\newblock {\em Proc. Nat. Acad. Sci. U.S.A.}, 76(5):2103--2106, 1979.

\bibitem{Che3}
J.~Cheeger.
\newblock On the {H}odge theory of {R}iemannian pseudomanifolds.
\newblock In {\em Geometry of the {L}aplace operator ({P}roc. {S}ympos. {P}ure
  {M}ath., {U}niv. {H}awaii, {H}onolulu, {H}awaii, 1979)}, Proc. Sympos. Pure
  Math., XXXVI, pages 91--146. Amer. Math. Soc., Providence, R.I., 1980.

\bibitem{Che2}
J.~Cheeger.
\newblock Spectral geometry of singular {R}iemannian spaces.
\newblock {\em J. Differential Geom.}, 18(4):575--657 (1984), 1983.

\bibitem{CY}
J.~Cheeger and S.~T. Yau.
\newblock A lower bound for the heat kernel.
\newblock {\em Comm. Pure Appl. Math.}, 34(4):465--480, 1981.

\bibitem{DS2}
N.~Dunford and J.~T. Schwartz.
\newblock {\em Linear operators. {P}art {II}}.
\newblock Wiley Classics Library. John Wiley \& Sons, Inc., New York, 1988.
\newblock Spectral theory. Selfadjoint operators in Hilbert space, With the
  assistance of William G. Bade and Robert G. Bartle, Reprint of the 1963
  original, A Wiley-Interscience Publication.

\bibitem{Gaf2}
M.~P. Gaffney.
\newblock The heat equation method of {M}ilgram and {R}osenbloom for open
  {R}iemannian manifolds.
\newblock {\em Ann. of Math. (2)}, 60:458--466, 1954.

\bibitem{Gaf1}
M.~P. Gaffney.
\newblock A special {S}tokes's theorem for complete {R}iemannian manifolds.
\newblock {\em Ann. of Math. (2)}, 60:140--145, 1954.

\bibitem{Gil}
P.~B. Gilkey.
\newblock {\em Invariance theory, the heat equation, and the {A}tiyah-{S}inger
  index theorem}.
\newblock Studies in Advanced Mathematics. CRC Press, Boca Raton, FL, second
  edition, 1995.

\bibitem{GM1}
M.~Goresky and R.~MacPherson.
\newblock Intersection homology theory.
\newblock {\em Topology}, 19(2):135--162, 1980.

\bibitem{GM2}
M.~Goresky and R.~MacPherson.
\newblock Intersection homology. {II}.
\newblock {\em Invent. Math.}, 72(1):77--129, 1983.

\bibitem{HS1}
L.~Hartmann and M.~Spreafico.
\newblock The analytic torsion of a cone over a sphere.
\newblock {\em J. Math. Pures Appl. (9)}, 93(4):408--435, 2010.

\bibitem{HS3}
L.~Hartmann and M.~Spreafico.
\newblock {$R$} torsion and analytic torsion of a conical frustum.
\newblock {\em J. G\"{o}kova Geom. Topol. GGT}, 6:28--57, 2012.

\bibitem{HS4}
L.~Hartmann and M.~Spreafico.
\newblock On the {C}heeger-{M}\"{u}ller theorem for an even-dimensional cone.
\newblock {\em Algebra i Analiz}, 27(1):194--217, 2015.

\bibitem{HS5}
L.~Hartmann and M.~Spreafico.
\newblock The analytic torsion of the finite metric cone over a compact
  manifold.
\newblock {\em J. Math. Soc. Japan}, 69(1):311--371, 2017.

\bibitem{HS6}
L.~Hartmann and M.~Spreafico.
\newblock Intersection torsion and analytic torsion of spaces with conical
  singularities.
\newblock arXiv:2001.07801, 2020.

\bibitem{KW}
F.~Kirwan and J.~Woolf.
\newblock {\em An introduction to intersection homology theory}.
\newblock Chapman \& Hall/CRC, Boca Raton, FL, second edition, 2006.

\bibitem{Lap}
F.~Lapp.
\newblock An index theorem for operators with hornsingularities.
\newblock 2013.
\newblock
  ahttps://edoc.hu-berlin.de/bitstream/handle/18452/17490/lapp.pdf?sequence=1.

\bibitem{Les2}
M.~Lesch.
\newblock A gluing formula for the analytic torsion on singular spaces.
\newblock {\em Anal. PDE}, 6(1):221--256, 2013.

\bibitem{LP}
M.~Lesch and N.~Peyerimhoff.
\newblock On index formulas for manifolds with metric horns.
\newblock {\em Comm. Partial Diff. Eq.}, 23(3), 1996.

\bibitem{Luc}
W.~L{\"{u}}ck.
\newblock Analytic and topological torsion for manifolds with boundary and
  symmetry.
\newblock {\em J. Differential Geom.}, 37(2):263--322, 1993.

\bibitem{Lud}
U.~Ludwig.
\newblock An extension of a theorem by {C}heeger and {M}\"{u}ller to spaces
  with isolated conical singularities.
\newblock {\em C. R. Math. Acad. Sci. Paris}, 356(3):327--332, 2018.

\bibitem{Mar}
V.~A. Marchenko.
\newblock {\em Sturm-{L}iouville operators and applications}.
\newblock AMS Chelsea Publishing, Providence, RI, revised edition, 2011.

\bibitem{Mas}
W.~S. Massey.
\newblock {\em A basic course in algebraic topology}, volume 127 of {\em
  Graduate Texts in Mathematics}.
\newblock Springer-Verlag, New York, 1991.

\bibitem{Mil0}
J.~Milnor.
\newblock Whitehead torsion.
\newblock {\em Bull. Amer. Math. Soc.}, 72:358--426, 1966.

\bibitem{Mul1}
W.~M\"{u}ller.
\newblock Analytic torsion and {$R$}-torsion of {R}iemannian manifolds.
\newblock {\em Adv. in Math.}, 28(3):233--305, 1978.

\bibitem{MV}
W.~M\"{u}ller and B.~Vertman.
\newblock The metric anomaly of analytic torsion on manifolds with conical
  singularities.
\newblock {\em Comm. Partial Differential Equations}, 39(1):146--191, 2014.

\bibitem{Mun}
J.~R. Munkres.
\newblock {\em Elements of algebraic topology}.
\newblock Addison-Wesley Publishing Company, Menlo Park, CA, 1984.

\bibitem{Mur}
J.~D. Murray.
\newblock {\em Asymptotic analysis}, volume~48 of {\em Applied Mathematical
  Sciences}.
\newblock Springer-Verlag, New York, second edition, 1984.

\bibitem{NZ}
H.-D. Niessen and A.~Zettl.
\newblock Singular {S}turm-{L}iouville problems: the {F}riedrichs extension and
  comparison of eigenvalues.
\newblock {\em Proc. London Math. Soc. (3)}, 64(3):545--578, 1992.

\bibitem{Olv}
F.~W.~J. Olver.
\newblock {\em Asymptotics and special functions}.
\newblock AKP Classics. A K Peters, Ltd., Wellesley, MA, 1997.
\newblock Reprint of the 1974 original [Academic Press, New York; MR0435697 (55
  \#8655)].

\bibitem{RS}
D.~B. Ray and I.~M. Singer.
\newblock {$R$}-torsion and the {L}aplacian on {R}iemannian manifolds.
\newblock {\em Advances in Math.}, 7:145--210, 1971.

\bibitem{Schm}
K.~Schm\"{u}dgen.
\newblock {\em Unbounded self adjoint operators on Hilbert space}, volume 265
  of {\em Graduate Texts in Mathematics}.
\newblock Springer-Verlag, New York-Berlin, 2001.

\bibitem{Sch}
G.~Schwarz.
\newblock {\em Hodge decomposition---a method for solving boundary value
  problems}, volume 1607 of {\em Lecture Notes in Mathematics}.
\newblock Springer-Verlag, Berlin, 1995.

\bibitem{Spa}
E.~H. Spanier.
\newblock {\em Algebraic topology}.
\newblock Springer-Verlag, New York, [1995?].
\newblock Corrected reprint of the 1966 original.

\bibitem{Spr9}
M.~Spreafico.
\newblock Zeta determinant for double sequences of spectral type.
\newblock {\em Proc. Amer. Math. Soc.}, 140(6):1881--1896, 2012.

\bibitem{Spr20}
M.~Spreafico.
\newblock Euler isomorphism, {E}uler basis, and {R}eidemeister torsion.
\newblock {\em Mosc. Math. J.}, 18(3):517--555, 2018.

\bibitem{Spr12}
M.~Spreafico.
\newblock Intersection {D}e {R}ham theory on spaces with singularities of
  conical type.
\newblock 2021.

\bibitem{Vis}
S.~M. Vishik.
\newblock Generalized {R}ay-{S}inger conjecture. {I}. {A} manifold with a
  smooth boundary.
\newblock {\em Comm. Math. Phys.}, 167(1):1--102, 1995.

\bibitem{Wei}
J.~Weidmann.
\newblock {\em Linear operators in {H}ilbert spaces}, volume~68 of {\em
  Graduate Texts in Mathematics}.
\newblock Springer-Verlag, New York-Berlin, 1980.
\newblock Translated from the German by Joseph Sz\"{u}cs.

\bibitem{Yos}
K.~Yosida.
\newblock {\em Functional analysis}.
\newblock Classics in Mathematics. Springer-Verlag, Berlin, 1995.
\newblock Reprint of the sixth (1980) edition.

\bibitem{Zet}
A.~Zettl.
\newblock {\em Sturm-{L}iouville theory}, volume 121 of {\em Mathematical
  Surveys and Monographs}.
\newblock American Mathematical Society, Providence, RI, 2005.

\end{thebibliography}

\end{document}